\documentclass{SCAE}%SCAEOL for online version; SCAE for publication version; SCAES for the paper dedicated to somebody.
\numberwithin{equation}{section}
\allowdisplaybreaks

\def\dsum{\displaystyle\sum}
\def\dint{\displaystyle\int}

\def\dsup{\displaystyle\sup}

\def\dlim{\displaystyle\lim}
\def\r{\right}
\def\lf{\left}
\def\pat{\partial}
\def\ls{\lesssim}

\def\fz{\infty}
\def\fz{\infty}
\def\az{\alpha}
\def\supp{{\mathop\mathrm{\,supp\,}}}

\def\loc{{\mathop\mathrm{\,loc\,}}}
\def\bmo{{\mathop\mathrm{BMO}}}
\def\dist{{\mathop\mathrm{dist}}}
\def\mol{{\mathop\mathrm{mol}}}
\def\riesz{{\mathop\mathrm{Riesz}}}

\def\lz{\lambda}

\def\bdz{\Delta}
\def\ez{\epsilon}

\def\bz{\beta}

\def\gz{{\gamma}}
\def\bgz{{\Gamma}}
\def\vz{\varphi}
\def\tz{\theta}
\def\sz{\sigma}
\def\wz{\widetilde}

\def\ls{\lesssim}

\def\boz{\Omega}

\def\pat{\partial}
\def\noz{\nonumber}

\def\rr{{\mathbb R}}
\def\rn{{{\rr}^n}}
\def\zz{{\mathbb Z}}
\def\nn{{\mathbb N}}

\def\zz{{\mathbb Z}}
\def\nn{{\mathbb N}}
\def\cc{{\mathbb C}}

\def\cl{{\mathcal L}}

\numberwithin{equation}{section}

\def\hs{\hspace{0.3cm}}

\begin{document}
\arraycolsep=1pt

%Basic Information
\Year{2012} %
\Month{January}
\Vol{55} %
\No{1} %
\BeginPage{1} %
\EndPage{XX} %
\AuthorMark{Cao J {\it et al.}} \ReceivedDay{October 17, 2011}
\AcceptedDay{February 20, 2012} \PublishedOnlineDay{; published
online March 22, 2012}
\DOI{10.1007/s11425-000-0000-0} % The author doesn't need fill in it.

% \title[short text for running head]{full title}{comments for title}
\title{Hardy spaces $H_L^p({\mathbb R}^n)$ associated with
operators satisfying $k$-Davies-Gaffney estimates}{}

% \author[]{Full name}{footnote}
% Remark:  One \author for one author

\author{CAO Jun}{}
\author{YANG DaChun}{Corresponding author}{}
%\author[4]{FIRST3 Last Name3}{}

\address{School of Mathematical Sciences, Beijing Normal University,
Laboratory of Mathematics}
\address{and Complex Systems, Ministry of
Education, Beijing {\rm 100875}, People's Republic of China}

\Emails{ caojun1860@mail.bnu.edu.cn, dcyang@bnu.edu.cn}

\maketitle

%Abstract is required.

 {\begin{center}
\parbox{14.5cm}{\begin{abstract}
Let $L$ be a one-to-one operator of type $\omega$ having a bounded
$H_\infty$ functional calculus and satisfying the $k$-Davies-Gaffney
estimates with $k\in{\mathbb N}$. In this paper, the authors
introduce the Hardy space $H_L^p(\mathbb{R}^n)$ with $p\in (0,\,1]$
associated with $L$ in terms of square functions defined via
$\{e^{-t^{2k}L}\}_{t>0}$ and establish their molecular and
generalized square function characterizations. Typical examples of
such operators include the $2k$-order divergence form homogeneous
elliptic operator $L_1$ with complex bounded measurable coefficients
and the $2k$-order Schr\"odinger type operator $L_2:=
(-\Delta)^k+V^k$, where $\Delta$ is the Laplacian and $0\le V\in
L^k_{\mathop\mathrm{loc}}(\mathbb{R}^n)$. Moreover, as an
application, for $i\in\{1,\,2\}$, the authors prove that the
associated Riesz transform $\nabla^k(L_i^{-1/2})$ is bounded from
$H_{L_i}^p(\mathbb{R}^n)$ to $H^p(\mathbb{R}^n)$ for
$p\in(n/(n+k),\,1]$ and establish the Riesz transform
characterizations of $H_{L_1}^p(\mathbb{R}^n)$ for $
p\in(rn/(n+kr),\,1]$ if $\{e^{-tL_1}\}_{t>0}$ satisfies the
$L^r-L^2$ $k$-off-diagonal estimates with $r\in (1,2]$. These
results when $k:=1$ and $L:= L_1$ are known.
\end{abstract}}\end{center}}

%  Keyword is required.

%  \subjclass is required.

%%%%%%%%%%%%%%%%%%%%%%%%%%%%%%%%%%%%%%%%%%%%%%%%%%%%%%%%%%%%
\renewcommand{\baselinestretch}{1.2}
\begin{center} \renewcommand{\arraystretch}{1.5}
{\begin{tabular}{lp{0.8\textwidth}} \hline \scriptsize {\bf
Citation:}\!\!\!\!&\scriptsize Cao J, Yang D C. Hardy spaces
$H_L^p({\mathbb R}^n)$ associated with operators satisfying
$k$-Davies-Gaffney estimates. Sci China Math, 2012, 55, doi:
10.1007/s11425-000-0000-0\vspace{1mm}
\\
\hline
\end{tabular}}\end{center}

%%%%%%%%%%%%%%%%%%%%%%%%%%%%%%%%%%%%%%%%%%%%%%%%%%%%%%%%%%%%
%% Text of article.
%%%%%%%%%%%%%%%%%%%%%%%%%%%%%%%%%%%%%%%%%%%%%%%%%%%%%%%%%%%%
%    Section headings
\baselineskip 11pt\parindent=10.8pt  \wuhao

\section{Introduction}\label{s1}

The Hardy spaces, as a suitable substitute of
Lebesgue spaces $L^p(\rn)$, play an important role in various fields
of analysis and partial differential equations. It is well known
that the Hardy spaces $H^p(\rn)$ are essentially related to the
Laplacian operator $\bdz:=\sum_{j=1}^{n}\frac{\pat^2}{\pat x_j^2}$,
which have been intensively studied; see, for example,
\cite{sw,fs,cw77,tw,s,gr1} and the references therein.

In recent years, the study of Hardy spaces associated with different
differential operators inspires great interests; see, for example,
\cite{adm2,amr08,ar,cly,cds99,cks92,cks93,
dy05-1,dy05-2,dxy,dz1,dz2,gly10,gly,hlmmy,hm09,hm09c,hmm,lbyz10,lyy11,
jyz09,yy12} and their references. In particular, in \cite{adm2},
when the operator $L$ satisfies a pointwise Poisson upper bound,
Auscher, McIntosh and Duong introduced the Hardy space $H_L^1(\rn)$
associated with $L$ in terms of area integral functions. Later, in
\cite{dy05-1,dy05-2}, Duong and Yan introduced the $\bmo$-type space
$\bmo_L(\rn)$ associated with such an $L$ and proved the dual space
of $H_L^1(\rn)$ is $\bmo_{L^\ast}(\rn)$, where $L^\ast$ denotes the
adjoint operator of $L$ in $L^2(\rn)$. Yan \cite{y08} further
generalized these results to the Hardy space $H^p_L(\rn)$ with $p\in
(0,1]$ close to $1$ and its dual space. Also, the Orlicz-Hardy space
and its dual space associated with such an $L$ were studied in
\cite{jyz09,jya}.

Auscher and Russ \cite{ar} studied the Hardy space $H^1_L$ on
strongly Lipschitz domains associated with a second order divergence
form elliptic operator $L$ whose heat kernels have the Gaussian
upper bounds and regularity. Very recently, Auscher, McIntosh and
Russ \cite{amr08} treated the Hardy space $H^1$ associated with the
Hodge Laplacian on a Riemannian manifold with doubling measure;
Hofmann--Mayboroda in \cite{hm09,hm09c} and
Hofmann--Mayboroda--McIntosh in \cite{hmm} introduced the Hardy and
Sobolev spaces associated with a second order divergence form
elliptic operator $L$ on $\rn$ with bounded measurable complex
coefficients and these operators may not have the pointwise heat
kernel bounds, while a theory of the Orlicz-Hardy space and its dual
space associated with $L$ was independently developed in \cite{jy10,
jy10a}.

Moreover, a theory of Hardy spaces associated with the
Schr\"odinger operators $-\bdz+V$ was well developed, where the
nonnegative potential $V$ satisfies the reverse H\"older
inequality (see, for example, Dziuba\'nski and Zienkiewicz
\cite{dz1,dz2} and Yang and Zhou \cite{yz} and their references).
More generally, for nonnegative self-adjoint operators $L$
satisfying the Davies-Gaffney estimates, Hofmann et al.
\cite{hlmmy} introduced a new Hardy space $H_L^1(\rn)$, which was
extended to the Orlicz-Hardy space by Jiang and Yang \cite{jy}.
Recently, the Hardy space $H_{(-\bdz)^2+V^2}^1(\rn)$ associated
with the Schr\"odinger type operators $(-\bdz)^2+V^2$ was also
studied in \cite{cly}.

From now on, in what follows of this paper, we {\it always let} $L$
be a one-to-one operator of type $\omega$ having a bounded
$H_\infty$ functional calculus and satisfying the $k$-Davies-Gaffney
estimates with $k\in{\mathbb N}$ (see \eqref{2.6} below). Motivated
by \cite{hmm,hlmmy}, in this paper, we introduce the Hardy space
$H_L^p(\mathbb{R}^n)$ with $p\in (0,\,1]$ associated with $L$ in
terms of the square function defined via $\{e^{-t^{2k}L}\}_{t>0}$
(see \eqref{4.1} below) and establish their molecular and
generalized square function characterizations. Typical examples of
such operators include the {\it $2k$-order divergence form
homogeneous elliptic operator $L_1$ with complex bounded measurable
coefficients} and the {\it $2k$-order Schr\"odinger type operator}
$L_2:= (-\Delta)^k+V^k$, where $\Delta$ is the Laplacian and $0\le
V\in L^k_{\mathop\mathrm{loc}}(\mathbb{R}^n)$. Moreover, as an
application, for $i\in\{1,\,2\}$, we prove that the associated Riesz
transform $\nabla^k(L_i^{-1/2})$ is bounded from
$H_{L_i}^p(\mathbb{R}^n)$ to $H^p(\mathbb{R}^n)$ for
$p\in(n/(n+k),\,1]$ and establish the Riesz transform
characterizations of $H_{L_1}^p(\mathbb{R}^n)$ for $
p\in(rn/(n+kr),\,1]$ if $\{e^{-tL_1}\}_{t>0}$ satisfies the
$L^r-L^2$ $k$-off-diagonal estimates with $r\in (1,2]$ (see
Definition \ref{d6.1} below for the definition). These results when
$k:=1$ and $L:= L_1$ were already obtained recently by
Hofmann-Mayboroda \cite{hm09,hm09c}, Jiang-Yang \cite{jy10,jy}, and
Hofmann-Mayboroda-McIntosh \cite{hmm}.

A new ingredient appearing in this paper is the introduction of the
$k$-Davies-Gaffney estimates with $k\in{\mathbb N}$, which is
naturally satisfied by $2k$-order Schr\"odinger operators
$(-\Delta)^k+V^k$. Via the perturbation technique (see, for example,
\cite{bk1, bk2}) and some ideas from the proof of \cite[Lemma
2]{da2}, and using the elliptic condition, we further show that the
semigroup $\{e^{-tL_1}\}_{t>0}$ also satisfies the
$k$-Davies-Gaffney estimates.

Another new observation of this paper is that the nonnegative
self-adjoint property of operators in \cite{hlmmy,jy} can be
weakened into the assumption that $L$ has a bounded $H_\fz$
functional calculus. We point out that when this manuscript was in
preparation, we learned from  Anh and Li \cite{al} that this was
also observed by Duong and Li \cite{dl}.

This paper is organized as follows. In Section \ref{s2}, we first
recall some results on the $H_\fz$ functional calculus and
describe some assumptions on operators considered in this paper.
In particular, we introduce the notion of $k$-Davies-Gaffney
estimates with $k\in\nn$ in \eqref{2.6} below. Some examples
satisfying these assumptions are also given in this section.

Let $L$ be an operator satisfying assumptions in Section \ref{s2}.
In Section \ref{s3}, using some ideas from
\cite{hlmmy,hm09,hm09c,hmm}, we establish some off-diagonal
estimates for some families of operators related to $L$. More
precisely, we show that if $\{e^{-tL}\}_{t>0}$ satisfies the
$k$-Davies-Gaffney estimates, then the family
$\{(zL)^me^{-zL}\}_{z\in S_{\ell(\pi/2-\omega)}^0}$ of operators for
any $m\in\nn\cup\{0\}$ also satisfies the $k$-Davies-Gaffney
estimates in $z$ (see Lemma \ref{l3.1}), the $k$-Davies-Gaffney
estimates are stable under compositions (see Lemma \ref{l3.2}) and
the family $\{\psi(tL)f(L)\}_{t>0}$ of operators satisfies the
$k$-Davies-Gaffney estimates of order $\sz$ (see \eqref{3.7} below
for the definition), where $\psi$ belongs to the decaying function
class $\Psi_{\sz,\tau}(S_\mu^0)$ as in \eqref{2.2} below (see Lemma
\ref{l3.3} below). Let $L_1$ be the $2k$-order divergence form
homogeneous elliptic operator with complex bounded measurable
coefficients and $L_2$ the $2k$-order Schr\"odinger type operator.
In this section, we also prove that the semigroup
$\{e^{-tL_1}\}_{t>0}$ and the family
$\{\sqrt{t}\nabla^ke^{-tL_i}\}_{t>0}$ of operators for $i\in\{1,2\}$
satisfy the $k$-Davies-Gaffney estimates, respectively, in
Propositions \ref{p3.1} and \ref{p3.2}.

In Definition \ref{4.1} of Section \ref{s4}, we first introduce the
Hardy space $H_L^p(\rn)$ for $p\in(0,\,1]$ in terms of the square
function $S_L$ defined via $\{e^{-t^{2k}L}\}_{t>0}$ and, in
Definition \ref{4.2}, the molecular Hardy space
$H_{L,\,\mol,\,M}^p(\rn)$ with $M\in (n(1/p-1/2)/(2k),\fz)$. Then,
by using Lemma \ref{l3.1}, we prove that for each
$(H_L^p,\,\ez,\,M)$-molecule $m$, $\|S_L(m)\|_{L^p(\rn)}$ is
uniformly bounded (see \eqref{4.4} below), which together with a
boundedness criteria from \cite{hmm} (see also Lemma \ref{4.1}
below) implies that $H_{L,\,\mol,\,M}^p(\rn)\subset H_L^p(\rn)$. On
the other hand, using the atomic decomposition of the tent space
$T^p(\mathbb{R}^{n+1}_{+})$ and the $k$-Davies-Gaffney estimate, we
obtain that the operator $\pi_{M,\,L}$ in \eqref{4.11} maps any
$T^p(\mathbb{R}^{n+1}_+)$-atom into an $(H_L^p,\,\ez,\,M)$-molecule
up to a harmless positive constant multiple in Lemma \ref{l4.2}
below. Then, by a Calder\'on reproducing formula, we establish a
molecular decomposition of $H_L^p(\rn)$ which yields another
inclusion $H_L^p(\rn)\subset H_{L,\,\mol,\,M}^p(\rn)$. Thus, we
obtain the molecular characterization of $H_L^p(\rn)$ in Theorem
\ref{t4.1} below.

Section \ref{s5} is devoted to the generalized square function
characterization of $H_L^p(\rn)$. Motivated by \cite{hmm}, we first
introduce the generalized square function Hardy space
$H_{\psi,L}^p(\rn)$ for $p\in(0,1]$ and some
$\psi\in\Psi_{\sz,\tau}(S_\mu^0)$ in Definition \ref{d5.1} below.
Then, for any $\psi\in\Psi_{\sz,\tau}(S_\mu^0)$ and all $f\in
H_\fz(S_\mu^0)$ (see \eqref{2.1} for the definition), we introduce
the operators $Q_{\psi,L}$, $\pi_{\psi,L}$ and their composition
$Q^{f}$ (see \eqref{5.1}, \eqref{5.4} and \eqref{5.5} for their
definitions). Using the $k$-Davies-Gaffney estimates of order $\sz$
for $\{\psi(tL)f(L)\}_{t>0}$ in Lemma \ref{l3.3} below, we prove
that the operator $Q^f$ is bounded on the tent space
$T^p(\mathbb{R}^{n+1}_{+})$ (see Lemma \ref{l5.2}), $Q_{\psi,L}$ is
bounded from $H_L^p(\rn)$ to $T^p(\mathbb{R}^{n+1}_{+})$ and
$\pi_{\psi,L}$ is bounded from $T^p(\mathbb{R}^{n+1}_{+})$ to
$H_L^p(\rn)$ for some $\psi$ (see Lemma \ref{l5.3} below). Combining
these boundedness and using a Calder\'on reproducing formula in
\eqref{5.14}, we then obtain the generalized square function
characterization of $H_L^p(\rn)$ in Theorem \ref{t5.1}, which is
used in obtaining the Riesz transform characterization of
$H_{L_1}^p(\rn)$ in Section \ref{s6}. For all $\az\in(0,\,\fz)$, let
$L^\az$ be the fractional power with exponent $\az$ of $L$ and the
Hardy space $H_{L^\az}^p(\rn)$ be defined as in \eqref{5.3} below
via the square function $S_{L^\az}$ as in \eqref{5.2}. As another
application of Theorem \ref{t5.1}, we then obtain in Corollary
\ref{c5.1} that $H_{L^\az}^p(\rn)=H_L^p(\rn)$ with equivalent norms,
in particular, $H_{(-\bdz)^k}^p(\rn)=H^p(\rn)$ with equivalent norms
for all $k\in\nn$, where $H^p(\rn)$ is the classical Hardy space in
\cite{sw,fs}.

Finally, in Section \ref{s6}, we concentrate on the behavior of
the Riesz transforms $\nabla^k L_i^{-1/2}$ on $H^p_{L_i}(\rn)$ for
$i\in\{1,\,2\}$. By the gradient estimates of the semigroup
$\{e^{-tL_i}\}_{t>0}$ in Proposition \ref{p3.2} and the
composition rule of $k$-Davies-Gaffney estimates in Lemma
\ref{l3.2}, we first show that the two families of operators,
$\{\nabla^kL_i^{-1/2}(I-e^{-tL_i})^M\}_{t>0}$ and
$\{\nabla^kL_i^{-1/2}(tL_ie^{-tL_i})^M\}_{t>0}$ for all $M\in\nn$,
satisfy some estimates similar to the  $k$-Davies-Gaffney
estimates of order $M$ (see Lemma \ref{l6.1} below). Then, using
these estimates, we prove that for each
$(H_{L_i}^p,\,\ez,\,M)$-molecule $m$ with $p\in(n/(n+k),\,1]$ and
$M\in(n(1/p-1/2)/(2k),\,\fz)$, $\nabla^k(L_i^{-1/2})(m)$ is a
classical $H^p(\rn)$-molecule up to a harmless constant multiple,
which further implies that Riesz transforms $\nabla^k(L_i^{-1/2})$
are bounded from $H_{L_i}^p(\rn)$ to the classical Hardy space
$H^p(\rn)$ in Theorem \ref{t6.1} below. In the remainder of this
section, motivated by \cite{hmm}, by assuming that the semigroup
$\{e^{-tL_1}\}_{t>0}$ satisfies the $L^r-L^2$ $k$-off-diagonal
estimates for $r\in (1, 2]$, we then establish the Riesz transform
characterization of $H_{L_1}^p(\rn)$. To this end, we first show
in Lemma \ref{l6.2} below that $\{tL_1e^{-tL_1}\}_{t>0}$ also
satisfies the $L^r-L^2$ $k$-off-diagonal estimates. We then recall
some known results concerning the homogeneous Triebel-Lizorkin
space $\dot F^\az_{p,q}(\rn)$ and their atomic characterizations
from \cite{tr,hpw,ck} and \cite[Proposition 4.3]{wz}.  Let $\dot
W^{k,2}(\rn)$ be the homogenous Sobolev space of order $k$. With
the help of these results, we show that if $f\in \dot
W^{k,2}(\rn)\cap\dot H^{k,p}(\rn)$ when $p\in (0,\,1]$, then its
atomic decomposition converges in both $\dot W^{k,2}(\rn)$ and
$\dot H^{k,p}(\rn)$ (see Lemma \ref{l6.3} below). Moreover, by the
$L^r-L^2$ $k$-off-diagonal estimates for
$\{tL_1e^{-tL_1}\}_{t>0}$, we prove that for each
$H^{k,p}(\rn)$-atom $b$, $S_1\sqrt{{L_1}}(b)$ is uniformly bounded
on $L^p(\rn)$ (see \eqref{6.12} below), which, together with the
generalized square function characterization of $H_{{L_1}}^p(\rn)$
in Theorem \ref{t5.1} and Lemma \ref{l6.3}, shows that
$S_1\sqrt{L_1}$ is bounded from the Hardy-Sobolev space $\dot
H^{k,p}(\rn)$ to $L^p(\rn)$. This, combined with the boundedness
of Riesz transforms on $H^p_{L_1}(\rn)$ in Theorem \ref{t6.1},
yields the Riesz transform characterization of $H_{L_1}^p(\rn)$ in
Theorem \ref{t6.2} below. We point out that in the proof of the
estimate \eqref{6.12}, we use the embedding result \eqref{6.14}
below on the homogeneous Triebel-Lizorkin space from \cite{tr} and
another key fact from \cite[Theorem 1.1]{ahmt} that
$\|\sqrt{L_1}f\|_{L^2(\rn)}\ls\|\nabla^kf\|_{L^2(\rn)}$. The
latter fact may not be true for $L_2$; see Remark \ref{r6.1}
below. Thus, it seems that one needs some new ideas to obtain the
Riesz characterization of $H_{L_2}^p(\rn)$.

We now make some conventions on the notation. Throughout the whole
paper, we always let $\nn:=\{1,2,\cdots\}$ and $\zz_+:=
\nn\cup\{0\}$. Denote the {\it differential operator}
$\frac{\pat^{|\az|}}{\pat x_1^{\az_1}\cdots \pat x_1^{\az_1}}$
simply by $\pat^\az$, where $\az:=(\az_1,\cdots,\az_n)$ and
$|\az|:=\az_1+\cdots+\az_n$. We also denote the {\it $2k$-order
divergence form homogenous elliptic operator with complex bounded
measurable coefficients}
$(-1)^k\sum_{|\alpha|=|\beta|=k}\partial^\alpha (a_{\alpha,\beta}
\partial^\beta)$
by $L_1$ and the {\it $2k$-order Schr\"odinger type operator
$(-\bdz)^k+V^k$} by $L_2$. We use $C$ to denote a {\it positive
constant}, that is independent of the main parameters involved but
whose value may differ from line to line, and $C(\az,\cdots)$ to
denote a {\it positive constant} depending on the parameters
$\az,$ $\cdots$. {\it Constants with subscripts}, such as $C_0$,
do not change in different occurrences. If $f\le Cg$, we then
write $f\ls g$; and if $f\ls g\ls f$, we then write $f\sim g$. For
all $x\in\rr^n$ and $r\in(0,\fz),$ let
$B(x,r):=\{y\in\rr^n:|x-y|<r\}$. Also, for any set $E\in \rn$, we
use $E^\complement$ to denote $\rn\setminus E$ and $\chi_E$ its
{\it characteristic function}.

\section{Preliminaries}\label{s2}

We first collect some basic results on the theory of $H_\fz$
functional calculus, developed by McIntosh in \cite{mc}, that we
need in what follows. For more details and further references
about functional calculus, we refer the reader to
\cite{adm1,ha,mc} and the references therein.

For $\tz\in[0,\,\pi)$, the {\it open and closed sectors, $S_\tz^0$
and $S_\tz$, of angle $\tz$} in the complex plane $\cc$ are defined
as follows:
$$S_\tz^0:=\lf\{z\in\cc\setminus\{0\}:\ |\arg z|<\tz\r\}$$
and
$$S_\tz:=\lf\{z\in\cc\setminus\{0\}:\ |\arg
z|\le\tz\r\}\cup\lf\{0\r\}.$$ Let $\omega\in[0,\,\pi)$. A closed
operator $T$ in $L^2(\rn)$ is called of {\it type} $\omega$, if
the spectrum of $T$, $\sz(T)$, is contained in $S_\omega$, and for
each $\tz\in (\omega,\,\pi)$, there exists a nonnegative constant
$C$ such that for all $z\in\cc\setminus S_\tz$,
$\|(T-zI)^{-1}\|_{\cl(L^2(\rn))}\le C|z|^{-1}$, here and in what
follows, $\|S\|_{\cl(\mathcal H)}$ denotes the {\it operator norm}
of the linear operator $S$ on the normed linear space $\mathcal
H$.

For $\mu\in[0,\,\pi)$ and $\sz,\tau\in(0,\,\fz)$, we need the
following spaces of  functions:
$$H(S_\mu^0):= \lf\{f:\ f\ \text{is holomorphic on}\
S_\mu^0\r\},$$
\begin{eqnarray}\label{2.1}
H_{\fz}(S_\mu^0):= \lf\{f\in H(S_\mu^0):\
\|f\|_{L^{\fz}(S^0_{\mu})}<\fz\r\}
\end{eqnarray}
and
\begin{eqnarray}\label{2.2}
\Psi_{\sz,\tau}(S_\mu^0):=\lf\{f\in H(S_\mu^0):\ |f(\xi)|\le C
\inf\{|\xi|^\sz,\,|\xi|^{-\tau}\}\ \text{for all}\ \xi\in
S_\mu^0\r\}.
\end{eqnarray}

It is known that every one-to-one operator $T$ of type $\omega$ in
$L^2(\rn)$ has a unique holomorphic functional calculus which is
consistent with the usual definition of polynomials of operators
(see, for example, \cite{mc}). More precisely, let $T$ be a
one-to-one operator of type $\omega$, with $\omega\in[0,\,\pi)$,
$\mu\in(\omega,\,\pi)$, $\sz,\tau\in(0,\,\fz)$, and
$f\in\Psi_{\sz,\tau}(S_\mu^0)$. The function of the operator $T$,
$f(T)$ can be defined by the $H_\fz$ functional calculus in the
following way,
\begin{eqnarray}\label{2.3}
f(T):=\frac{1}{2\pi i}\dint_\gamma(\xi I-T)^{-1}f(\xi)\,d\xi,
\end{eqnarray}
where $\gamma:=\{re^{i\nu}:\ \fz>r>0\}\cup\{re^{-i\nu}:\ 0<r<\fz\}$,
$\nu\in(\omega,\,\mu)$, is a curve consisting of two rays
parameterized anti-clockwise. It is well known that the above
definition is independent of the choice of $\nu\in(\omega,\,\mu)$
and the integral in \eqref{2.3} is absolutely convergence in
$\cl(L^2(\rn))$ (see \cite{mc,ha}).

In what follows, we {\it always assume $\omega\in[0,\,\pi/2)$}.
Then, it follows from \cite[Proposition 7.1.1]{ha} that for every
operator $T$ of type $\omega$ in $L^2(\rn)$, $-T$ generates a
holomorphic $C_0$-semigroup $\{e^{-zL}\}_{z\in S^0_{\pi/2-\omega}}$
on the open sector $S^0_{\pi/2-\omega}$ such that
$\|e^{-zL}\|_{\cl(L^2(\rn))}\le1$ for all $z\in S^0_{\pi/2-\omega}$
and, moreover, every nonnegative self-adjoint operator is of type
$0$.

Let $\Psi(S_\mu^0):=\cup_{\sz,\tau>0}\Psi_{\sz,\tau}(S_\mu^0)$. By
the relationship between the associated semigroup and the resolvent
of $T$, for all $f\in\Psi(S_\mu^0)$, $f(T)$ can further be
represented as
\begin{eqnarray}\label{2.4}
f(T)=\dint_{\Gamma_+} e^{-zT}\eta_+(z)\,dz+\dint_{\Gamma_{-}}
e^{-zT}\eta_-(z)\,dz,
\end{eqnarray}
where
\begin{eqnarray}\label{2.5}
\eta_\pm(z):=\frac{1}{2\pi i}\dint_{\gamma\pm}e^{\xi z}f(\xi)\,
d\xi, \ \ \ z\in\Gamma_{\pm},
\end{eqnarray}
$\Gamma_{\pm}:=\rr^+e^{\pm i(\pi/2-\tz)}$, $\gamma_{\pm}:=
\rr^+e^{\pm i\nu}$ and $0\le\omega<\tz<\nu<\mu<\pi/2$. Here and in
what follows, $\rr^+:=(0,\fz)$.

It is well known that the above holomorphic functional calculus
defined on $\Psi(S_\mu^0)$ can be extended to $H_\fz(S_\mu^0)$ via a
limit process (see \cite{mc}).  Recall that for $\mu\in(0,\,\pi)$,
the operator $T$ is said to {\it have a bounded $H_\fz(S_\mu^0)$
functional calculus} in the Hilbert space $\mathcal{H}$, if there
exists a positive constant $C$ such that for all $\psi\in
H_\fz(S_\mu^0)$, $\|\psi(T)\|_{\cl(\mathcal{H})}\le
C\|\psi\|_{L^\fz(S_\mu^0)}$ and $T$ is called to have {\it a bounded
$H_\fz$ functional calculus} in the Hilbert space $\mathcal{H}$ if
there exists $\mu\in (0,\,\pi)$ such that $T$ has a bounded
$H_\fz(S_\mu^0)$ functional calculus.

Now, we describe our assumptions of operators $L$ considered in this
paper. Throughout the whole paper, we {\it always assume} that $L$
satisfies the following {\it assumptions}:
\begin{enumerate}
\item[(A$_1$)]
 The operator $L$ is a one-to-one operator of
type $\omega$ in $L^2(\rn)$ with $\omega\in[0,\,\pi/2)$;
\item[(A$_2$)] The operator $L$ has a bounded $H_\fz$
functional calculus in $L^2(\rn)$;
\item[(A$_3$)] Let $k\in \nn$. The operator $L$
generates a holomorphic semigroup $\{e^{-tL}\}_{t>0}$ which
satisfies the {\it $k$-Davies-Gaffney estimate}, namely, there exist
positive constants $\wz C$ and $C_1$ such that for all closed sets
$E$ and $F$ in $\rn$, $t\in(0,\,\fz)$ and $f\in L^2(\rn)$ supported
in $E$,
\begin{eqnarray}\label{2.6}
\|e^{-tL}f\|_{L^2(F)}\le \wz C
\exp\lf\{-\frac{\lf[\dist(E,\,F)\r]^{2k/(2k-1)}}{C_1t^{1/(2k-1)}}\r\}\|f\|_{L^2(E)},
\end{eqnarray}
here and in what follows, $\dist(E,\,F):= \inf_{x\in E,\,y\in
F}|x-y|$ is the distance between $E$ and $F$.
\end{enumerate}

\begin{remark} \label{r2.1}
We point out that when $k=1$, the $k$-Davies-Gaffney estimate is
usually called the {\it Davies-Gaffney estimate} (or the {\it
$L^2$ off-diagonal estimate} or just the {\it Gaffney estimate})
(see, for example, \cite{hm09,hm09c,hlmmy,jy,hmm}).
\end{remark}

Let $k\in\nn$. Examples of operators, satisfying the above
assumptions (A$_1$), (A$_2$) and (A$_3$), include the following
$2k$-order divergence form homogeneous elliptic operator:
$$L_1:=(-1)^k\dsum_{|\az|=|\bz|=k}\pat^\az(a_{\az,\bz}\pat^\bz)$$
with complex bounded measurable coefficients $a_{\az,\bz}$ for all
multi-indices $\az,\,\bz$ and the $2k$-order Schr\"odinger type
operator $L_2:=(-\bdz)^k+V^k$ with $0\le V\in L^k_{\loc}(\rn)$. More
precisely, let $ W^{k,2}(\rn)$ be the {\it Sobolev space of order
$k$} endowed with the {\it norm}
$$\|\cdot\|_{W^{k,2}(\rn)}:=\sum_{0\le|\az|\le k}\|\pat^\az(\cdot)\|_{L^2(\rn)}.$$
Denote by $\mathfrak{a}$ the {\it sesquilinear form} given by
\begin{eqnarray}\label{2.7}
\mathfrak{a}(f,\,g):=\dint_\rn\dsum_{|\az|=|\bz|=k}a_{\az,\bz}(x)\pat^\bz
f(x)\overline{\pat^\az g(x)}\,dx
\end{eqnarray}
with domain $D(\mathfrak{a}):= W^{k,2}(\rn)$. We further assume that
$\mathfrak{a}$ satisfies the {\it ellipticity condition}, that is,
there exist positive constants $0<\lz\le\Lambda<\fz$ such that
\begin{eqnarray}\label{2.8}
\|a_{\az,\bz}\|_{L^\fz(\rn)}\le\Lambda\ \ \  \text{for all}\
\az,\,\bz\ \text{with}\ |\az|=k=|\bz|
\end{eqnarray}
and
\begin{eqnarray}\label{2.9}
\Re \mathfrak{a}(f,\,f)\ge \lz\,\|\nabla^kf\|_{L^2(\rn)}^2\ \ \
\text{for all} \ f\in W^{k,2}(\rn),
\end{eqnarray}
here and in what follows, $\Re z$ for any $z\in\cc$ denotes the
{\it real part} of $z$. The {\it 2k-order divergence form
homogeneous elliptic operator} $L_1$ with complex bounded
measurable coefficients
 is then defined to be the operator associated with the form $\mathfrak{a}$.

Let $\omega\in[0,\,\pi/2]$. Recall that an operator $T$ in the
Hilbert space $\mathcal{H}$ is called $m$-$\omega$-{\it accretive}
if
\begin{enumerate}
\item[(i)] the range of the operator $T+I$, $R(T+I)$, is dense in $\mathcal{H}$;
\item[(ii)] for all $u\in D(T)$, $|\arg(Tu,\,u)|\le\omega$,
\end{enumerate}
where $D(T)$ denotes the {\it domain} of $T$ and $\arg(Tu,\,u)$ the
{\it argument} of $(Tu,\,u)$. It is known by \cite[Proposition
7.1.1]{ha} that every closed $m$-$\omega$-accretive operator is of
type $\omega$ (see \cite[p.\,173]{ha}).

From \cite{ahmt}, it follows that $L_1$ is closed and maximal
accretive (see \cite[p.\,327]{ha} for the definition), which
further yields that $R(L_1+I)$ is dense in $L^2(\rn)$ (see, for
example \cite[Proposition C.7.2]{ha}). Moreover, by the
ellipticity condition \eqref{2.8} and \eqref{2.9}, we obtain that
for all $f\in W^{k,2}(\rn)$,
$$\lf|\tan\lf(\arg(L_1 f,\,f)\r)\r|=\lf|\frac{\Im(L_1 f,\,f)}{\Re(L_1 f,\,f)}\r|
\le\frac{\Lambda}{\lz},$$ here and in what follows, $\Im z$ for
any $z\in\cc$ denotes the {\it imaginary part} of $z$. Thus,
$|\arg(L_1 f,\,f)|\le\arctan\frac{\Lambda}{\lz}$, which, together
with the fact that $R(L_1+I)$ is dense in $L^2(\rn)$, shows that
$L_1$ is an $m$-$\arctan\frac{\Lambda}{\lz}$-accretive operator in
$L^2(\rn)$ with the angle $\arctan\frac{\Lambda}{\lz}\in
[\pi/4,\,\pi/2)$. Thus, $L_1$ is an operator of type
$\arctan\frac{\Lambda}{\lz}$.

Now, we show that $L_1$ is one-to-one. Let $N(L_1):= \{f\in D(L_1):\
L_1f=0\}$ be the {\it null space} of $L_1$. For any fixed $f\in
N(L_1)$, by the elliptic condition \eqref{2.8} and \eqref{2.9}, we
have
\begin{eqnarray}\label{2.10}
\dint_{\rn}\lf|\nabla^kf(x)\r|^2\,dx\sim|(L_1f,\,f)|=0,
\end{eqnarray}
which implies that $\nabla^kf=0$ almost everywhere in $\rn$. In
what follows, denote by $C_c^\fz(\rn)$ the {\it space of all
$C^\fz$ functions with compact support in $\rn$}. For all $f\in
W^{k,2}(\rn)$, by the density of $C_c^\fz(\rn)$ in $
W^{k,2}(\rn)$, there exists a sequence $\{f_j\}_{j\in\nn}$ of
functions in $C_c^\fz(\rn)$ such that $\lim_{j\to \fz}f_j=f$ in $
W^{k,2}(\rn)$. Denote the {\it Fourier transform} and the {\it
inverse Fourier transform} of $f$, respectively, by $\widehat f$
and $f^\vee$. If $f\in N(L_1)$, by \eqref{2.10}, the fact that
$f_j\in C_c^\fz(\rn)$, the multiplication formula of Fourier
transform and Plancherel's theorem (see, for example,
\cite[Theorem 2.2.14]{gr}), we have that for all $\vz\in
C_c^\fz(\rn)$,
\begin{eqnarray*}
0&&=(\nabla^kf,\,\widehat\vz)=\dlim_{j\to\fz}(\nabla^kf_j,\,
\widehat{\vz})=\dlim_{j\to\fz}(-1)^k(f_j,\,\nabla^k\widehat{\vz})=
\dlim_{j\to\fz}(-1)^k(\widehat{f_j},\,(\nabla^k\widehat{\vz})^\vee)\\
&&=\dlim_{j\to\fz}i^k(\widehat{f_j},\,(\cdot)^k\vz(\cdot))
=i^k(\widehat{f},\,(\cdot)^k\vz(\cdot)),
\end{eqnarray*}
which implies that $\supp{\widehat{f}}\subset\{0\}$. By \cite[Corollary
2.4.2]{gr}, we have that $f$ is a polynomial, which, together with
the fact that $f\in L^2(\rn)$, implies that $f=0$. Hence,
$N(L_1)=\{0\}$ and $L_1$ is one-to-one.

Since $L_1$ is maximal accretive, from \cite{adm1}, it follows
that $L_1$ has a bounded holomorphic functional calculus. Finally,
in Proposition \ref{p3.1} below, we show that the semigroup
$\{e^{-tL_1}\}_{t>0}$ satisfies the $k$-Davies-Gaffney estimate.
Thus, the {\it$2k$-order divergence form homogenous elliptic
operator $L_1$  with complex bounded measurable coefficients
satisfies the assumptions {\rm(A$_1$), (A$_2$)} and {\rm(A$_3$)}}.

Let $k\in\nn$, $\bdz:=\sum_{j=1}^{n}\frac{\pat^2}{\pat x_j^2}$ be
the Laplace operator and $0\le V\in L_{\loc}^k(\rn)$. The {\it
$2k$-order Schr\"odinger type operator} $L_2:=(-\bdz)^k+V^k$ is the
associated operator of the following sesquilinear form
\begin{eqnarray*}
\mathfrak{b}(f,\,g):=
\dint_\rn\nabla^kf(x)\overline{\nabla^kg(x)}\,dx+\dint_\rn
[V(x)]^kf(x)\overline{g(x)}\,dx
\end{eqnarray*}
with domain $D(\mathfrak{b}):= \{f\in W^{k,2}(\rn):\ \int_{\rn}
[V(x)]^k |f(x)|^2\,dx<\fz\}$ which is also dense in $L^2(\rn)$,
since $C_c^\fz(\rn)\subset D(\mathfrak{b})$.

It is easy to see that the $2k$-order Schr\"odinger type operator
$L_2$ is a nonnegative self-adjoint operator. From \cite{ha}, it
follows that $L_2$ is $m$-0-accretive. Thus, by \cite[Proposition
7.1.1]{ha}, $L_2$ is a one-to-one operator of type $0$. Therefore,
$L_2$ has a bounded $H_\fz$ functional calculus. Moreover, by
\cite{bd}, the semigroup $\{e^{-tL_2}\}_{t>0}$ satisfies a Gaussian
type estimate, that is, the integral kernel $e^{-tL_2}(x,\,y)$ of
$e^{-tL_2}$ has the property that there exist positive constants
$C_2$ and $C_3$ such that for all $t\in(0,\,\fz)$ and $x,\,y\in
\rn$,
\begin{eqnarray*}
|e^{-tL_2}(x,\,y)|\le
C_2t^{-n/(2k)}\exp\lf\{-C_3\frac{|x-y|^{2k/(2k-1)}}{t^{1/(2k-1)}}\r\},
\end{eqnarray*}
which implies that the semigroup $\{e^{-tL_2}\}_{t>0}$ satisfies the
$k$-Davies-Gaffney estimate immediately. Thus,  the {\it$2k$-order
schr\"odinger type operator $L_2$ also satisfies the assumptions
{\rm(A$_1$), (A$_2$)} and {\rm(A$_3$)}}.

\section{$k$-Davies-Gaffney estimates}\label{s3}

 In this section, we prove some properties about
the $k$-Davies-Gaffney estimates. We point out that when $k=1$ and
$L$ is a non-negative self-adjoint operator or a second order
divergence form elliptic operator with complex bounded measurable
coefficients, these properties are already well known (see, for
example, \cite{ahlmt,hm09,hm09c,hlmmy,jy,hmm}).

Let $\tz\in[0,\,\pi/2)$ and $E, \,F$ be two closed sets in $\rn$. A
family $\{T(z)\}_{z\in S_\tz^0}$ of operators is said to satisfy the
{\it $k$-Davies-Gaffney estimate in $z$} if there exist positive
constants $C_4$ and $C_5$ such that for all $f\in L^2(\rn)$
supported in $E$ and $z\in S_\tz^0$,
\begin{eqnarray}\label{3.1}
\|T(z)f\|_{L^2(F)}\le C_5
\exp\lf\{-\frac{\lf[\dist(E,\,F)\r]^{2k/(2k-1)}}{C_4|z|^{1/(2k-1)}}
\r\}\|f\|_{L^2(E)}.
\end{eqnarray}

For any operator satisfying the assumptions {\rm(A$_1)$},
{\rm(A$_2)$} and {\rm(A$_3)$} in Section \ref{s2}, we have the
following property.

\begin{lemma}\label{l3.1}
Assume that the operator $L$ defined in $L^2(\rn)$ satisfies the
assumptions {\rm(A$_1)$}, {\rm(A$_2)$} and {\rm(A$_3)$} in Section
\ref{s2}. Then for all $\ell\in(0,\,1)$, $ m\in\zz_+$, the family of
operators, $\{(zL)^me^{-zL}\}_{z\in
S_{\ell(\frac{\pi}{2}-\omega)}^0}$, satisfies the $k$-Davies-Gaffney
estimate in $z$, \eqref{3.1}, with positive constants $C_4$ and
$C_5$ depending only on $m$, $\ell$, $n$, $k$, $\omega$, $\wz C$ and
$C_1$.
\end{lemma}

\begin{proof}
We prove this lemma by using some ideas from \cite{hlmmy}. Since
$L$ is of type $\omega$, we know that the semigroup
$\{e^{-tL}\}_{t>0}$ can be extended to a holomorphic semigroup
$\{e^{-zL}\}_{z\in S_{\pi/2-\omega}^0}$. Thus, for all $z\in
S_{\pi/2-\omega}^0$, closed sets $E$, $F\subset\rn$ and $f,\, g\in
L^2(\rn)$ supported respectively in $E$ and $F$, the function
$$G(z):\  z\longmapsto(e^{-zL}f,\,g)$$
is holomorphic on
$S_{\pi/2-\omega}^0$. Moreover, $G$ satisfies the following
properties:
\begin{enumerate}
\item[(i)] there exists a nonnegative constant $C$ such that for all
$z\in S_{\pi/2-\omega}^0$,
\begin{eqnarray*}
|G(z)|\le C \|f\|_{L^2(E)}\|g\|_{L^2(F)},
\end{eqnarray*}
\item[(ii)] there exist nonnegative constants $C$ and $C_1$ such
that for all $t\in(0,\,\fz)$,
\begin{eqnarray*}
|G(t)|\le C \exp\lf\{{-\frac{\lf[\dist(E,\,F)\r]^{2k/(2k-1)}}{
C_1t^{1/(2k-1)}}}\r\} \|f\|_{L^2(E)}\|g\|_{L^2(F)}.
\end{eqnarray*}
\end{enumerate}

In \cite[Lemma 6.18]{ou}, letting $\psi:=\pi/2-\omega$, $a:=
C\|f\|_{L^2(E)}\|g\|_{L^2(F)}$,
$b:=\frac{1}{C_1}\lf[\dist(E,\,F)\r]^{2k/(2k-1)}$, $r:= t$,
$\az:=\frac{1}{2k-1}$ and $\bz:=0$, we then obtain that for any $z:=
re^{i\tz}\in S_{\ell(\pi/2-\omega)}^0$,
\begin{eqnarray}\label{3.2}
\hs\hs\hs|F(z)|&&\ls\exp\lf\{-
\frac{\lf[\dist(E,\,F)\r]^{2k/(2k-1)}}{2(2k-1)C_1r^{1/(2k-1)}}
\sin(\pi/2-\omega-|\tz|)\r\}\|f\|_{L^2(E)}\|g\|_{L^2(F)}\nonumber\\
&&\ls\exp\lf\{-\frac{\lf[\dist(E,\,F)\r]^{2k/(2k-1)}}
{C_4r^{1/(2k-1)}} \r\}\|f\|_{L^2(E)}\|g\|_{L^2(F)},
\end{eqnarray}
where $C_4:=
\frac{2C_1(2k-1)}{\sin((1-\ell)(\frac{\pi}{2}-\omega))}$. From the
analytic property of semigroups and the Cauchy integral formula,
it follows that for all $m\in\nn$ and $z\in
S_{\ell(\frac{\pi}{2}-\omega)}^0$,
\begin{eqnarray}\label{3.3}
(zL)^me^{-zL}=(-z)^m\frac{m!}{2\pi i}\dint_{|\xi-z|=\eta |z|}
e^{-\xi L}\frac{d\xi}{(\xi-z)^{m+1}},
\end{eqnarray}
where $\eta\in(0,\,\sin((1-\ell)(\pi/2-\omega)))$. Thus, for any
$z\in S_{\ell(\frac{\pi}{2}-\omega)}^0$, the ball $B(z,\eta
|z|)\subset S_{\pi/2-\omega}^0$. Combining \eqref{3.2} and
\eqref{3.3}, by Minkowski's inequality, we obtain
\begin{eqnarray*}
\lf\|(zL)^me^{-zL}f\r\|_{L^2(F)}&&\ls |z|^m\dint_{|\xi-z|=\eta
|z|}\lf|\frac{1}{(\xi-z)^{m+1}}\r|
\lf\|e^{-\xi L}f\r\|_{L^2(F)}|d\xi|\\
&&\ls \exp\lf\{-\frac{\lf[\dist(E,\,F)\r]^{2k/(2k-1)}} {C_4
|z|^{1/(2k-1)}}\r\}\|f\|_{L^2(E)},
\end{eqnarray*}
which implies that
$\{(zL)^me^{-zL}\}_{S_{\ell(\frac{\pi}{2}-\omega)}^0}$ satisfies the
$k$-Davies-Gaffney estimate in $z$. This finishes the proof of Lemma
\ref{l3.1}.
\end{proof}

\begin{lemma}\label{l3.2}
Let $\{A_t\}_{t>0}$, $\{B_s\}_{s>0}$ be two families of linear
operators, $C_6$ and $C_7$ two positive constants. Assume that for
all closed sets $E$, $F\subset\rn$, $f\in L^2(\rn)$ supported in
$E$ and $t\in(0,\fz)$, the following estimates hold:
\begin{eqnarray}\label{3.4}
\|A_tf\|_{L^2(F)}\le C_6
\exp\lf\{-\frac{\lf[\dist(E,\,F)\r]^{2k/(2k-1)}}
{C_7t^{1/(2k-1)}}\r\}\|f\|_{L^2(E)},
\end{eqnarray}
and
\begin{eqnarray}\label{3.5}
\|B_sf\|_{L^2(F)}\le C_6
\exp\lf\{-\frac{\lf[\dist(E,\,F)\r]^{2k/(2k-1)}}
{C_7s^{1/(2k-1)}}\r\}\|f\|_{L^2(E)}.
\end{eqnarray}
Then, there exists a positive constant $C$ such that for all
$t,\,s\in(0,\fz)$, all closed sets $E$, $F\subset\rn$ and $f\in
L^2(\rn)$ supported in $E$,
\begin{eqnarray}\label{3.6}
\|A_tB_sf\|_{L^2(F)}\le C
\exp\lf\{-\frac{\lf[\dist(E,\,F)\r]^{2k/(2k-1)}}{\wz C_7(\max\{t,\,s
\})^{1/(2k-1)}}\r\}\|f\|_{L^2(E)},
\end{eqnarray}
where $\wz C_7:= C_72^{2k/(2k-1)}$.
\end{lemma}

\begin{proof} If $\dist (E,\,F)=0$, then \eqref{3.6} is a simple corollary
of \eqref{3.4} and \eqref{3.5}. Now, we assume that
$\dist(E,\,F)>0$. As in \cite{hm}, let $\rho:= \dist(E,\,F)$ and
$G:=\{x\in\rn:\ \dist(x,\,F)<\rho/2\}$. Denote by $\overline G$ the
{\it closure} of $G$. It is clear that $\dist(E,\,\overline
G)\ge\rho/2$. Moreover, by \eqref{3.4} and \eqref{3.5}, we have
\begin{eqnarray*}
\|A_t(\chi_G B_sf)\|_{L^2(F)} &&\le\|A_t(\chi_G B_sf)
\|_{L^2(\rn)}\ls\|B_sf\|_{L^2(\overline G)}\\
&&\ls\exp\lf\{-\frac{[\dist(E,\,\overline G)]^{2k/(2k-1)}}
{C_7 s^{1/(2k-1)}}\r\}\|f\|_{L^2(E)}\\
&&\sim\exp\lf\{-\frac{\lf[\dist(E,\,F)\r]^{2k/(2k-1)}} {C_7
2^{2k/(2k-1)} s^{1/(2k-1)}}\r\}\|f\|_{L^2(E)}.
\end{eqnarray*}
Let $\wz C_7:= C_72^{2k/(2k-1)}$. Similarly, by \eqref{3.4} and
\eqref{3.5}, we obtain
\begin{eqnarray*}
\|A_t(\chi_{\rn\setminus G}B_sf)\|_{L^2(F)}
&&\ls\exp\lf\{-\frac{\lf[\dist(\rn\setminus G,\,F)\r]^{2k/(2k-1)}}
{C_7t^{1/(2k-1)}}\r\}\|B_sf\|_{L^2(E)}\\
&&\ls\exp\lf\{-\frac{\lf[\dist(E,\,F)\r]^{2k/(2k-1)}} {\wz
C_7t^{1/(2k-1)}}\r\}\| f\|_{L^2(E)}.
\end{eqnarray*}
 Combining the above estimates, we have
\begin{eqnarray*}
\|A_tB_sf\|_{L^2(F)}
&&\le\|A_t(\chi_GB_sf)\|_{L^2(F)}+\|A_t(\chi_{\rn\setminus
G}B_sf)\|_{L^2(F)}\\
&&\ls\lf[\exp\lf\{-\frac{\lf[\dist(E,\,F)\r]^{2k/(2k-1)}} {\wz
C_7s^{1/(2k-1)}}\r\}+\exp\lf\{-\frac{\lf[\dist(E,\,F)\r]^{2k/(2k-1)}}
{\wz C_7t^{1/(2k-1)}}\r\}\r]\|f\|_{L^2(E)}\\
&&\ls\exp\lf\{-\frac{\lf[\dist(E,\,F)\r]^{2k/(2k-1)}} {\wz
C_7\max \{t,\,s\}^{1/(2k-1)}}\r\}\|f\|_{L^2(E)},
\end{eqnarray*}
which completes the proof of Lemma \ref{l3.2}.
\end{proof}

Let $\sz\in [0,\fz)$. As in \cite{hmm}, a family $\{T_t\}_{t>0}$ of
operators is said to satisfy the {\it $k$-Davies-Gaffney estimate of
order $\sz$}, if there exists a positive constant $C_8$, depending
on $\sz$, such that for all closed sets $E$, $F\subset\rn$, $g\in
L^2(\rn)$ supported in $E$ and $t\in(0,\,\fz)$,
\begin{eqnarray}\label{3.7}
\|T_tg\|_{L^2(F)}\le
C_8\min\lf\{1,\,\frac{t}{\lf[\dist(E,\,F)\r]^{2k}}\r\}^\sz\|g\|_{L^2(E)}.
\end{eqnarray}

\begin{lemma}\label{l3.3}
Let $\mu\in(\omega,\,\pi/2)$, $\psi\in\Psi_{\sz,\,\tau}(S_\mu^0)$
for some $\sz\in(0,\,\fz),\,\tau\in(1,\,\fz)$, and $f\in
H_\fz(S_\mu^0)$. Then the family $\{\psi(tL)f(L)\}_{t>0}$ of
operators satisfies the $k$-Davies-Gaffney estimate of order $\sz$,
\eqref{3.7}, with the positive constant $C_8$ controlled by
$\|f\|_{L^\fz(S_\mu^0)}$.
\end{lemma}

\begin{proof}
For any fixed $\psi\in \Psi_{\sz,\tau}(S_\mu^0)\subset\Psi
(S_\mu^0)$ and $f\in H_\fz(S_\mu^0)$, by \eqref{2.4} and
\eqref{2.5}, we have
\begin{eqnarray}\label{3.8}
\psi(tL)f(L)=\dint_{\Gamma_+}e^{-zL}\eta_+(z)\,dz+\dint_{
\Gamma_-}e^{-zL}\eta_-(z)\,dz,
\end{eqnarray}
where $\Gamma_\pm:=\mathbb{R}^+e^{\pm i (\pi/2-\tz)}$ and for all
$z\in \Gamma_\pm$,
$$\eta_\pm(z)=\frac{1}{2\pi
i}\dint_{\gz_\pm}e^{\xi z}\psi(t\xi) f(\xi)\,d\xi,$$
$\gz_\pm:=\mathbb{R}^+e^{\pm i\nu}$ and
$0\le\omega<\tz<\nu<\mu<\pi/2$. It was proved in \cite[(2.32)]{hmm}
that for all $z\in\Gamma_\pm$,
\begin{eqnarray}\label{3.9}
|\eta_\pm(z)|\ls\frac{\|f\|_{L^\fz(S_\mu^0)}}{t}\min\lf\{1,\,
\lf(\frac{t}{|z|}\r)^{\sz+1}\r\}.
\end{eqnarray}

Thus, by \eqref{3.8} and Minkowski's inequality, we have that for
all $g\in L^2(\rn)$ supported in $E$,
\begin{eqnarray*}
\lf\|\psi(tL)f(L)g\r\|_{L^2(F)}\le\dint_{\Gamma_+}
\lf\|e^{-zL}g\r\|_{L^2(F)}|\eta_+(z)|\,|dz|+\dint_{\Gamma_-}
\lf\|e^{-zL}g\r\|_{L^2(F)}|\eta_-(z)|\,|dz|
=:\mathrm{J}_++\mathrm{J}_-.
\end{eqnarray*}

Since $\pi/2-\tz<\pi/2-\omega$, there exists a positive number
$\ell\in(0,\,1)$ such that $\pi/2-\tz<\ell(\pi/2-\omega)$, which
immediately yields that $S_{\pi/2-\tz}^0\subset
S_{\ell(\pi/2-\omega)}^0$.  Thus, by Lemma \ref{l3.1}, the family
$\{e^{-zL}\}_{z\in S_{\pi/2-\tz}^0}$ satisfies the
$k$-Davies-Gaffney estimate in $z$, which, together with
\eqref{3.9}, implies that
\begin{eqnarray*}
\mathrm{J}_\pm&&\ls\|g\|_{L^2(E)}\dint_{\Gamma_\pm}
\exp\lf\{-\frac{\lf[\dist(E,\,F)\r]^{2k/(2k-1)}}
{C_4|z|^{1/(2k-1)}}\r\}|\eta_\pm(z)|\,|dz|\\
&&\ls\|f\|_{L^\fz(S_\mu^0)}\|g\|_{L^2(E)}\dint_{\Gamma_\pm
}\exp\lf\{-\frac{\lf[\dist(E,\,F)\r]^{2k/(2k-1)}}{C_4|z|^{1/(2k-1)}}\r\}
\min\lf\{1,\,\lf(\frac{t}{|z|}\r)^{\sz+1}\r\}\frac{1}{t}\,|dz|\\
&&\ls\lf[\dint_{ \{z\in\Gamma_\pm:\ |z|\le
t\}}\exp\lf\{-\frac{\lf[\dist(E,\,F)\r]^{2k/(2k-1)}}
{C_4|z|^{1/(2k-1)}}\r\}\min\lf\{1,\,\lf(\frac{t}{|z|}\r)^{\sz+1}\r\}
\,\frac{1}{t}\,|dz|\r.\\
&&\hs\lf.+\dint_{\{z\in\Gamma_\pm:\ |z|> t\}}\cdots\r]
\|f\|_{L^\fz(S_\mu^0)}\|g\|_{L^2(E)}\\
&&=:\lf[\mathrm{O}_1+
\mathrm{O}_2\r] \|f\|_{L^\fz(S_\mu^0)}\|g\|_{L^2(E)}.
\end{eqnarray*}
We estimate $\mathrm{O}_1$ by
\begin{eqnarray*}
\mathrm{O}_1\ls\dint_{\lf\{z\in\Gamma_\pm:\ |z|\le t\r\}}
\exp\lf\{-\frac{\lf[\dist(E,\,F)\r]^{2k/(2k-1)}}{C_4|z|^{1/(2k-1)}}\r\}
\frac{1}{t}\,|dz|\ls\exp\lf\{-\frac{\lf[\dist(E,\,F)\r]^{2k/(2k-1)}}
{C_4|t|^{1/(2k-1)}}\r\}.
\end{eqnarray*}
On the other hand, $\mathrm{O}_2$ can be written into
\begin{eqnarray*}
\mathrm{O}_2\ls\dint_{\{z\in\Gamma_\pm:\ |z|> t\}}
\exp\lf\{-\frac{\lf[\dist(E,\,F)\r]^{2k/(2k-1)}}{C_4|z|^{1/(2k-1)}}\r\}
\lf(\frac{t}{|z|}\r)^{\sz+1}\frac{1}{t}\,|dz|.
\end{eqnarray*}
If $t\ge \lf[\dist(E,\,F)\r]^{2k}$, in this case, we have
\begin{eqnarray*}
\mathrm{O}_2\ls\dint_{\{z\in\Gamma_\pm:\ |z|> t\}}\lf(\frac{t}
{|z|}\r)^{\sz+1}\frac{1}{t}\,|dz|\sim1.
\end{eqnarray*}
If $t<\lf[\dist(E,\,F)\r]^{2k}$, by choosing $N\in[\sz,\,\fz)$, we
obtain
\begin{eqnarray*}
\mathrm{O}_2&&\ls\dint_{\{z\in\Gamma_\pm:\ t<|z|\le
\lf[\dist(E,\,F)\r]^{2k}\}}
\lf(\frac{|z|}{\lf[\dist(E,\,F)\r]^{2k}}\r)^{N}\lf(\frac{t}{|z|}\r)^{\sz+1}
\frac{1}{t}\,|dz|\\
&&\hs+\dint_{\{z\in\Gamma_\pm:\ |z|>
\lf[\dist(E,\,F)\r]^{2k}\}}\lf(\frac{t}{|z|}
\r)^{\sz+1}\frac{1}{t}\,|dz|\\
&&\ls\lf(\frac{t}{\lf[\dist(E,\,F)\r]^{2k}}\r)^N
+\lf(\frac{t}{\lf[\dist(E,\,F)\r]^{2k}}\r)^\sz \sim
\lf(\frac{t}{\lf[\dist(E,\,F)\r]^{2k}}\r)^\sz.
\end{eqnarray*}
Combining the estimates of $\mathrm{O}_1$ and $\mathrm{O}_2$, we
obtain that $\{\psi(tL)f(L)\}_{t>0}$ satisfies the
$k$-Davies-Gaffney estimate of order $\sz$, which completes the
proof of Lemma \ref{l3.3}.
\end{proof}

Now, we turn to some properties of the operators $L_1$ and $L_2$
given in Section \ref{s2}. First, we introduce the definition of the
Legendre transform. Let $h$ be a real valued function defined on
$[0,\,\fz)$. The {\it Legendre transform} $h^\sharp$ of $h$ is
defined by setting, for all $s\in\rr$,
\begin{eqnarray}\label{3.10}
h^\sharp(s):=\sup_{t\ge0}\{st-h(t)\}.
\end{eqnarray}
We have the following proposition about the operator $L_1$.
\begin{proposition}\label{p3.1}
Let $L_1$ be the $2k$-order divergence form homogeneous elliptic
operator defined as in Section \ref{s2}. Then, the semigroup
$\{e^{-tL_1}\}_{t>0}$ satisfies the $k$-Davies-Gaffney estimate.
\end{proposition}

\begin{proof}
We prove Proposition \ref{p3.1} by borrowing some ideas from
\cite{bk1,bk2,da2}. In \cite[Theorem 1.2]{bk2}, letting
$(\boz,\,\mathscr{U},\,\mu,\,d)$ be the usual Euclidean space
$\rn$, endowed with the Lebesgue measure $dx$ and the Euclidean
distant $d$, and with the set class $\mathscr{U}$ being the set of
all Lebesgue measurable sets, and also letting
$$\mathcal{A}:=\lf\{\phi\in C^\fz(\rn)\cap L^\fz(\rn):\
\|D^\az\phi\|_{L^\fz(\rn)}\le1,\,1\le|\az|\le k\r\},$$
$p\equiv
q:=2$, $\az\equiv\bz\equiv\gamma:=0$, $r:= t^{1/(2k)}$, $h(x):=
x^{2k/(2k-1)}$ for all $x\in[0,\,\fz)$ and $\mathcal{R}:=
e^{-tL_1}$, we then obtain the following two equivalent
statements:
\begin{enumerate}
\item[(i)] There exists a positive constant $C(k)$,
depending on $k$, such that for all $\phi\in\mathcal{A}$,
$\rho\in[0,\,\fz)$ and $t\in(0,\fz)$,
\begin{eqnarray}\label{3.11}
\|e^{-\rho\phi} e^{-tL_1}e^{\rho\phi}\|_{\cl(L^2(\rn))}\le
e^{C(k)h^\sharp(\rho t^{1/(2k)})},
\end{eqnarray}
where, by \eqref{3.10},
\begin{eqnarray*}
h^\sharp(\rho t^{1/(2k)})&&:=\sup_{s\ge0}\lf\{\rho
t^{1/(2k)}s-s^{2k/(2k-1)}\r\}=\lf[\frac{(2k-1)^{2k-1}}{(2k)^{2k}}
\r]\rho^{2k}t;
\end{eqnarray*}
\item[(ii)] There exists a positive constant $C_1$ such that
for all the closed sets $E$ and $F$ of $\rn$ and $t\in (0,\fz)$,
\begin{eqnarray}\label{3.12}
\lf\|\chi_E e^{-tL_1}\chi_F\r\|_{\cl(L^2(\rn))}\le \exp{
\lf(-\frac{\dist(E,\,F)}{C_1t^{1/(2k)}}\r)^{2k/(2k-1)}}.
\end{eqnarray}
\end{enumerate}
Recall that in this case, by \cite[Lemma 4]{da2}, $d(E,\,F)$
defined in \cite[(1.4)]{bk2} is equivalent to $\inf_{x\in E,\,y\in
F} |x-y|$ and, moreover, $C(k)$ in \cite[Theorem 1.2(ii)]{bk2} is
assumed to be $1$. However, by the change of variables, we can
easily see that the above equivalent statements are a directly
corollary of the equivalence of (ii) and (iii) of Theorem 1.2 in
\cite{bk2}.

Notice that by the density of the simple functions in $L^2(\rn)$,
\eqref{3.12} is equivalent to the $k$-Davies-Gaffney estimate.
Thus, to prove that the semigroup $\{e^{-tL_1}\}_{t>0}$ satisfies
the $k$-Davies-Gaffney estimate, by the above equivalence of (i)
and (ii), it suffices to prove \eqref{3.11}.

To this end, let $\mathfrak{a}$ be the sesquilinear form as in
\eqref{2.7} associated with $L_1$. Recall that its {\it twisted
form} is defined by setting, for all $\rho\in[0,\,\fz)$, $\phi\in
\mathcal{A}$ and $f,\,g\in W^{k,2}(\rn)$,
$$\mathfrak{a}_{\rho\phi}(f,\,g):=\mathfrak{a}
(e^{\rho\phi}f,\,e^{-\rho\phi}g),$$ which, together with the Leibniz
formula, further yields that there exist positive constants
$C(\az,\gz)$ and $C(\bz,\wz\gz)$ with $|\az|=|\bz|=k$, $0<\gz\le\az$
and $0<\wz\gz\le\bz$,
\begin{eqnarray*}
\mathfrak{a}_{\rho\phi}(f,\,f)&&=\dsum_{|\az|=|\bz|=k}\dint_\rn
a_{\az,\bz}
(x)\pat^\az(e^{\rho\phi}f)(x)\overline{\pat^\bz(e^{-\rho\phi}f)(x)}\,dx\\
&&=\dsum_{|\az|=|\bz|=k}\dint_\rn
a_{\az,\bz}(x)\lf\{\lf[\dsum_{0<\gz\le\az} C(\az,\gz)\pat^\gz
e^{\rho\phi(x)}\pat^{\az-\gz}f(x)+
e^{\rho\phi(x)}\pat^\az f(x)\r]\r.\\
&&\hs\lf.+\lf[\dsum_{0<\wz\gz\le\bz}
C(\bz,\wz\gz)\overline{\pat^{\wz\gz}e^{-\rho\phi(x)}\pat^{\bz-\wz\gz}
f(x)}+\overline{e^{-\rho\phi(x)}\pat^\bz f(x)}\r]\r\}\,dx\\
&&=\dsum_{|\az|=|\bz|=k}\dint_\rn a_{\az,\bz}(x)\lf\{\lf[\dsum_{0<\gz\le\az,
0<\wz\gz\le\bz}C(\az,\gz)C(\bz,\wz\gz)\pat^\gz e^{\rho\phi(x)}
\pat^{\az-\gz}f(x)\overline{\pat^{\wz\gz}e^{-\rho\phi(x)}
\pat^{\bz-\wz\gz}f(x)}\r]\r.\\
&&\hs\lf.+\lf[\overline{e^{-\rho\phi(x)}\pat^\bz f(x)}
\dsum_{0<\gz\le\az}C(\az,\gz)\pat^{\gz}e^{\rho\phi(x)}
\pat^{\az-\gz}f(x)\r]\r.\\
&&\hs\lf.+\lf[\overline{e^{\rho\phi(x)}\pat^\az f(x)}
\dsum_{0<\wz\gz\le\bz}C(\bz,\wz\gz)\pat^{\wz\gz}e^{-\rho\phi(x)}
\pat^{\bz-\wz\gz}f(x)\r]\r\}\,dx+\mathfrak{a}(f,\,f).
\end{eqnarray*}
Let $C(\az,0):=1=: C(\bz,0)$ and $\wz
C(k):=\Lambda\sum_{|\az|=|\bz|=k}
[\sum_{0\le\gz\le\az,0\le\wz\gz\le\bz}C(\az,\gz)C(\bz,\wz\gz)]$,
where $\Lambda$ is as in \eqref{2.8}. By this estimate,
$\phi\in\mathcal{A}$, \eqref{2.8} and H\"older's inequality, we
further have
\begin{eqnarray*}
&&\lf|\mathfrak{a}_{\rho,\phi}(f,\,f)-\mathfrak{a}(f,\,f)\r|\\
&&\hs\le
\Lambda\dsum_{|\az|=|\bz|=k}\dint_{\rn}\lf\{\dsum_{0\le\gz\le\az,
0\le\wz\gz\le\bz}C(\az,\gz)C(\bz,\wz\gz)\lf|\rho^{|\gz|}
\pat^{\az-\gz}f(x)\rho^{|\wz\gz|}\overline{\pat^{\bz-\wz\gz}f(x)}
\r|\r\}\,dx\\
&&\hs\le \Lambda\dsum_{|\az|=|\bz|=k}\dsum_{0\le\gz\le\az,
0\le\wz\gz\le\bz}C(\az,\gz)C(\bz,\wz\gz)
\lf\{\dint_{\rn}\lf|\rho^{|\gz|}\pat^{\az-\gz}f(x)\r|^2\,dx\r\}^{1/2}
\lf\{\dint_{\rn}\lf|\rho^{|\wz\gz|}
\pat^{\bz-\wz\gz}f(x)\r|^2\,dx\r\}^{1/2}\\
&&\hs=: \Lambda\dsum_{|\az|=|\bz|=k}\dsum_{0\le\gz\le\az,
0\le\wz\gz\le\bz}C(\az,\gz)C(\bz,\wz\gz)\,\mathrm{I}_1
\times\mathrm{I}_2.
\end{eqnarray*}
Applying Plancherel's theorem, \eqref{2.9} and Young's inequality
with $\ez\in(0,\,\frac{\lz}{4 \wz C(k)})$, we obtain that there
exists a positive constant $C(\ez)$ such that for all
$\wz\lz\in(C(\ez)\wz C(k),\,\fz)$,
\begin{eqnarray*}
\lf(\mathrm{I}_1\r)^2&&\le\dint_\rn\lf[\rho^{|\gz|}
|\xi|^{k-|\gz|}\lf|\widehat f(\xi)\r|\r]^2\,d\xi\le
\dint_\rn\lf[C(\ez)\rho^{2k}
+\ez|\xi|^{2k}\r]\lf|\widehat f(\xi)\r|^2\,d\xi\\
&&\le C(\ez)\rho^{2k}\|f\|_{L^2(\rn)}^2 +\ez\|\nabla^k
f\|_{L^2(\rn)}^2 \le
C(\ez)\rho^{2k}\|f\|_{L^2(\rn)}^2+\frac{\ez}{\lz}\Re\mathfrak{a}
(f,\,f),
\end{eqnarray*}
which, together with a similar estimate for $\mathrm{I}_2$, shows
that
\begin{eqnarray}\label{3.13}
\lf|\mathfrak{a}_{\rho\phi}(f,\,f)-\mathfrak{a}(f,\,f)\r|\le
\frac{1}{4}\Re\mathfrak{a}(f,\,f)+\wz\lz\rho^{2k}\|f\|_{L^2(\rn)}^2.
\end{eqnarray}
Denote by $L_{\rho \phi}(=e^{-\rho\phi}L_1e^{\rho\phi})$ the
\emph{operator associated with $\mathfrak{a}_{\rho\phi}$}. Let
$f_t:= e^{-tL_{\rho\phi}}f$. By \eqref{3.13}, we have
\begin{eqnarray*}
\frac{d}{dt}\lf\|f_t\r\|_{L^2(\rn)}^2
&&=-(L_{\rho\phi}f_t,\,f_t)-(f_t,\,L_{\rho\phi}f_t)
=-2\Re \mathfrak{a}_{\rho\phi}(f_t,\,f_t)\\
&&=2\lf[\Re\lf(\mathfrak{a}(f_t,\,f_t)-\mathfrak{a}_{\rho\phi}(f_t,\,f_t)\r)-
\Re\mathfrak{a}(f_t,\,f_t)\r]\\
&&\le2|\mathfrak{a}_{\rho\phi}(f_t,\,f_t)-
\mathfrak{a}(f_t,\,f_t)|-2\Re\mathfrak{a}(f_t,\,f_t)\\
&&\le\frac{1}{2}\Re\mathfrak{a}(f_t,\,f_t)+2\wz\lz\rho^{2k}
\|f_t\|_{L^2(\rn)}^2-2\Re\mathfrak{a}(f_t,\,f_t)
\le2\wz\lz\rho^{2k}\|f_t\|_{L^2(\rn)}^2.
\end{eqnarray*}
Thus,
\begin{eqnarray*}
\|f_t\|_{L^2(\rn)}^2=\|e^{-tL_{\rho\phi}}f\|_{L^2(\rn)}^2
\le\exp\{2\wz\lz\rho^{2k}t\}\|f\|_{L^2(\rn)}^2
\le \exp\lf\{2C(k)h^\sharp(\rho t^{1/(2k)})\r\}
\|f\|^2_{L^2(\rn)}.
\end{eqnarray*}
That is, \eqref{3.11} holds. Therefore, $\{e^{-tL_1}\}_{t>0}$
satisfies the $k$-Davies-Gaffney estimate, which completes the proof
of Proposition \ref{p3.1}.
\end{proof}

\begin{remark}\label{r3.1}
In the proof of Proposition \ref{p3.1}, we obtain the estimate
\eqref{3.11} by following the proof of \cite[Lemma 2]{da2}. The same
method should also work for the proof of \cite[Proposition
3.1]{bk1}. Notice that the scaling method mentioned in the last two
lines of \cite[p.\,143]{bk1} may not be valid when used to remove
the factor $e^{(\az w+\ez)t}$ appearing in the proof of
\cite[Proposition 3.1]{bk1}, as the authors claimed therein.
\end{remark}

We also have the following gradient estimate for $L_1$ and $L_2$.

\begin{proposition}\label{p3.2}
Let $k\in\nn$, $L_1$ be the $2k$-order divergence form elliptic
operator and $L_2$ the $2k$-order Schr\"odinger type operator
defined as in Section \ref{s2}. Then,
$\{\sqrt{t}\nabla^ke^{-tL_i}\}_{t>0}$ for $i\in\{1,\,2\}$ satisfies
the $k$-Davies-Gaffney estimate.
\end{proposition}

\begin{proof}
For any Hilbert space $\mathcal{H}$, let
$(\cdot,\,\cdot)_{\mathcal{H}}$ be the inner product of
$\mathcal{H}$. By H\"older's inequality and the fact that
$\{tL_ie^{-tL_i}\}_{t>0}$ and $\{e^{-tL_i}\}_{t>0}$ satisfy the
$k$-Davies-Gaffney estimate which are deduced respectively from
Proposition \ref{p3.1} and Lemma \ref{l3.1}, we conclude that for
all closed sets $E$, $F\subset\rn$, $f\in L^2(\rn)$ supported in
$E$, and $t\in(0,\,\fz)$,
\begin{eqnarray*}
\lf\|\sqrt{t}\nabla^ke^{-tL_i}f\r\|_{L^2(F)}^2
&&\ls\lf|\lf(tL_ie^{-tL_i}f,\,e^{-tL_i}f\r)_{L^2(F)}\r|\\
&&\ls\lf\|tL_ie^{-tL_i}f\r\|_{L^2(F)}\lf\|e^{-tL_i}f\r\|_{L^2(F)}\\
&&\ls\lf(\exp\lf\{-\frac{[\dist(E,\,F)]^{2k/(2k-1)}}{C_1t^{1/(2k-1)}}\r\}\r)^2
\|f\|_{L^2(E)}^2,
\end{eqnarray*}
which implies that $\{\sqrt{t}\nabla^ke^{-tL_i}\}_{t>0}$ also
satisfies the $k$-Davies-Gaffney estimate. This finishes the proof
of Proposition \ref{p3.2}.
\end{proof}

\section{Molecular characterizations of $H_L^p(\rn)$}\label{s4}

Assume that the operator $L$ satisfies the
assumptions (A$_1$), (A$_2$) and (A$_3$) in Section \ref{s2}. In
this section, we introduce the Hardy space $H_L^p(\rn)$ by means of
the $L$-adapted square function and characterize these Hardy spaces
by the molecular decomposition. First, we recall some notions.

Let
$$\Gamma(x):=\lf\{(y,\,t)\in\rn\times(0,\,\fz):\ |x-y|<t\r\}$$
be the cone with vertex $x\in\rn$. For all $f\in L^2(\rn)$ and
$x\in\rn$, the $L$-adapted square function $S_Lf$ is defined by
\begin{eqnarray}\label{4.1}
S_Lf(x):=\lf\{\iint_{\Gamma(x)}|t^{2k}Le^{-t^{2k}L}f(y)|^2
\frac{dy\,dt}{t^{n+1}}\r\}^{1/2}.
\end{eqnarray}

\begin{definition}\label{d4.1}
Let $p\in(0,\,1]$ and $L$ satisfy the assumptions (A$_1$), (A$_2$)
and (A$_3$) in Section \ref{s2}. A function $f\in L^2(\rn)$ is said
to be in $\mathbb{H}_L^p(\rn)$ if $S_Lf\in L^p(\rn)$; moreover,
define $\|f\|_{H_L^p(\rn)}:=\|S_Lf\|_{L^p(\rn)}.$ The {\it Hardy
space} $H_L^p(\rn)$ is then defined to be the completion of
$\mathbb{H}_L^p(\rn)$ with respect to the quasi-norm
$\|\cdot\|_{H_L^p(\rn)}$.
\end{definition}

\begin{remark}\label{r4.1}
Since both the $2k$-order divergence form homogenous elliptic
operator $L_1$ with complex bounded measurable coefficients and the
$2k$-order Schr\"odinger type operator $L_2$ satisfy the assumptions
(A$_1$), (A$_2$) and (A$_3$), we then define the Hardy spaces
$H_{L_1}^p(\rn)$ and $H_{L_2}^p(\rn)$, respectively, associated with
$L_1$ and $L_2$ as in Definition \ref{d4.1}. In particular, when
$k=1$, $H_{L_1}^p(\rn)$ is just the Hardy space
$H_{-\rm{div}(A\nabla)}^p(\rn)$ associated with the second order
divergence form elliptic operator $-\rm{div}(A\nabla)$ with complex
bounded measurable coefficients in \cite{hm09,hm09c,hmm,jy10} and
$H_{L_2}^1(\rn)$ appears in \cite{dz1,dz2,hlmmy,yz}; when $k=2$,
$H_{L_2}^1(\rn)$ was also studied in \cite{cly}.
\end{remark}

In what follows, a {\it cube} always means a closed cube whose sides
are parallel to the coordinate axes. Let $Q\subset\rn$ be a cube
with the side length $l(Q)$. For $i\in\zz_+$, denote by $S_i(Q)$ the
{\it dyadic annuli} based on $Q$, namely,
$S_0(Q):= Q$ and $S_i(Q):={2^i}Q\setminus{(2^{i-1}Q)}$ for $i\in\nn$,
where $2^iQ$ is the cube with the same center as $Q$ and the side
length $2^il(Q)$.

\begin{definition}\label{d4.2}
Let $p\in(0,\,1]$, $\ez\in(0,\,\fz)$, $M\in\nn$ and $L$ satisfy the
assumptions (A$_1$), (A$_2$) and (A$_3$) in Section \ref{s2}. A
function $m\in L^2(\rn)$ is called an $(H_L^p,\,\ez,\,M)$-{\it
molecule} if there exists a cube $Q\subset\rn$ such that
\begin{enumerate}
\item[(i)] for each $\ell\in\{1,\,\cdots,\,M\}$, $m$ belongs to the range of $L^\ell$ in $L^2(\rn)$;
\item[(ii)] for all $i\in\zz_+$ and $\ell\in\{0,\,1,\,\cdots,\,M\}$,
\begin{eqnarray}\label{4.2}
\lf\|\lf([l(Q)]^{-2k}L^{-1}\r)^\ell m\r\|_{L^2(S_i(Q))}
\le[2^il(Q)]^{n(\frac{1}{2}-\frac{1}{p})}2^{-i\ez}.
\end{eqnarray}
\end{enumerate}

Assume that $\{m_j\}_{j=0}^\fz$ is a sequence of
$(H_L^p,\,\ez,\,M)$-molecules and $\{\lz_j\}_{j=0}^\fz\in l^p$. For
any $f\in L^2(\rn)$, if $f=\sum_{j=0}^\fz \lz_jm_j$ in $L^2(\rn)$,
then $\sum_{j=0}^\fz \lz_jm_j$ is called a {\it molecular}
$(H_L^p,\,2,\,\ez,\,M)$-{\it representation} of $f$.
\end{definition}

We now introduce the notion of a molecular Hardy space
$H_{L,\,\mol,\,M}^p(\rn)$ generated by
$(H_L^p,\,\ez,\,M)$-molecules.

\begin{definition}\label{d4.3}
Let $p\in(0,\,1]$, $\ez\in(0,\,\fz)$, $M\in\nn$ and $L$ satisfy the
assumptions (A$_1$), (A$_2$) and (A$_3$) in Section \ref{s2}. The
{\it molecular Hardy space $H_{L,\,\mol,\,M}^p(\rn)$} is defined to
be the completion of the space
$$\mathbb{H}_{L,\,\mol,\,M}^p(\rn):=\{f:\ f\ \text{ has a molecular}\
(H_L^p,\,2,\,\ez,\,M)-\text{representation}\}$$  with respect to
the \emph{quasi-norm}
\begin{eqnarray*}
\|f\|_{H_{L,\,\mol,\,M}^p(\rn)} :=&&\inf\lf\{\lf(\dsum_{j=0
}^\fz|\lz_j|^p\r)^{1/p}:\
f=\dsum_{j=0}^\fz\lz_jm_j \ \text{is a molecular}\r.\\
&&\hspace{2.8cm}(H_L^p,\, 2,\,\ez,\,M)-\text{representation}\Bigg\},
\end{eqnarray*}
where the infimum is taken over all the molecular $(H_L^p,\,2,\,
\ez,\,M)$-representations of $f$ as above.
\end{definition}

Now, we establish the molecular characterization of the Hardy space
$H_L^p(\rn)$.

\begin{theorem}\label{t4.1}
Let $p\in(0,\,1]$, $\ez\in(0,\,\fz)$, $L$ satisfy the assumptions
{\rm(A$_1)$}, {\rm(A$_2)$} and {\rm(A$_3)$} in Section \ref{s2} and
$M\in\nn$ such that $M>\frac{n}{2k}(\frac{1}{p}-\frac{1}{2})$. Then,
$H_L^p(\rn)= H_{L,\,\mol,\,M}^p(\rn)$ with equivalent norms.
\end{theorem}

To prove Theorem \ref{t4.1}, by Definitions \ref{d4.1} and
\ref{d4.3}, it suffices to prove that
\begin{eqnarray*}
\mathbb{H}_L^p(\rn)=\mathbb{H}_{L,\,\mol,\,M}^p(\rn),\ \
M>\frac{n}{2k}\lf(\frac{1}{p}-\frac{1}{2}\r),
\end{eqnarray*}
with equivalent norms. We divide the proof into two parts: (i)
$\mathbb{H}_{L,\,\mol,\,M}^p(\rn)\subset \mathbb{H}_L^p(\rn)$; (ii)
$\mathbb{H}_L^p(\rn)\subset \mathbb{H}_{L,\,\mol,\,M}^p(\rn)$.

To prove the inclusion $\mathbb{H}_{L,\,\mol,\,M}^p(\rn)\subset
\mathbb{H}_L^p(\rn)$, we need the following key lemma which is just
\cite[Lemma 3.8]{hmm}. Recall that a {\it nonnegative sublinear
operator} $T$ means that $T$ is sublinear and $Tf\ge 0$ for all $f$
in the domain of $T$.

\begin{lemma}\label{l4.1}
Let $p\in(0,\,1]$, $M\in\nn$ and $T$ be a linear operator, or a
nonnegative sublinear operator, which is of weak type $(2,\,2)$,
that is, there exists a positive constant $C$ such that for all
$\eta\in(0,\,\fz)$ and $f\in L^2(\rn)$,
$$\lf|\{x\in\rn:\ |Tf(x)|>\eta\}\r|\le C\eta^{-2}\|f\|_{L^2(\rn)
}^2.$$ Assume that there exists a positive constant $C$ such that
for all $(H_L^p,\,\ez,\,M)$-molecules $m$, $\|Tm\|_{L^p(\rn)}\le C$.
Then the operator $T$ is bounded from $H_{L,\,\mol,\,M}^p(\rn)$ to
$L^p(\rn)$.
\end{lemma}

\begin{proof}[Proof of Theorem \ref{t4.1}: the inclusion
$\mathbb{H}_{L,\,\mol,\,M}^p(\rn) \subset\mathbb{H}_L^p(\rn)$.]
Recall that $L$ is a one-to-one operator of type $\omega$ having a
bounded $H_\fz$ functional calculus. For all $x\in\rn$,
$\psi\in\Psi(S_\mu^0)$ defined as in Section \ref{s2}, set
$\psi_t(x):= \psi(tx)$ for all $t\in (0,\fz)$. The {\it quadratic
norm} $\|g\|_{T,\psi}$, associated with the operator $L$ in
$L^2(\rn)$ and $\psi$, is defined by
$$\|g\|_{T,\psi}:=\lf\{\dint_0^\fz\|\psi_t(T)g\|^2_{L^2(\rn)}\,
\frac{dt}{t}\r\}^{1/2}$$ for all $g\in [L^2(\rn)]_{T,\psi}$ which is
a {\it subspace of $L^2(\rn)$ such that the above integral is
finite}. Since $L$ has a bounded $H_\fz$ functional calculus on
$L^2(\rn)$, it follows from \cite{adm1} that
$[L^2(\rn)]_{T,\psi}=L^2(\rn)$ and for all $g\in L^2(\rn)$,
\begin{eqnarray}\label{4.3}
\|g\|_{T,\psi}\ls\|g\|_{L^2(\rn)}.
\end{eqnarray}
By Fubini's theorem, we have that $\|S_L
g\|_{L^2(\rn)}\sim\|g\|_{T,\psi_0}$, where $\psi_0(z):=
ze^{-z}\in\Psi(S_{\mu}^0)$ for all $\mu\in(0,\,\pi/2)$. Thus, $S_L$
is bounded on $L^2(\rn)$. By Lemma \ref{l4.1}, to prove the
inclusion
$\mathbb{H}_{L,\,\mol,\,M}^p(\rn)\subset\mathbb{H}_L^p(\rn)$, it
suffices to prove that for all $(H_L^p,\,\ez,\,M)$-molecules $m$
with $M>\frac{n}{2k}(\frac{1}{p}-\frac{1}{2})$,
\begin{eqnarray}\label{4.4}
\|S_Lm\|_{L^p(\rn)}\ls1.
\end{eqnarray}
Let $Q$ be the cube associated with $m$ as in Definition
\ref{d4.2}. Let $j_0\in\nn$ be such that $2^{j_0-1}<\sqrt n\le
2^{j_0}$. By Minkowski's inequality, H\"older's inequality and the
$L^2(\rn)$-boundedness of $S_L$, we see that
\begin{eqnarray*}
\lf\|S_Lm\r\|_{L^p(\rn)}&&\le\|S_Lm\|_{L^p(2^{j_0+4}Q)}+
\dsum_{j=j_0+5}^\fz\lf\|S_Lm\r\|_{L^p(S_j(Q))}\\
&&\ls\|m\|_{L^2(\rn)}|2^{j_0+4}Q|^{\frac{1}{p}-\frac{1}{2}}
+\dsum_{j=j_0+5}^\fz\|S_Lm\|_{L^2(S_j(Q))}
|S_j(Q)|^{\frac{1}{p}-\frac{1}{2}}.
\end{eqnarray*}
For $\|m\|_{L^2(\rn)}$, from Minkowski's inequality and the size
condition \eqref{4.2} of $m$, it follows that
\begin{eqnarray}\label{4.5}
\|m\|_{L^2(\rn)}&&\le\dsum_{j=0}^\fz\|m\|_{L^2(S_j(Q))}
\le\dsum_{j=0}^\fz[2^jl(Q)]^{n(\frac{1}{2} -\frac{1}{p})}2^{-j\ez}
\ls[l(Q)]^{n(\frac{1}{2}-\frac{1}{p})}.
\end{eqnarray}
For $j\in\{j_0+5,\cdots\}$, let
$\mathrm{I}_j:=\|S_Lm\|_{L^2(S_j(Q))}$. Then,
\begin{eqnarray*}
(\mathrm{I}_j)^2
&&=\dint_{S_j(Q)}|S_Lm|^2\,dx=\dint_{S_j(Q)}\dint_0^\fz
\dint_{|y-x|<t}|t^{2k}Le^{-t^{2k}L}m(y)|^2\frac{dy\,dt}{t^{n+1}}\,dx\\
&&=\dint_{S_j(Q)}\dint_0^{2^{\tz(j-5)}l(Q)}
\dint_{|y-x|<t}|t^{2k}Le^{-t^{2k}L}m(y)|^2\frac{dy\,dt}{t^{n+1}}\,dx
+ \dint_{S_j(Q)}\dint_{2^{\tz(j-5)}l(Q)}^\fz
\dint_{|y-x|<t}\cdots\\
&&=:\mathrm{B}_j+\mathrm{D}_j,
\end{eqnarray*}
where $\tz\in(0,\,1)$ is determined later.

For $\mathrm{D}_j$, let $b:= L^{-M}m$. By Fubini's theorem, Lemma
\ref{l3.1} and the size condition \eqref{4.2} of $m$, we conclude
that
\begin{eqnarray*}
\mathrm{D}_j&&=\dint_{S_j(Q)}\dint_{2^{\tz(j-5)}l(Q)}^\fz
\dint_{|y-x|<t}\lf|t^{2k}Le^{-t^{2k}L}L^Mb(y)\r|^2\frac{dy\,dt}{t^{n+1}}\,dx\\
&&=\dint_{S_j(Q)}\dint_{2^{\tz(j-5)}l(Q)}^\fz
\dint_{|y-x|<t}\lf|t^{2k(M+1)}L^{M+1}e^{-t^{2k}L}b(y)\r|^2
\frac{dy\,dt}{t^{n+1+4kM}}\,dx\\
&&\ls\dint_{2^{\tz(j-5)}l(Q)}^\fz\lf\|t^{2k(M+1)}L^{M+1}e^{-t^{2k}L}b
\r\|_{L^2(\rn)}^2
\frac{dt}{t^{4kM+1}}\ls\dint_{2^{\tz(j-5)}l(Q)}^\fz\|b\|_{L^2(\rn)}^2
\frac{dt}{t^{4kM+1}}\\
&&\sim\|b\|_{L^2(\rn)}^2\lf[2^{\tz(j-5)}l(Q)\r]^{-4kM}
\sim\lf[\dsum_{i=0}^\fz\|b\|^2_{L^2(S_i(Q))}\r]
\lf[2^{\tz(j-5)l(Q)}\r]^{-4kM}\\
&&\ls[l(Q)]^{4kM+2n(\frac{1}{2}-\frac{1}{p})}
\lf[2^{\tz(j-5)}l(Q)\r]^{-4kM}
\sim2^{-j[4kM\tz+n(1-\frac{2}{p})]}\lf[2^jl(Q)\r]^{n(1-\frac{2}{p})}.
\end{eqnarray*}
Recall that $M>\frac{n}{2k}(1/p-1/2)$.  Letting $\tz$ be
sufficiently close to $1$ such that $\az_0:=2kM\tz-n(1/p-1/2)>0$, we
then obtain
\begin{eqnarray}\label{4.6}
\mathrm{D}_j\ls2^{-2j\az_0}\lf|S_j(Q)\r|^{1-\frac{2}{p}}.
\end{eqnarray}
To estimate $\mathrm{B}_j$, let $\wz S_j(Q):=2^{j+j_0+1}Q\setminus
(2^{j-j_0-2}Q)$ and $\widehat
S_j(Q):=2^{j+j_0+2}Q\setminus(2^{j-j_0-3}Q)$. By Fubini's theorem,
we see that
\begin{eqnarray*}
\mathrm{B}_j &&\ls\dint_0^{2^{\tz(j-5)}l(Q)}
\dint_{\wz S_j(Q)}\lf|t^{2k}Le^{-t^{2k}L}m(y)\r|^2\frac{dy\,dt}{t}\\
&&\ls \dint_0^{2^{\tz(j-5)}l(Q)}\dint_{\wz
S_j(Q)}\lf|t^{2k}Le^{-t^{2k}L}
\lf(\chi_{2^{j-j_0-3}Q}m\r)(y)\r|^2\frac{dy\,dt}{t}\\
&&\hs+ \dint_0^{2^{\tz(j-5)}l(Q)}\dint_{\wz
S_j(Q)}\lf|t^{2k}Le^{-t^{2k}L}
\lf(\chi_{\widehat S_j(Q)}m\r)(y)\r|^2\frac{dy\,dt}{t}\\
&&\hs+ \dint_0^{2^{\tz(j-5)}l(Q)}\dint_{\wz
S_j(Q)}\lf|t^{2k}Le^{-t^{2k}L} \lf(\chi_{\rn\setminus
2^{j+j_0+2}Q}m\r)(y)\r|^2\frac{dy\,dt}{t}\\
&&=:\mathrm{B}_{j,1}+\mathrm{B}_{j,2}+\mathrm{B}_{j,3}.
\end{eqnarray*}
From the $k$-Davies-Gaffney estimate, \eqref{4.5} and choosing
$\az\in(2n(1/p-1/2)/(1-\tz),\,\fz)$, we infer that
\begin{eqnarray*}
\mathrm{B}_{j,1}+\mathrm{B}_{j,3}&&\ls\dint_0^{2^{\tz(j-5)}l(Q)}
\exp\lf\{-\wz
C\lf[\frac{2^jl(Q)}{t}\r]^{2k/(2k-1)}\r\}\|m\|_{L^2(\rn)}^2
\,\frac{dt}{t}\\
&&\ls\|m\|_{L^2(\rn)}^2\dint_0^{2^{\tz(j-5)}l(Q)}\lf[\frac{t}{2^jl(Q)}
\r]^\az\,\frac{dt}{t}\sim[l(Q)]^{2n(\frac{1}{2}-\frac{1}{p})}
\lf[2^{\tz(j-5)-j}\r]^\az,
\end{eqnarray*}
where $\wz C$ denotes a positive constant. Let
$\az_1:=(1-\tz)\az/2-n(1/p-1/2)$. Then $\az_1\in(0,\fz)$ and we
have
\begin{eqnarray}\label{4.7}
\mathrm{B}_{j,1}+\mathrm{B}_{j,3}\ls\lf[2^jl(Q)\r]^{2n(1/2-1/p)}
2^{-2j\az_1}.
\end{eqnarray}
Finally, by \eqref{4.3} and the size condition \eqref{4.2} of $m$,
we conclude that
\begin{eqnarray*}
\mathrm{B}_{j,2}\ls\|m\|_{L^2(\widehat S_j(Q))}^2\ls
\dsum_{\ell=j-j_0-2}^{j+j_0+2}\|m\|^2_{L^2(S_\ell(Q))}\ls2^{-2j\ez}
\lf[2^jl(Q)\r]^{2n(1/2-1/p)},
\end{eqnarray*}
which, together with \eqref{4.6} and \eqref{4.7}, shows that there
exists a positive constant $\az_2:=\min\{\az_0,\,\az_1,\,\ez\}$ such
that for all $j\in\{j_0+5,\,\cdots\}$,
\begin{eqnarray}\label{4.8}
\mathrm{I}_j\ls[2^jl(Q)]^{n(1/2-1/p)}2^{-j\az_2}\sim|S_j(Q)|^{1/2-1/p}2^{-j\az_2}.
\end{eqnarray}
Combining \eqref{4.5} and \eqref{4.8}, we see that
\begin{eqnarray*}
\|S_Lm\|_{L^p(\rn)}\ls\lf[l(Q)\r]^{n(\frac{1}{2}-\frac{1}{p})}
|2^{j_0+4}Q|^{\frac{1}{p} -\frac{1}{2}}+\dsum_{j=j_0+5}^\fz
2^{-j\az_2}\ls1,
\end{eqnarray*}
from which we deduce \eqref{4.4}. Thus, the inclusion
$\mathbb{H}_{L,\,\mol,\,M}^p(\rn)\subset\mathbb{H}_L^p(\rn)$ holds,
which completes the proof of part one of Theorem \ref{t4.1}.
\end{proof}

Now, we prove the inclusion $\mathbb{H}_L^p(\rn)\subset
\mathbb{H}_{L,\,\mol,\,M}^p(\rn)$. To this end, we need to use some
results concerning the tent space from \cite{cms}. Let $F$ be a
function on $\mathbb{R}^{n+1}_+:=\rn\times(0,\fz)$. The
$\mathcal{A}$-{\it functional} of $F$ is defined by setting, for all
$x\in\rn$,
$$\mathcal{A}(F)(x):=\lf\{\iint_{\Gamma(x)}\lf|F(y,\,t)\r|^2
\frac{dy\,dt}{t^{n+1}}\r\}^{\frac{1}{2}}.$$ For $p\in(0,\,\fz)$, the
{\it tent space} $T^p(\mathbb{R}^{n+1}_+)$ is defined by
$$T^p(\mathbb{R}^{n+1}_+):=\lf\{F:\ \mathbb{R}^{n+1}_+\rightarrow
\cc: \ \|F\|_{T^p(\mathbb{R}^{n+1}_+)}:=\|\mathcal{A}(F)
\|_{L^p(\rn)}<\fz \r\}.$$ For any cube $Q$, denote by $R_Q:=
Q\times(0,\,l(Q))$ the {\it Carleson box of Q}. A measurable
function $A$ on $\mathbb{R}^{n+1}_+$ is called a
$T^p(\mathbb{R}_+^{n+1})$-{\it atom associated with $Q$} with
$p\in(0,\,1]$, if $A$ satisfies the following properties:
\begin{eqnarray}\label{4.9}
\supp A\subset R_Q
\end{eqnarray}
and
\begin{eqnarray}\label{4.10}
\lf\{\iint_{R_Q}|A(x,\,t)|^2\frac{dx\,dt}{t}\r\}^{1/2}\le
\lf|Q\r|^{\frac{1}{2}-\frac{1}{p}}.
\end{eqnarray}

For the tent space $T^p(\mathbb{R}^{n+1}_+)$ with $p\in(0,\,1]$, we
have the following atomic decomposition from \cite{cms} (see also
\cite[Proposition 3.25]{hmm}).

\begin{theorem}[\cite{cms}]\label{t4.2}
Let $p\in(0,\,1]$. For all $F\in T^p(\mathbb{R}^{n+1}_+)$, there
exist a numerical sequence $\{\lz_j\}_{j=0}^\fz$ and a sequence
$\{A_j\}_{j=0}^\fz$ of $T^p(\mathbb{R}^{n+1}_+)$-atoms such that for
almost every $(x,\,t)\in \mathbb{R}^{n+1}_+$,
$$F(x,\,t)=\dsum_{j=0}^\fz\lz_jA_j(x,\,t).$$
Moreover,
$$\dsum_{j=0}^\fz|\lz_j|^p\sim\|F\|_{T^p(\mathbb{R}^{n+1}_+)}^p,$$
where the implicit equivalent positive constants depend only on $n$.
Finally, if $F\in T^p(\mathbb{R}^{n+1}_+)\cap
T^2(\mathbb{R}^{n+1}_+),$ then the decomposition also converges in
$T^2(\mathbb{R}^{n+1}_+)$.
\end{theorem}

Let $M\in\nn$. For all $F\in T^2(\mathbb{R}^{n+1}_+)$, define the
operator $\pi_{M,L}$ by setting, for all $x\in\rn$,
\begin{equation}\label{4.11}
\pi_{M,L}F(x):=\dint_0^\fz\lf(t^{2k}L\r)^{M+1}e^{-t^{2k}L}
F(x,\,t)\,\frac{dt}{t}.
\end{equation}
For this operator, we have the following useful properties.

\begin{lemma}\label{l4.2}
Let $M\in\nn$, $p\in(0,\,1]$, $\ez\in(0,\,\fz)$ and the operator $L$
satisfy the assumptions {\rm(A$_1)$}, {\rm(A$_2)$} and {\rm(A$_3)$}
in Section \ref{s2}. Let $\pi_{M,\,L}$ be as in \eqref{4.11}. Then
\begin{enumerate}
\item[\rm(i)] The operator $\pi_{M,\,L}$ is bounded from
$T^2(\mathbb{R}^{n+1}_+)$ to $L^2(\rn)$.
\item[\rm(ii)] For any
$T^p(\mathbb{R}^{n+1}_+)$-atom $A$, $\pi_{M,\,L}A$ is an
$(H_L^p,\,\ez,\,M)$-molecule up to a harmless positive constant
multiple. \item[\rm(iii)] If $M\in(n(1/p-1/2)/(2k),\,\fz)$, then
the operator $\pi_{M,L}$ is bounded from the tent space
$T^p(\rr^{n+1}_+)$ to the molecular Hardy space
$H_{L,\,\mol,\,M}^p(\rn)$.
\end{enumerate}
\end{lemma}

\begin{proof} We first show (i).
Let $L^*$ be the \emph{adjoint operator} of $L$ in $L^2(\rn)$.
Observe that $L^*$ also satisfies the assumptions (A$_1$), (A$_2$)
and (A$_3$) in Section \ref{s2}. By Fubini's theorem, H\"older's
inequality and the quadratic estimate \eqref{4.3} with $L$
replaced by $L^*$, we see that for all $F\in
T^2(\mathbb{R}^{n+1}_+)$ and $g\in L^2(\rn)$,
\begin{eqnarray*}
|(\pi_{M,L}F,\,g)| &&=\lf|\dint_\rn\dint_0^{\fz}\lf(t^{2k}L\r)^{M+1}
e^{-t^{2k}L}F(x,\,t)\overline{g(x)}\,\frac{dt}{t}\,dx\r|\\
&&=\lf|\dint_0^{\fz}\dint_\rn
F(x,\,t)\overline{\lf(t^{2k}L^*\r)^{M+1}
e^{-t^{2k}L^*}g(x)}\,dx\,\frac{dt}{t}\r|\\
&&\ls\lf\{\dint_0^{\fz}\dint_\rn
|F(x,\,t)|^2\,dx\,\frac{dt}{t}\r\}^{1/2}\lf\{\dint_0^{\fz}\dint_\rn\lf|\lf(t^{2k}L^* \r)^{M+1}
e^{-t^{2k}L^*}g(x)\r|^2\,dx
\,\frac{dt}{t}\r\}^{1/2}\\
&&\ls\|F\|_{T^2(\mathbb{R}^{n+1}_+)}\|g\|_{L^2(\rn)},
\end{eqnarray*}
which further implies that the operator $\pi_{M,L}$ is bounded from
$T^2(\mathbb{R}^{n+1}_+)$ to $L^2(\rn)$. Thus, (i) holds.

To prove (ii), let $A$ be a $T^p(\mathbb{R}^{n+1}_+)$-atom $A$
associated with the cube $Q$. From \eqref{4.11}, it follows that for
all $\ell\in\{0,\,\cdots,\,M\}$ and $x\in\rn$,
\begin{eqnarray}\label{4.12}
\pi_{M,L}A(x)=\dint_0^\fz\lf(t^{2k}L\r)^{M+1}e^{-t^{2k}L}
A(x,\,t)\,\frac{dt}{t}=L^\ell\dint_0^\fz t^{2k(M+1)}L^{M+1-\ell}
e^{-t^{2k}L}
A(x,\,t)\,\frac{dt}{t}.
\end{eqnarray}
Observe that
$$\dint_0^\fz
t^{2k(M+1)}L^{M+1-\ell}e^{-t^{2k}L} A(x,\,t)\,\frac{dt}{t}
=\dint_0^\fz
t^{2k(M+1)}\lf(L^{1-\frac\ell{M+1}}\r)^{M+1}e^{-t^{2k}L}
A(x,\,t)\,\frac{dt}{t},$$ which belongs to $L^2(\rn)$ via a dual
argument similar to that used in the proof of (i). This, combined
with \eqref{4.12}, implies that $\pi_{M,L}(A)$ satisfies Definition
\ref{d4.2}(i).

For all $x\in\rn$, letting
$$b(x):= \dint_0^\fz t^{2k(M+1)}Le^{-t^{2k}L}
A(x,\,t)\,\frac{dt}{t},$$ we then have $\pi_{M,L}A(x)=L^M b(x).$
For all $g\in L^2(\rn)$, from H\"older's inequality, \eqref{4.9},
\eqref{4.10} and the quadratic estimate \eqref{4.3} with $L$
replaced by $L^*$, we deduce that
\begin{eqnarray}\label{4.13}
&&\lf|\dint_\rn\lf([l(Q)]^{2k}L\r)^\ell b(x)\overline{g(x)}\,dx\r|
\nonumber\\
\nonumber &&\hs=\lf|\dint_\rn \lf\{\dint_0^{\fz}\lf[l(Q)\r]^{2k\ell}
L^{\ell+1} t^{2k(M+1)}e^{-t^{2k}L}A(x,\,t)\,\frac{dt}{t}
\r\}\overline{g(x)}\,dx\r|\\ \nonumber
&&\hs=\lf[l(Q)\r]^{2k\ell}\lf|\iint_{R_Q}A(x,\,t)\overline{(L^*)^{\ell+1}
t^{2k(M+1)}e^{-t^{2k}L^*}g(x)}\,\frac{dx\,dt}{t}\r|\\ \nonumber
&&\hs\ls[l(Q)]^{2kl}\lf[\iint_{R_Q}|A(x,\,t)|^2\,\frac{dx\,dt}{t}
\r]^{1/2}\lf[\iint_{R_Q}\lf|\overline{(L^*)^{\ell+1}t^{2k(M+1)}
e^{-t^{2k}L^*}
g(x)}\r|^2\,\frac{dx\,dt}{t}\r]^{1/2}\nonumber\\ \nonumber
&&\hs\ls[l(Q)]^{n(1/2-1/p)+2kM} \lf\{\iint_{R_Q}
\lf|(t^{2k}L^*)^{\ell+1}e^{-t^{2k}L^*}
g(x)\r|^2\frac{dx\,dt}{t}\r\}^{1/2}\\
&&\hs\ls[l(Q)]^{n(1/2-1/p)+2kM}\|g\|_{L^2(\rn)},
\end{eqnarray}
which further implies that for all $\ell\in\{0,\,\cdots,\,M\}$,
\begin{eqnarray*}
\lf\|\lf([l(Q)]^{2k}L\r)^\ell b\r\|_{L^2(2Q)}\ls[l(Q)]^{2kM}
|Q|^{\frac{1}{2}-\frac{1}{p}}.
\end{eqnarray*}
Thus, by this, we conclude that for all $\wz
\ell\in\{0,\,\cdots,\,M\}$ and $j\in\{0,\,1\}$,
\begin{eqnarray}\label{4.14}
\lf\|\lf([l(Q)]^{-2k}L^{-1}\r)^{\wz\ell}(\pi_{M,\,L}A)\r\|_{L^2(S_j(Q))}
&&=\lf\|\lf([l(Q)]^{-2k}L^{-1}\r)^{\wz\ell}L^Mb\r\|_{L^2(S_j(Q))}\noz\\
&&\ls \lf\|\lf([l(Q)]^{2k(M-\wz\ell)}L^{M-\wz\ell}\r)
b\r\|_{L^2(2(Q))}|l(Q)|^{-2kM}\noz\\
&&\ls[l(Q)]^{n(1/2-1/p)},
\end{eqnarray}
which is desired.

Moreover, for all $\wz\ell\in\{0,\,\cdots,\,M\}$ and
$j\in\{2,\,3,\,\cdots\}$, letting $g\in L^2(\rn)$ with $\supp
g\subset S_j(Q)$, choosing $\az\in(n(1/p-1/2)(2-1/k),\,\fz)$ and
using Lemma \ref{l3.1}, similar to the estimate for \eqref{4.13},
we see that
\begin{eqnarray*}
&&\lf|\dint_{\rn}\lf([l(Q)]^{2k}L\r)^{\ell}b(x)\overline{g(x)}\,dx\r|\\
&&\hs\ls[l(Q)]^{n(1/2-1/p)+2kM}\lf\{\dint_0^{l(Q)}\dint_Q\lf|
(t^{2k}L^*)^{\ell+1}e^{-t^{2k}L^*}g(x)
\r|^2\frac{dx\,dt}{t}\r\}^{1/2}\\
&&\hs\ls[l(Q)]^{n(1/2-1/p)+2kM}\lf[\dint_0^{l(Q)}
\exp\lf\{-C\lf[\frac{\dist(Q,\,S_j(Q))}
{t}\r]^{2k/(2k-1)}\r\}\,\frac{dt}{t}\r]^{1/2}\|g\|_{L^2(S_j(Q))}\\
&&\hs\ls[l(Q)]^{n(1/2-1/p)+2kM}2^{-jk\az/(2k-1)}
\|g\|_{L^2(S_j(Q))},
\end{eqnarray*}
which further implies that for all $\wz\ell\in\{0,\,\cdots,\,M\}$
and $j\in\{2,\,3,\,\cdots\}$,
\begin{eqnarray}\label{4.15}
\lf\|\lf([l(Q)]^{-2k}L^{-1}\r)^{\wz\ell}(\pi_{M,\,L}A)\r\|_{L^2(S_j(Q))}
&&\le\lf\|\lf([l(Q)]^{-2k}L^{-1}\r)^{\wz\ell}L^Mb\r\|_{L^2(S_j(Q))}\noz\\
\nonumber &&\ls \lf\|\lf([l(Q)]^{2k(M-\wz\ell)}L^{M-\wz\ell}\r)
b\r\|_{L^2(S_j(Q))}|l(Q)|^{-2kM}\\
&&\ls[2^jl(Q)]^{n(1/2-1/p)}2^{-j\lf[k\az/(2k-1)-n(1/p-1/2)\r]}\noz\\
&&\sim[2^jl(Q)]^{n(1/2-1/p)}2^{-j\ez},
\end{eqnarray}
where $\ez:= k\az/(2k-1)-n(1/p-1/2)\in (0,\fz)$.

Combining \eqref{4.14} and \eqref{4.15}, we know that $\pi_{M,L}A$
satisfies Definition \ref{d4.2}(ii) up to a harmless positive
constant multiple. Thus, $\pi_{M,L}A$ is an
$(H_L^p,\,\ez,\,M)$-molecule up to a harmless positive constant
multiple, which completes the proof of (ii).

To show (iii), by density, we only need to show that for all $F\in
T^p(\rr^{n+1}_+)\cap T^2(\rr^{n+1}_+)$,
\begin{eqnarray*}
\lf\|\pi_{M,L}F\r\|_{H_{L,\,\mol,\,M}^p(\rn)}\ls\lf\|F\r\|_{
T^p(\rr^{n+1}_+)}.
\end{eqnarray*}
To this end, by Theorem \ref{t4.2}, there exist a sequence
$\{A_i\}_{i=0}^\fz$ of $T^p(\rr^{n+1}_+)$-atoms and
$\{\lz_i\}_{i=0}^\fz\in l^p$ such that $F=\sum_{i=0}^\fz\lz_iA_i$ in
both pointwise  and $T^2(\rr^{n+1}_+)$, and
\begin{eqnarray*}
\lf(\dsum_{i=0}^\fz\lf|\lz_i\r|^p\r)^{1/p}
\sim\lf\|F\r\|_{T^p(\rr^{n+1}_+)}.
\end{eqnarray*}
By (i) of this lemma, we know that
\begin{eqnarray*}
\pi_{M,L}F=\dsum_{i=0}^\fz\lz_i\pi_{M,L}A_i
\end{eqnarray*}
in $L^2(\rn)$, which, combining (ii) of Lemma \ref{l4.2}, shows that
$\sum_{i=0}^\fz\lz_i\pi_{M,L}A_i$ is a molecular
$(H_L^p,\,2,\,\ez,\,M)$-representation of $\pi_{M,L}F$. Thus,
$\pi_{M,L}F\in H_{L,\,\mol,\,M}^p(\rn)$ and
\begin{eqnarray*}
\lf\|\pi_{M,L}F\r\|_{H_{L,\,\mol,\,M}^p(\rn)}
\ls\lf\{\dsum_{i=0}^\fz\lf|\lz_i\r|^p\r\}^{1/p}
\sim\lf\|F\r\|_{T^p(\rr^{n+1}_+)},
\end{eqnarray*}
which completes the proof of (iii) and hence Lemma \ref{l4.2}.
\end{proof}

\begin{proof}[Proof of Theorem \ref{t4.1}: the inclusion
$\mathbb{H}_L^p(\rn)\subset\mathbb{H}_{L,\,\mol,\,M}^p(\rn)$.] For
all $f\in\mathbb{H}_L^p(\rn)$, $t\in(0,\,\fz)$ and $x\in\rn$, let
\begin{eqnarray}\label{4.16}
F(x,\,t):= t^{2k}Le^{-t^{2k}L} f(x).
\end{eqnarray}
By $S_Lf\in L^p(\rn)$ and $f\in L^2(\rn)$, together with the fact
that $S_L$ is bounded on $L^2(\rn)$, we know that $F\in
T^p(\mathbb{R}^{n+1}_{+})\cap T^2 (\mathbb{R}^{n+1}_{+})$.
Moreover, by the $H_\fz$ functional calculus in $L^2(\rn)$, we
see the following Calder\'on reproducing formula that for all
$g\in L^2(\rn)$,
\begin{eqnarray*}
g=C_9\dint_0^\fz(t^{2k} L)^{M+2}e^{-2t^{2k}L}g\,\frac{dt}{t},
\end{eqnarray*}
where $C_9$ is a positive constant such that $C_9 \int_0^\fz
t^{2k(M+2)}e^{-2t^{2k}}\,\frac{dt}{t}=1$. Thus, for all $f\in
\mathbb{H}_L^p(\rn)$, if letting $F$ be as in \eqref{4.16}, then
$f=C_9\pi_{M,L}F$ and, by Lemma \ref{l4.2}(iii) and its proof, we
further know that $f\in \mathbb{H}_{L,\,\mol,\,M}^p(\rn)$ and
$\|f\|_{H_{L,\,\mol,\,M}}^p(\rn) \ls\|f\|_{H_{L}^p(\rn)}$.
Therefore, $\mathbb{H}_L^p(\rn)\subset
\mathbb{H}_{L,\,\mol,\,M}^p(\rn)$, which completes the proof of
Theorem \ref{t4.1}.
\end{proof}

\section{Generalized square function
characterizations of $H_L^p(\rn)$}\label{s5}

This section is devoted to the generalized square function
characterization of $H_L^p(\rn)$. We first introduce the notion of
the Hardy space $H_{\psi,L}^p(\rn)$ defined via the generalized
square function. Let $\omega\in[0,\,\pi/2)$, $\az\in(0,\,\fz)$,
$\bz\in(n(1/p-1/2)/(2k),\,\fz)$ and
$\psi\in\Psi_{\az,\bz}(S_{\mu}^0)$ with $\mu\in(\omega,\,\pi/2)$.
For all $f\in L^2(\rn)$ and $(x,\,t)\in \mathbb{R}^{n+1}_{+}$,
define the \emph{operator $Q_{\psi,L}f$} by,
\begin{eqnarray}\label{5.1}
Q_{\psi,L}f(x,\,t):=\psi(t^{2k}L)f(x).
\end{eqnarray}

\begin{definition}\label{d5.1}
Let $p\in(0,\,1]$, $\omega\in[0,\,\pi/2)$,  $L$ be the operator of
type $\omega$ satisfying the assumptions (A$_1$), (A$_2$) and
(A$_3$) in Section \ref{s2}, $\az\in(0,\,\fz)$,
$\bz\in(n(1/p-1/2)/(2k),\,\fz)$, $\mu\in(\omega,\,\pi/2)$ and
$\psi\in\Psi_{\az,\bz}(S_{\mu}^0)$. The {\it generalized square
function Hardy space} $H_{\psi,L}^p(\rn)$ is defined to be the
completion of the space
\begin{eqnarray*}
\mathbb{H}_{\psi,L}^p(\rn):=\lf\{f\in L^2(\rn):\ Q_{\psi,L}f\in
T^p(\mathbb{R}^{n+1}_{+})\r\}
\end{eqnarray*}
with respect to the \emph{quasi-norm} $\|f\|_{H_{\psi,L}^p(\rn)}
:=\|Q_{\psi,L}f\|_{T^p(\mathbb{R}^{n+1}_{+})}.$
\end{definition}

The following theorem, which establishes the generalized square
function characterization of $H_L^p(\rn)$, is the main result of
this section.

\begin{theorem}\label{t5.1}
Let $p\in(0,\,1]$, $\omega\in[0,\,\pi/2)$, $L$ be the operator of
type $\omega$ satisfying the assumptions {\rm(A$_1)$}, {\rm(A$_2)$}
and {\rm(A$_3)$} in Section \ref{s2}, $\az\in(0,\,\fz)$,
$\bz\in(n(1/p-1/2)/(2k),\,\fz)$, $\mu\in(\omega,\,\pi/2)$ and
$\psi\in\Psi_{\az,\bz}(S_{\mu}^0)$. Then the Hardy space
$H_L^p(\rn)=H_{\psi,L}^p(\rn)$ with equivalent norms.
\end{theorem}

Before proving Theorem \ref{t5.1}, we first give an application of
this theorem. Let $\az\in(0,\,\fz)$ and $L^\az$ be the {\it
fractional power with exponent $\az$ of $L$} defined by the $H_\fz$
functional calculus in $L^2(\rn)$ (see, for example, \cite{mc,ha}).
More precisely, choose $m\in\nn$ such that $m>\az$. Then,
$z^\az(1+z)^{-m}\in\Psi_{\az,\,m-\az}(S_\mu^0)$ for all
$\mu\in[0,\,\pi/2)$ and $L^\az$ is defined by setting
\begin{eqnarray*}
L^\az:= (z^\az)(L):=(1+L)^m\lf(\frac{z^\az}{(1+z)^m}\r)(L).
\end{eqnarray*}
For more details about $L^\az$, we refer the reader to \cite{mc,ha}
and the references therein.

{\it Assume that $-L^\az$ generates a  bounded holomorphic semigroup
$\{e^{-tL^\az}\}_{t>0}$}. From \cite[Example 3.4.6]{ha}, it follows
that this is true when $\az\in(0,\,1]$, and in this case,
$\{e^{-tL^\az}\}_{t>0}$ is called the {\it subordinated semigroup}
(see \cite[p.\,80]{ha} for more details). For all $f\in L^2(\rn)$,
define the {\it $L^\az$-adapted square function} $S_{L^\az}$ by
setting, for all $x\in\rn$,
\begin{eqnarray}\label{5.2}
S_{L^\az}f(x):= \lf\{\iint_{\bgz(x)}\lf|t^{2k\az}L^\az e^{-t^{2k\az
}L^\az} f(y)\r|^2\frac{dy\,dt}{t^{n+1}}\r\}^{1/2}.
\end{eqnarray}

For $p\in(0,\,1]$, we also define the {\it Hardy space
$H_{L^\az}^p(\rn)$ associated with $L^\az$} to be  the completion of
the set
\begin{eqnarray}\label{5.3}
\mathbb{H}_{L^\az}^p(\rn):=\lf\{f\in L^2(\rn): \lf\|S_{L^\az}f
\r\|_{L^p(\rn)}<\fz\r\}
\end{eqnarray}
with respect to the quasi-norm
$\|f\|_{H_{L^\az}^p(\rn)}:=\lf\|S_{L^\az}f\r\|_{L^p(\rn)}.$

With the help of Theorem \ref{t5.1}, we immediately obtain the
following interesting corollary.

\begin{corollary}\label{c5.1}
Let $p\in (0,1]$ and $L$ satisfy the assumptions {\rm(A$_1)$},
{\rm(A$_2)$} and {\rm(A$_3)$}. Assume further that when $\az\in
(1,\,\fz)$, $-L^\az$ generates a bounded holomorphic semigroup.
Then, for all $\az\in(0,\,\fz)$, the Hardy spaces
$H_{L^\az}^p(\rn)=H_L^p(\rn)$ with equivalent norms.
\end{corollary}

\begin{proof}
Let $\omega\in[0,\,\pi/2)$. Recall that $L$ is an operator of type
$\omega$. For all $\az\in(0,\,\fz)$, $\mu\in(\omega,\,\pi/2)$ and
$\xi\in S_\mu^0$,  set $\psi(\xi):= \xi^\az e^{-\xi^\az}$. Then,
for all $\bz\in(n(1/p-1/2)/(2k),\,\fz)$, $\psi\in
\Psi_{\az,\bz}(S_\mu^0)$ and hence, by Theorem \ref{t5.1}, we
conclude that for all $f\in L^2(\rn)$,
\begin{eqnarray*}
\|f\|_{H_{L^\az}^p(\rn)}=\|S_{L^\az}f\|_{L^p(\rn)}=
\|Q_{\psi,L}f\|_{T^p(\mathbb{R}^{n+1}_+)}
=\|f\|_{H_{\psi,L}^p(\rn)}\sim\|f\|_{H_{L}^p(\rn)},
\end{eqnarray*}
which, together with the density of $L^2(\rn)$ in $H_L^p(\rn)$ and
$H_{L^\az}^p(\rn)$, shows that $H_L^p(\rn)=H_{L^\az}^p(\rn)$ with
equivalent norms. This finishes the proof of Corollary \ref{c5.1}.
\end{proof}

Let $\omega\in[0,\,\pi/2)$ be as in Section \ref{s2} and
$\mu\in(\omega,\,\pi/2)$. To prove Theorem \ref{t5.1}, we introduce
two operators as follows:
\begin{enumerate}
\item[(i)] For all $F\in T^2(\mathbb{R}^{n+1}_{+})$ and
$\psi\in\Psi(S_\mu^0)$, the \emph{operator $\pi_{\psi,L}$} is
defined by setting, for all $x\in\rn$,
\begin{eqnarray}\label{5.4}
\pi_{\psi,L}F(x):=\dint_0^\fz\psi(t^{2k}L)F(x,\,t)\,\frac{dt}{t};
\end{eqnarray}
\item[(ii)] For all $\psi$,  $\wz\psi\in\Psi(S_\mu^0)$, $f\in
H_\fz(S_{\mu}^0)$ and $F\in T^2(\mathbb{R}^{n+1}_{+})$, the
\emph{operator $Q^{f}$} is defined by setting, for all $x\in\rn$
and $s\in(0,\,\fz)$,
\begin{eqnarray}\label{5.5}
Q^{f}F(x,\,s)&&:= Q_{\psi,L}\circ f(L)\circ
\pi_{\wz\psi,L}F(x,\,s)=\dint_0^\fz\psi(s^{2k}L)f(L)\wz\psi(t^{2k}L)
F(x,\,t)\,\frac{dt}{t},
\end{eqnarray}
where the operator $Q_{\psi,L}$ is defined as in \eqref{5.1}.
\end{enumerate}

Observe that by \eqref{4.3}, $Q_{\psi,L}$ is bounded from
$L^2(\rn)$ to $T^2(\mathbb{R}^{n+1}_+)$ and so is
$Q_{\overline{\psi},L^*}$. By Fubini's theorem and H\"older's
inequality, we see that for all $F\in T^2(\mathbb{R}^{n+1}_{+})$
and $g\in L^2(\rn)$,
\begin{eqnarray*}
\dint_{\rn}\pi_{\psi,L}F(x)\overline{g(x)}\,dx&&=\dint_{\rn}
\dint_0^\fz \psi(t^{2k}L)F(x,\,t)\,\frac{dt}{t}\overline{g(x)}\,dx
=\dint_0^\fz\dint_{\rn}
F(x,\,t)\overline{Q_{\overline\psi,L^*}(g)(x)} \,dx\,\frac{dt}{t}.
\end{eqnarray*}
Thus, $Q_{\overline\psi,L^*}$ is the adjoint operator of
$\pi_{\psi,L}$, which, together with the above observation, shows
that $\pi_{\psi,L}$ is bounded from $T^2(\mathbb{R}^{n+1}_{+})$ to
$L^2(\rn)$. From these facts and \eqref{5.5} together with that $L$
has a bounded $H_\fz$ functional calculus in $L^2(\rn)$, it follows
that $Q^f$ is bounded on $T^2(\mathbb{R}^{n+1}_{+})$.

Let $\sz_1$, $\sz_2$, $\tau_1$, $\tau_2\in (0,\,\fz)$. Assume that
$\psi\in\Psi_{\sz_1,\tau_1}(S_\mu^0)$ and
$\wz\psi\in\Psi_{\sz_2,\tau_2}(S_\mu^0)$. We now consider the
operator $\psi(s^{2k}L)f(L)\wz\psi(t^{2k}L)$ in \eqref{5.5}.  Let
$a\in(0,\,\min\{\sz_1,\tau_2\})$ and $b\in(0,\,
\min\{\sz_2,\tau_1\})$. For $s$, $t\in (0,\,\fz)$, when $s\le t$, we
write
\begin{eqnarray}\label{5.6}
\psi(s^{2k}L)f(L)\wz\psi(t^{2k}L)&&=\lf(\frac{s^{2k}}{t^{2k}}\r)^a
\lf(s^{2k}L\r)^{-a}\psi(s^{2k}L)f(L)\lf(t^{2k}L\r)^a
\wz\psi(t^{2k}L)=: \lf(\frac{s^{2k}}{t^{2k}}\r)^a
T_{s^{2k},t^{2k}},
\end{eqnarray}
while when $s> t$, we write
\begin{eqnarray}\label{5.7}
\psi(s^{2k}L)f(L)\wz\psi(t^{2k}L)&&=\lf(\frac{t^{2k}}{s^{2k}}\r)^b
\lf(s^{2k}L\r)^{b}\psi(s^{2k}L)f(L)\lf(t^{2k}L\r)^{-b}
\wz\psi(t^{2k}L)=: \lf(\frac{t^{2k}}{s^{2k}}\r)^b
T_{s^{2k},t^{2k}}.
\end{eqnarray}
Then, we have the following useful estimate on
$\{T_{s,\,t}\}_{s,t>0}$.

\begin{lemma}\label{l5.1}
Let $k\in\nn$ be as in \eqref{2.6}, $\sz_1$, $\sz_2$, $\tau_1$,
$\tau_2\in(0,\,\fz)$, $\omega\in [0,\,\pi/2)$,
$\mu\in(\omega,\,\pi/2)$, $\psi\in\Psi_{\sz_1,\tau_1}(S_{\mu}^0)$,
$\wz\psi\in \Psi_{\sz_2,\tau_2}(S_{\mu}^0)$,
$a\in(0,\,\min\{\sz_1,\,\tau_2\})$ and
$b\in(0,\,\min\{\sz_2,\,\tau_1\})$. Let $f\in H_\fz(S_\mu^0)$. Let
$\{T_{s,\,t}\}_{s,t>0}$ be as in \eqref{5.6} and \eqref{5.7} with
$s^{2k}$ and $t^{2k}$ replaced, respectively, by $s$ and $t$. Then,
there exists a positive constant $C$ such that  for all
$M\in(0,\,\min\{\sz_2+a,\,\tau_1+b\})$, $s,\,t\in(0,\,\fz)$, closed
sets $E$, $F\subset\rn$ and $g\in L^2(\rn)$ supported in $E$,
\begin{eqnarray}\label{5.8}
\hs\hs\|T_{s,t}g\|_{L^2(F)}\le C\|f\|_{L^\fz(\rn)}\min\lf\{1,\,
\frac{\max\{t,\,s\}}{\lf[\dist(E,\,F)\r]^{2k}}\r\}^{M}
\|g\|_{L^2(E)}.
\end{eqnarray}
\end{lemma}

\begin{proof} We prove this lemma by considering two cases. If $s\le t$,
since $a\in(0,\, \min\{\sz_1,\tau_2\})$, we conclude that for all
$\xi\in S_{\mu}^0$,
\begin{eqnarray*}
\lf|\lf(s\xi\r)^{-a}\psi(s\xi)f(\xi)\r|\ls\frac{|s\xi
|^{\sz_1-a}}{1+|s\xi|^{\sz_1+\tau_1}}\|f\|_{L^\fz(\rn)}\ls1
\end{eqnarray*}
and
\begin{eqnarray*}
\lf|\lf(t\xi\r)^{a}\wz\psi(t\xi)\r|\ls\frac{\lf|t\xi\r|^
{\sz_2+a}}{1+|t\xi|^{\sz_2+\tau_2}},
\end{eqnarray*}
which, together with Lemma \ref{l3.3} with $\psi$ and $f$ therein
replaced by $(t\xi)^a\wz\psi(t\xi)$ and
$(s\xi)^{-a}\psi(s\xi)f(\xi)$ respectively, implies that the family
$\{T_{s,t}\}_{s\le t}$ of operators satisfies the $k$-Davies-Gaffney
estimate of order $\sz_2+a$ in $t$.

Similarly, if $s>t$, since $b\in(0,\, \min\{\sz_2,\tau_1\})$, we
see that for all $\xi\in S_\mu^0$,
\begin{eqnarray*}
\lf|f(\xi)\lf(t\xi\r)^{-b}\wz\psi(t\xi)\r|\le\frac{|t\xi|^{\sz_2-b}}
{1+|t\xi|^{\sz_2+\tau_2}}\|f\|_{L^\fz(\rn)}\ls1
\end{eqnarray*}
and
\begin{eqnarray*}
\lf|\lf(s\xi\r)^{b}\psi(s\xi)\r|\le\frac{\lf|s\xi\r|^{\tau_1+b}}
{1+|s\xi|^{\sz_1+\tau_1}},
\end{eqnarray*}
which, together with Lemma \ref{l3.3} with $\psi$ and $f$ therein
replaced by $(s\xi)^b\psi(s\xi)$ and
$f(\xi)(t\xi)^{-b}\wz\psi(t\xi)$ respectively, implies that the
family $\{T_{s,t}\}_{s>t}$ of operators satisfies the
$k$-Davies-Gaffney estimate of order $\tau_1+b$ in $s$.

Thus, for all $M\in (0,\,\min\{\sz_2+a,\,\tau_1+b\})$, we
immediately obtain \eqref{5.8}, which completes the proof of Lemma
\ref{l5.1}.
\end{proof}

\begin{lemma}\label{l5.2}
Let $p\in(0,\,1]$, $L$ be the operator of  type $\omega$ satisfying
the assumptions {\rm(A$_1)$}, {\rm(A$_2)$} and {\rm(A$_3)$} in
Section \ref{s2}, $\az\in(0,\,\fz)$,
$\bz\in(n(1/p-1/2)/(2k),\,\fz)$, $\omega\in[0,\,\pi/2)$,
$\mu\in(\omega,\,\pi/2)$, $\psi\in\Psi_{\az,\bz}(S_{\mu}^0)$ and
$\wz\psi\in\Psi_{\bz,\az}(S_{\mu}^0)$. Then the operator $Q^f$
originally defined in \eqref{5.5} on $T^2(\mathbb{R}^{n+1}_{+})$ can
be continuously extended to a bounded linear operator on
$T^p(\mathbb{R}^{n+1}_{+})$. Moreover, there exists a positive
constant $C$ such that for all $F\in T^p(\mathbb{R}^{n+1}_{+})$ and
$f\in H_\fz(S_\mu^0)$,
\begin{eqnarray}\label{5.9}
\lf\|Q^f F\r\|_{T^p(\mathbb{R}^{n+1}_{+})}\le
C\|f\|_{L^\fz(S_\mu^0)} \|F\|_{T^p(\mathbb{R}^{n+1}_{+})}.
\end{eqnarray}
\end{lemma}

\begin{proof}
By the density of $T^2(\mathbb{R}^{n+1}_{+})\cap
T^p(\mathbb{R}^{n+1}_{+})$ in $T^p(\mathbb{R}^{n+1}_{+})$ (see
\cite{cms}), it suffices to prove \eqref{5.9} for all $F\in
T^2(\mathbb{R}^{n+1}_{+})\cap T^p(\mathbb{R}^{n+1}_{+})$. To this
end, by borrowing some ideas from the proof of Theorem 1.1 in
\cite{yy}, we only need to show that for all
$T^p(\mathbb{R}^{n+1}_+)$-atoms $A$,
\begin{eqnarray}\label{5.10}
\lf\|Q^f A\r\|_{T^p(\mathbb{R}^{n+1}_{+})}\ls
\|f\|_{L^\fz(S_\mu^0)}.
\end{eqnarray}
Indeed, if \eqref{5.10} holds, then from Theorem \ref{t4.2} and
the $T^2(\mathbb{R}^{n+1}_{+})$-boundedness of $Q^fA$, it follows
that for any $F\in T^2(\mathbb{R}^{n+1}_+)\cap
T^p(\mathbb{R}^{n+1}_+)$, there exist a sequence
$\{A_j\}_{j=0}^\fz$ of $T^p(\mathbb{R}^{n+1}_+)$-atoms and
$\{\lz_j\}_{j=0}^\fz\in l^p$ such that $F:=\sum_{j=0}^\fz
\lz_jA_j$ with the summation converges in both pointwise and
$T^2(\mathbb{R}^{n+1}_{+})$, and $\{\sum_{j=0}^\fz
|\lz_j|^p\}^{1/p}\sim \|F\|_{T^p(\mathbb{R}^{n+1}_{+})}$. We claim
that for $\frac{dx dt}{t}$-almost every $(x,\,t)\in
\mathbb{R}^{n+1}_{+}$,
\begin{eqnarray}\label{5.11}
\lf|Q^f \lf(\dsum_{j=0}^\fz\lz_jA_j\r)(x,t)\r|\le\dsum_{j=0}^\fz
\lf|\lz_jQ^f A_j(x,\,t)\r|.
\end{eqnarray}
Assume this claim for the moment. By \eqref{5.11} and $p\in(0,1]$,
together with the monotonicity of $l^p$, we have
\begin{eqnarray*}
\lf\|Q^f F\r\|_{T^p(\mathbb{R}^{n+1}_{+})}&&\le
\lf\{\dsum_{j=0}^\fz|\lz_j|^p\|Q^fA_j\|_{T^p
(\mathbb{R}^{n+1}_+)}^p\r\}^{1/p}\\
&&\le\dsup_{j\in\zz_+}
\lf\{\|Q^f(A_j)\|_{T^p(\mathbb{R}^{n+1}_+)}\r\}\lf\{\dsum_{j=0}^\fz
|\lz_j|^p\r\}^{1/p}\\
&&\ls\|f\|_{L^\fz(S_\mu^0)}\lf\{\dsum_{j=0}^\fz
|\lz_j|^p\r\}^{1/p}\sim
\|f\|_{L^\fz(S_\mu^0)}\|F\|_{T^p(\mathbb{R}^{n+1}_{+})}.
\end{eqnarray*}
That is, $Q^f$ is bounded on $T^p(\mathbb{R}^{n+1}_+)$. To show
the claim \eqref{5.11}, for simplicity of the notation, let
$d\mu(x,\,t):= \frac{dx\,dt}{t}$ for all
$(x,\,t)\in\mathbb{R}^{n+1}_+$. By the
$T^2(\mathbb{R}^{n+1}_+)$-boundedness of $Q^f$ and the
$T^2(\mathbb{R}^{n+1}_+)$-convergence of $F=\sum_{j=0}^\fz
\lz_jA_j$, we conclude that for any $\eta\in(0,\,\fz)$,
\begin{eqnarray*}
\dlim_{N\to\fz}\mu\lf(\lf\{x\in\rn:\
\lf|Q^f\lf(\dsum_{i=N+1}^\fz\lz_jA_j\r)\r|
>\eta\r\}\r)\ls\dlim_{N\to\fz}\frac{1}{\eta^2}\lf\|
\dsum_{i=N+1}^\fz\lz_jA_j\r\|_{T^2(\mathbb{R}^{n+1}_+)}^2=0.
\end{eqnarray*}
This, combined with the Riesz theorem, implies that there exists a
subsequence
$$\lf\{Q^f\lf(\sum_{j=N_\ell+1}^\fz\lz_jA_j\r)\r\}_{l\in\nn}$$
of $\{Q^f(\sum_{j=N+1}^\fz\lz_jA_j)\}_{N\in\nn}$ such that for
$\mu$-almost every $(x,\,t)\in\mathbb{R}^{n+1}_+$,
\begin{eqnarray*}
\dlim_{\ell\to\fz}Q^f\lf(\dsum_{j=N_\ell+1}^\fz\lz_jA_j\r)(x,\,t)=0,
\end{eqnarray*}
where $\{N_\ell\}_{\ell\in\nn}\subset \nn$ and
$\lim_{\ell\to\fz}N_\ell=\fz$. Therefore, for $\mu$-almost every
$(x,\,t)\in\mathbb{R}^{n+1}_+$ and all $\ell\in\nn$,
\begin{eqnarray*}
\lf|Q^f\lf(\dsum_{j=0}^\fz\lz_j A_j\r)(x,\,t)\r|
&&\le\dsum_{j=0}^{N_\ell}\lf|\lz_jQ^fA_j(x,\,t)\r|
+\lf|Q^f\lf(\dsum_{j=N_\ell+1}^{\fz}\lz_jA_j\r)(x,\,t)\r|,
\end{eqnarray*}
which, together with letting $\ell\to\fz$, shows the claim
\eqref{5.11}.

To finish the proof of Lemma \ref{l5.2}, we still need to prove
\eqref{5.10}. By the homogeneity of the norm
$\|\cdot\|_{T^p(\rr_+^{n+1})}$, without loss of generality, we may
assume that $\|f\|_{L^\fz(S_\mu^0)}=1$. Let $Q$ be the cube
associated with the $T^p(\mathbb{R}^{n+1}_+)$-atom $A$ and $R_Q:=
Q\times (0,\,l(Q))$, where $l(Q)$ denotes the \emph{side length}
of $Q$. For all $i\in\nn$, set
$2^iR_Q:=2^iQ\times(0,\,2^il(Q))\subset \mathbb{R}^{n+1}_{+}$ and
$S_i(R_Q) :=2^iR_Q\setminus(2^{i-1}R_Q)$.

For $i=1$, by H\"older's inequality, the
$T^2(\mathbb{R}^{n+1}_{+})$-boundedness of $Q^f$ and the size
condition \eqref{4.10} of $T^p(\mathbb{R}^{n+1}_+)$-atoms, we see
that
\begin{eqnarray}\label{5.12}
\lf\|\chi_{2R_Q}Q^f A\r\|_{T^p(\mathbb{R}^{n+1}_{+})}&&=
\lf\|\mathcal{A}(\chi_{2R_Q}Q^f A)\r\|_{L^p(\rn)}\nonumber\\ \nonumber
&&\le\lf\|\mathcal{A}(\chi_{2R_Q}Q^f A)\r\|_{L^2(\rn)}
\lf|2(\sqrt n+2)Q\r|^{1/p-1/2}\\
&&\ls\lf\{\iint_{R_Q}|A(x,\,t)|^2\frac{dx\,dt}{t}\r\}^{1/2}
\lf|Q\r|^{1/p-1/2}\ls1.
\end{eqnarray}

For $i\ge2$, using H\"older's inequality and Fubini's theorem, we
then conclude that
\begin{eqnarray*}
\lf\|\chi_{S_i(R_Q)}Q^f A\r\|_{T^p(\mathbb{R}^{n+1}_{+})}
&&=\lf\|\mathcal{A}(\chi_{S_i(R_Q)}Q^f A)\r\|_{L^p(\rn)}
\ls\lf\|\mathcal{A}(\chi_{S_i(R_Q)}Q^f A)\r\|_{L^2(\rn)}
\lf|2^i(2+\sqrt n)Q\r|^{\frac{1}{p}-\frac{1}{2}}\\
&&\sim\lf\{\lf[\dint_0^{2^{i-1}l(Q)}\dint_{\rn}
\chi_{S_i(R_Q)}(x,\,s)\lf|Q^f A(x,\,s)\r|^2
\frac{dx\,ds}{s}\r]^{1/2}\r.\\
&&\hs\lf.+\lf[\dint_{2^{i-1}l(Q)}^{2^{i}l(Q)}\dint_{\rn}\cdots
\r]^{1/2}\r\}\lf|2^iQ\r|^{\frac{1}{p}-\frac{1}{2}}\\
&&=:\lf\{\mathrm{I}+\mathrm{O}\r\}
\lf|2^iQ\r|^{\frac{1}{p}-\frac{1}{2}}.
\end{eqnarray*}

To estimate $\mathrm{O}$, from \eqref{5.5}, Minkowski's inequality,
Fubini's inequality, Lemma \ref{l5.1} and H\"older's inequality, we
deduce that
\begin{eqnarray*}
\mathrm{O}&&\sim\lf\{\dint_{2^{i-1}l(Q)}^{2^{i}l(Q)}\dint_\rn
\chi_{S_i(R_Q)}(x,\,s)\lf|Q^f A(x,\,s)\r|^2\,dx\,\frac{ds}{s}\r\}^{1/2}\\
&&\sim\lf\{\dint_{2^{i-1}l(Q)}^{2^{i}l(Q)}\dint_\rn\chi_{S_j(R_Q)}(x,\,s)
\lf|\dint_0^\fz\psi(s^{2k}L)f(L)\wz\psi(t^{2k}L)
A(x,\,t)\,\frac{dt}{t}\r|^2\,dx\,\frac{ds}{s}\r\}^{1/2}\\
&&\sim\lf\{\dint_{2^{i-1}l(Q)}^{2^{i}l(Q)}\dint_\rn\chi_{S_i(R_Q)}(x,\,s)
\lf|\dint_{0}^\fz\lf(\frac{t}{s}
\r)^{2kb}T_{s^{2k},t^{2k}}A(x,\,t)\,\frac{dt}{t}
\r|^2\,dx\,\frac{ds}{s}\r\}^{1/2}\\
&&\ls\dint_{0}^\fz\lf[\frac{t}{2^il(Q)}\r]^{2kb}
\lf[\dint_{2^{i-1}l(Q)}^{2^{i}l(Q)}
\dint_\rn\lf|T_{s^{2k},t^{2k}}A(x,\,t)\r|^2
\chi_{S_i(R_Q)}(x,\,s)\,dx\,\frac{ds}{s}\r]^{1/2}\,\frac{dt}{t}\\
&&\ls\dint_{0}^{l(Q)}\lf[\frac{t}{2^il(Q)}\r]^{2kb}
\lf\|A(\cdot,\,t)\r\|_{L^2(\rn)} \,\frac{dt}{t}\\
&&\ls\lf\{\dint_0^{l(Q)}\lf\|A(\cdot,\,t)\r\|_{L^2(\rn)}^2
\frac{dt}{t}\r\}^{1/2}\lf\{\dint_0^{l(Q)}\lf[\frac{t}
{2^il(Q)}\r]^{4kb}\frac{dt}{t}\r\}^{1/2}\\
&&\ls2^{-2ikb}|Q|^{1/2-1/p}\sim
2^{-i[2kb-n(1/p-1/2)]}|2^iQ|^{1/2-1/p} \sim
2^{-i\gz_1}|2^iQ|^{1/2-1/p},
\end{eqnarray*}
where $b\in(n(1/p-1/2)/(2k),\,\bz)$ and $\gz_1:=2kb-n(1/p-1/2)>0$.

Let $a\in(0,\,\az)$. To estimate $\mathrm{I}$, by Fubini's theorem
and Minkowski's inequality, we see that
\begin{eqnarray*}
\mathrm{I} &&\sim\lf\{\dint_0^{2^{i-1}l(Q)}\dint_\rn
\chi_{S_i(R_Q)}(x,\,s)
\lf|Q^f A(x,\,s)\r|^2\,dx\,\frac{ds}{s}\r\}^{1/2}\\
&&\sim\lf\{\dint_0^{2^{i-1}l(Q)}\dint_\rn\chi_{S_i(R_Q)}(x,\,s)
\lf|\dint_0^\fz\psi(s^{2k}L)f(L)\wz\psi(t^{2k}L)A(x,\,t)
\,\frac{dt}{t}\r|^2\,dx\,\frac{ds}{s}\r\}^{1/2}\\
&&\sim\lf\{\dint_0^{2^{i-1}l(Q)}
\dint_\rn\chi_{S_i(R_Q)}(x,\,s)\lf|\dint_0^\fz\min\lf\{\lf(\frac{s}{t}\r)^{2ka},
\,\lf(\frac{t}{s}\r)^{2kb}\r\} T_{s^{2k},t^{2k}}A(x,\,t)
\,\frac{dt}{t}\r|^2\,dx\,\frac{ds}{s}\r\}^{1/2}\\
&&\ls\dint_0^{l(Q)} \lf\{\dint_0^t\dint_\rn
\lf(\frac{s}{t}\r)^{4ka}\lf|T_{s^{2k},t^{2k}}A(x,\,t)
\r|^2\chi_{S_i(R_Q)}(x,\,s)\,dx\,\frac{ds}{s}\r\}^{1/2}\frac{dt}{t}\\
&&\hs+\dint_0^{l(Q)}\lf\{\dint_t^{2^{i-1}l(Q)}\dint_\rn
\lf(\frac{t}{s}\r)^{4kb}\lf|T_{s^{2k},t^{2k}}A(x,\,t)\r|^2
\chi_{S_i(R_Q)}(x,\,s)\,dx\,\frac{ds}{s}\r\}^{1/2}\frac{dt}{t}\\
&&=: \mathrm{I}_1+\mathrm{I}_{2}.
\end{eqnarray*}
Let $M\in(n(1/p-1/2)/(2k),\,\min\{\az+b,\,\bz+a\})$. It follows,
from Lemma \ref{l5.1} and H\"older's inequality, that
\begin{eqnarray*}
\mathrm{I}_1&&\ls\dint_0^{l(Q)} \lf[\dint_0^t
\lf(\frac{s}{t}\r)^{4ka}\lf\{\frac{t^{2k}}
{[\dist(R_Q,\,S_i(R_Q))]^{2k}}\r\}^{2M}
\lf\|A(\cdot,\,t)\r\|_{L^2(\rn)}^2\,\frac{ds}s\r]^{1/2}\frac{dt}{t}\\
&&\ls\dint_0^{l(Q)}\lf[\dint_0^t \lf(\frac{s}{t}\r)^{4ka}
\lf\{\frac{t^{2k}}{[2^il(Q)]^{2k}}\r\}^{2M}
\lf\|A(\cdot,\,t)\r\|_{L^2(\rn)}^2\,\frac{ds}s\r]^{1/2}\frac{dt}{t}\\
&&\sim\frac{1}{[2^il(Q)]^{2kM}}\dint_0^{l(Q)}\lf\|A(\cdot,\,t)
\r\|_{L^2(\rn)}t^{2kM}\lf[\dint_0^t\lf(\frac{s}{t}\r)^{4ka}\,\frac{ds}
{s}\r]^{1/2}\,\frac{dt}{t}\\
&&\ls\frac{1}{[2^il(Q)]^{2kM}}\lf\{\iint_{R_Q}\lf|A(x,\,t)\r|^2
\,\frac{dx\,dt}{t}\r\}^{1/2}
\lf\{\dint_0^{l(Q)}t^{4kM}\,\frac{dt}{t}\r\}^{1/2}\\
&&\sim
2^{-2ikM}\lf\{\iint_{R_Q}|A(x,\,t)|^2\,\frac{dx\,dt}{t}\r\}^{1/2}
\ls2^{-i[2kM-n(1/p-1/2)]}|2^iQ|^{1/p-1/2}\sim2^{-i\gz_2}
|2^iQ|^{1/p-1/2},
\end{eqnarray*}
where $\gz_2:= 2kM-n(1/p-1/2)>0$.

For $\mathrm{I}_{2}$, via some similar calculations to the
estimate of $\mathrm{I}_1$, we see that
\begin{eqnarray*}
\mathrm{I}_{2} &&\ls\dint_0^{l(Q)}\lf[\dint_t^{2^il(Q)}
\lf(\frac{t}{s}\r)^{4kb}
\lf\{\frac{s^{2k}}{[2^il(Q)]^{2k}}\r\}^{2M}\,\frac{ds}{s}\r]^{1/2}
\lf\|A(\cdot,\,t)\r\|_{L^2(\rn)}\,\frac{dt}{t}\\
&&\ls\dint_0^{l(Q)}\lf\{\lf[\frac{t}{2^il(Q)}\r]^{2kb}+
\lf[\frac{t}{2^il(Q)}\r]^{2kM}\r\}
\lf\|A(\cdot,\,t)\r\|_{L^2(\rn)} \,\frac{dt}{t}\\
&&\ls\lf\{\iint_{R_Q}\lf|A(x,\,t)\r|^2\,\frac{dx\,dt}{t}\r\}^{1/2}
\lf\{\dint_0^{l(Q)}\lf(\lf[\frac{t}{2^il(Q)}\r]^{4kb}+
\lf[\frac{t}{2^il(Q)}\r]^{4kM}\r)\,\frac{dt}{t}\r\}^{1/2}\\
&&\ls\lf(2^{-2ikb}+2^{-2ikM}\r)|Q|^{1/2-1/p}
\sim\lf(2^{-i\gz_1}+2^{-i\gz_2}\r)\lf|2^iQ\r|^{1/2-1/p}.
\end{eqnarray*}

Combining  the estimates of  $\mathrm{I}_1$ and $\mathrm{I}_2$, we
obtain
\begin{eqnarray}\label{5.13}
\mathrm{O}\ls\lf(2^{-i\gz_1}+2^{-i\gz_2}\r)|2^iQ|^{1/2-1/p}.
\end{eqnarray}
By \eqref{5.12} and \eqref{5.13}, we conclude that
\begin{eqnarray*}
\lf\|Q^f A\r\|_{T^p(\mathbb{R}^{n+1}_{+})}^p
&&\ls\lf\|\chi_{2R_Q}Q^f A\r\|_{T^p(\mathbb{R}^{n+1}_{+})}^p+
\dsum_{i=2}^\fz\lf\|\chi_{S_i(R_Q)}Q^f A\r\|_{T^p(\mathbb{R}^{n+1}_{+})}^p
\ls 1+\dsum_{i=2}^\fz\lf(2^{-i\gz_1p}+2^{-i\gz_2p}\r) \ls1.
\end{eqnarray*}
Thus, \eqref{5.10} holds, which completes the proof of Lemma
\ref{l5.2}.
\end{proof}

As an application of Lemma \ref{l5.2}, we obtain the following
boundedness of $Q_{\psi,L}$ and $\pi_{\psi,L}$.

\begin{lemma}\label{l5.3}
Let $p\in(0,\,1]$, $\omega\in[0,\,\pi/2)$, $L$ be the operator of
type $\omega$ satisfying the assumptions {\rm(A$_1)$}, {\rm(A$_2)$}
and {\rm(A$_3)$} in Section \ref{s2}, $\az\in(0,\,\fz)$,
$\bz\in(n(1/p-1/2)/(2k),\,\fz)$ and $\mu\in(\omega,\,\pi/2)$. Then

\begin{enumerate}
\item[{\rm(i)}] the operator $Q_{\psi,L}$, originally defined on
$L^2(\rn)$ as in \eqref{5.1} with
$\psi\in\Psi_{\az,\bz}(S_\mu^0)$, can be extended to a bounded
linear operator from $H_L^p (\rn)$ to $T^p(\mathbb{R}^{n+1}_{+})$.

\item[{\rm(ii)}] the operator $\pi_{\psi,L}$, originally defined
on $T^2 (\mathbb{R}^{n+1}_{+})$ as in \eqref{5.4} with
$\psi\in\Psi_{\bz,\az}(S_\mu^0)$, can be extended to a bounded
linear operator from $T^p(\mathbb{R}^{n+1}_{+})$ to $H_L^p(\rn)$.
\end{enumerate}
\end{lemma}

\begin{proof} The proof of Lemma \ref{l5.3} is quite similar
to that of \cite[Proposition 4.9]{hmm}. For the convenience of the
reader, we present the details. We first recall a Calder\'on
reproducing formula from \cite[(4.12)]{hmm}. For all
$\psi\in\Psi(S_\mu^0)$, there exists a function $\wz \psi\in
\Psi(S_\mu^0)$ such that
\begin{eqnarray*}
\dint_0^\fz\psi(t)\wz\psi(t)\,\frac{dt}{t}=1.
\end{eqnarray*}
Moreover, we have
\begin{eqnarray}\label{5.14}
\pi_{\psi,L}\circ Q_{\wz\psi,L}=\pi_{\wz\psi,L} \circ Q_{\psi,L}=I\
\ \text{in}\ L^2(\rn).
\end{eqnarray}
In particular, let $\psi_0(z):= z e^{-z}$ for all $z\in S_\mu^0$. We
then choose $\wz\psi_0(z):= C(M)z^{M}e^{-z}$ for all $z\in S_\mu^0$
such that $\wz\psi_0\in\Psi_{M,N}(S_\mu^0)$ for any $N\in(0,\,\fz)$,
where $M$ is the smallest positive integer larger than
$n(1/p-1/2)/(2k)$ and $C(M)\int_0^\fz t^Me^{-2t}\,dt=1$.

By Definition \ref{d5.1} and \eqref{5.1}, we see that for all
$f\in H_L^p(\rn) \cap L^2(\rn)$,
\begin{eqnarray*}
\|Q_{\psi_0,L}f\|_{T^p(\mathbb{R}^{n+1}_{+})}=\|f\|_{H_L^p(\rn)},
\end{eqnarray*}
which implies that $Q_{\psi_0,L}$ is bounded from $H_L^p(\rn)$ to
$T^p(\mathbb{R}^{n+1}_{+})$.  For all
$\psi\in\Psi_{\az,\bz}(S_\mu^0)$, by this, together with the
Calder\'on reproducing formula \eqref{5.14} and Lemma \ref{l5.2}
with $f:=1$ therein, we conclude that for all $f\in H_L^p(\rn)$,
\begin{eqnarray*}
\|Q_{\psi,L}f\|_{T^p(\mathbb{R}^{n+1}_{+})}\sim
\lf\|Q_{\psi,L}\circ\pi_{\wz\psi_0,L}\circ Q_{\psi_0,L}f\r\|
_{T^p(\mathbb{R}^{n+1}_{+})}
\ls\|Q_{\psi_0,L}f\|_{T^p(\mathbb{R}^{n+1}_{+})}\sim\|f\|_{H_L^p(\rn)}.
\end{eqnarray*}
That is, $Q_{\psi,L}$ is bounded from $H_L^p(\rn)$ to
$T^p(\mathbb{R}^{n+1}_{+})$, which completes the proof of (i).

On the other hand, for all $\psi\in\Psi_{\bz,\az}(S_\mu^0)$, since
$\psi_0\in \Psi_{1,\bz}(S_\mu^0)$, it follows from Lemma \ref{l5.2}
with $f:=1$ therein that for all $F\in T^p(\mathbb{R}^{n+1}_{+})\cap
T^2(\mathbb{R}^{n+1}_{+})$,
\begin{eqnarray*}
\|\pi_{\psi,L}F\|_{H_L^p(\rn)}=\|Q_{\psi_0,L}\circ
\pi_{\psi,L}F\|_{T^p(\mathbb{R}^{n+1}_{+})}\ls\|F\|_{
T^p(\mathbb{R}^{n+1}_{+})},
\end{eqnarray*}
which shows that $\pi_{\psi,L}$ is bounded from
$T^p(\mathbb{R}^{n+1}_{+})$ to $H_L^p(\rn)$. This finishes  the
proof of  (ii) and hence Lemma \ref{l5.3}.
\end{proof}

\begin{proof}[Proof of Theorem \ref{t5.1}]
By Definitions \ref{d4.1} and \ref{d5.1}, to show Theorem
\ref{t5.1}, it suffices to prove that
$\mathbb{H}_L^p(\rn)=\mathbb{H}_{\psi,L}^p(\rn)$ with equivalent
norms.

The inclusion
$\mathbb{H}_L^p(\rn)\subset\mathbb{H}_{\psi,L}^p(\rn)$ is an easy
consequence of the boundedness of $Q_{\psi,L}$ from $H_L^p(\rn)$
to $T^p(\mathbb{R}^{n+1}_{+})$, which is true by Lemma
\ref{l5.3}(i). We now prove
$\mathbb{H}_{\psi,L}^p(\rn)\subset\mathbb{H}_L^p(\rn)$. Let
$\psi_0(z):= ze^z$ for all $z\in S_\mu^0$. Observe that for any
$\psi\in \Psi_{\az,\bz}(S_\mu^0)$, we can choose $\wz\psi(z):=\wz
C(M)z^{M}e^{-z}$ for all $z\in S_\mu^0$ such that \eqref{5.14}
holds, where $\wz C(M)$ is a constant such that $\wz
C(M)\int_0^\fz t^{M-1}e^{-t}\psi(t)\,dt=1$. From \eqref{5.14},
Lemma \ref{l5.2} with $f:=1$ therein, and Lemma \ref{l5.3}(i), we
infer that for all $f\in\mathbb{H}_{\psi,L}^p(\rn)$,
\begin{eqnarray*}
\|f\|_{H_L^p(\rn)}=\|Q_{\psi_0,L}f\|_{T^p(\mathbb{R}^{n+1}_{+})}
=\|Q_{\psi_0,L}\circ\pi_{\wz\psi,L}
\circ Q_{\psi,L}f\|_{T^p(\mathbb{R}^{n+1}_{+})}\ls \|Q_{\psi,L}f\|_{T^p(\mathbb{R}^{n+1}_{+})}\sim
\|f\|_{H_{\psi,L}^p(\rn)},
\end{eqnarray*}
which implies that
$\mathbb{H}_{\psi,L}^p(\rn)\subset\mathbb{H}_L^p(\rn)$. This
finishes the proof of Theorem \ref{t5.1}.
\end{proof}

\section{Riesz transforms on $H_{L^i}^p(\rn)$ for
$i\in{\{1,\,2\}}$}\label{s6}

In this section, for the $2k$-order divergence form
homogeneous elliptic operator $L_1$ with complex bounded measurable
coefficients and the $2k$-order Schr\"odinger type operator $L_2$,
we  consider the behavior of their Riesz transforms
$\nabla^k{{L_i}^{-1/2}}$ on the Hardy space $H_{L_i}^p(\rn)$,
respectively for $i\in\{1,\,2\}$. First, we study the boundedness of
$\nabla^k{{L_i}^{-1/2}}$ on $H_{L_i}^p(\rn)$ for $i\in\{1,\,2\}$. To
this end, we need the following useful estimates.

\begin{lemma}\label{l6.1}
Let $p\in(0,\,1]$, $M$,\,$k\in\nn$, $L_1$ be the  $2k$-order
divergence form  homogenous elliptic operator with complex bounded
measurable coefficients and $L_2$ the $2k$-order Schr\"odinger type
operator. Then, there exists a positive constant $C$ such that for
all $i\in\{1,\,2\}$, closed sets $E$, $F$ in $\rn$ with
$\dist(E,\,F)>0$, $f\in L^2(\rn)$ supported in $E$ and
$t\in(0,\,\fz)$,
\begin{eqnarray}\label{6.1}
\lf\|\nabla^{k}{L_i}^{-1/2}\lf(I-e^{-t{L_i}}\r)^{M}f\r\|_{L^2(F)}\le
C \lf(\frac{t}{\lf[\dist(E,\,F)\r]^{2k}}\r)^M\|f\|_{L^2(E)}
\end{eqnarray}
and
\begin{eqnarray}\label{6.2}
\lf\|\nabla^{k}{L_i}^{-1/2}\lf(t{L_i}e^{-t{L_i}}\r)^{M}f\r\|_{L^2(F)}\le
C \lf(\frac{t}{\lf[\dist(E,\,F)\r]^{2k}}\r)^M\|f\|_{L^2(E)}.
\end{eqnarray}
\end{lemma}

\begin{proof}
We prove this lemma by borrowing some ideas from \cite{hm}. Let
$i\in\{1,\,2\}$. From \cite[Theorem 1.1]{ahmt} and \cite[Theorem
8.1]{ou}, we deduce that $\nabla^k{{L_i}^{-1/2}}$ is bounded on
$L^2(\rn)$. Thus, it suffices to prove Lemma \ref{l6.1} in the
case that $t<\lf[\dist(E,\,F)\r]^{2k}$. By the $H_\fz$ functional
calculus in $L^2(\rn)$, we know that for all $f\in L^2(\rn)$,
\begin{eqnarray}\label{6.3}
{L_i}^{-1/2}f=\frac{1}{2\sqrt{\pi}}\dint_0^\fz
e^{-s{L_i}}s^{-1/2}f\,ds,
\end{eqnarray}
which, together with the change of variables, yields that
\begin{eqnarray*}
\nabla^{k}{L_i}^{-1/2}\lf(I-e^{-t{L_i}}\r)^{M}f
&&=\frac{1} {2\sqrt{\pi}}\dint_0^\fz
\nabla^{k}e^{-s{L_i}}\lf(I-e^{-t{L_i}}\r)^{M}f\, \frac{ds}{\sqrt{s}}\\
&&=\frac{\sqrt{M+2}}{2\sqrt{\pi}}\dint_0^\fz
\nabla^{k}e^{-(M+2)s{L_i}}\lf(I-e^{-t{L_i}}\r)^{M}f
\,\frac{ds}{\sqrt{s}}\\
&&=\frac{\sqrt{M+2}}{2\sqrt{\pi}}\dint_0^t\sqrt{s}\nabla^{k}
e^{-(M+2)s{L_i}} \lf[\dsum_{j=0}^M \binom{M}{j}(-1)^{j}
e^{-jt{L_i}}\r]f\,\frac{ds}{s}\\
&&\hs+\frac{\sqrt{M+2}}{2\sqrt{\pi}}\dint_t^\fz\sqrt{s}\nabla^{k}
e^{-(M+2)s{L_i}}\lf(I-e^{-tL_i}\r)^Mf\,\frac {ds}{s}\\
&&=:\mathrm{I}+\mathrm{O},
\end{eqnarray*}
where $\binom{M}{j}$ denotes the \emph{binomial coefficient}.

To estimate $\mathrm{I}$, we write
\begin{eqnarray*}
\mathrm{I}&&=\frac{\sqrt{M+2}}{2\sqrt{\pi}}\dint_0^t\sqrt{s}\nabla^{k}
e^{-s{L_i}}e^{-(M+1)s{L_i}}f\,\frac{ds}{s}+\dsum_{j=1}^M
\frac{\sqrt{M+2}}{2\sqrt{\pi}}
\binom{M}{j}(-1)^{j}\dint_0^t\nabla^{k}e^{-jt{L_i}}
e^{-(M+2)s{L_i}}f\, \frac{ds}{\sqrt{s}}\\
&&=: \mathrm{I}_0
+\dsum_{j=1}^M \mathrm{I}_j.
\end{eqnarray*}
For $\mathrm{I}_0$, it follows, from Minkowski's inequality,
Propositions \ref{p3.1} and \ref{p3.2}, Lemma \ref{l3.2} and the
assumption $t< \lf[\dist(E,\,F)\r]^{2k}$, that
\begin{eqnarray*}
\|\mathrm{I_0}\|_{L^2(F)}&&\ls
\dint_0^t\lf\|\sqrt{s}\nabla^ke^{-s{L_i}}e^
{-(M+1)s{L_i}}f\r\|_{L^2(F)}\,\frac{ds}s\\
&&\ls\dint_0^t\exp\lf\{-\frac{\wz C\lf[\dist(E,\,F)\r]^{2k/(2k-1)}}
{s^{1/(2k-1)}}\r\}\|f\|_{L^2(E)}\,\frac{ds}s\\
&&\ls
\exp\lf\{-\frac{\wz C_1\lf[\dist(E,\,F)\r]^{2k/(2k-1)}}{t^{1/(2k-1)}}\r\}
\dint_0^t\exp\lf\{-\frac{\wz
C_2\lf[\dist(E,\,F)\r]^{2k/(2k-1)}}
{s^{1/(2k-1)}}\r\}\,\frac{ds}s\|f\|_{L^2(E)}\\
&&\ls \exp\lf\{-\frac{\wz
C_1\lf[\dist(E,\,F)\r]^{2k/(2k-1)}}{t^{1/(2k-1)}}\r\}
\frac{t}{[\dist(E,\,F)]^{2k}}
\|f\|_{L^2(E)}\\
&&\ls\lf(\frac{t}{\lf[\dist(E,\,F)\r]^{2k}}\r)^M\|f\|_{L^2(E)},
\end{eqnarray*}
where $\wz C$, $\wz C_1$, $\wz C_2$ are positive constants such that
$\wz C_1+\wz C_2=\wz C$.

For each $\mathrm{I}_j$, $j\ge1$, by Lemma \ref{l3.2} and
Propositions \ref{p3.1} and \ref{p3.2}, we see that
\begin{eqnarray*}
\lf\|\mathrm{I}_j\r\|_{L^2(F)}&&\ls\frac{1}{\sqrt{t}}
\dint_0^t\lf\|\lf(\sqrt{jt}\nabla^{k}e^{-jt{L_i}}\r)
\circ\lf(e^{-(M+2)s{L_i}}\r)f\r\|_{L^2(F)}\frac{ds}{\sqrt{s}}\\
&&\ls\frac{1}{\sqrt{t}}\|f\|_{L^2(E)} \dint_0^t\exp\lf\{-\frac{\wz
C\lf[\dist(E,\,F)\r]^{2k/(2k-1)}}
{t^{1/(2k-1)}}\r\}\frac{ds}{\sqrt{s}} \\
&&\ls\lf(\frac{t}{\lf[\dist(E,\,F)\r]^{2k}}\r)^M\|f\|_{L^2(E)},
\end{eqnarray*}
which, together with the estimate of $\mathrm{I}_0$, implies that
\begin{eqnarray}\label{6.4}
\|\mathrm{I}\|_{L^2(F)}\ls\lf(\frac{t}{
\lf[\dist(E,\,F)\r]^{2k}}\r)^M\|f\|_{L^2(E)},
\end{eqnarray}
here and in what follows, $\wz C$ always denotes a positive
constant. We now estimate $\mathrm{O}$ by writing
\begin{eqnarray*}
\mathrm{O}&&\sim\dint_t^\fz \sqrt{s}\nabla^k e^{-s{L_i}}e^{-Ms{L_i}}
\lf(I-e^{-t{L_i}}\r)^Me^{-s{L_i}}f\,\frac{ds}s\\
&&\sim \dint_t^\fz\lf(\sqrt{s}\nabla^ke^{-s{L_i}}\r)\circ\lf(
e^{-s{L_i}}-e^{-(s+t){L_i}}\r)^M\circ\lf(e^{-s{L_i}}\r)f\,\frac{ds}s.
\end{eqnarray*}
Using the analytic property of semigroups and Lemma \ref{l3.1}, we
conclude that for all $g\in L^2(\rn)$ supported in the closed set
$E$ and $t<s$,
\begin{eqnarray*}
\lf\|\lf[e^{-s{L_i}}-e^{-(s+t){L_i}}\r]g\r\|_{L^2(F)}
&&=\lf\|-\dint_0^t
\frac{\pat}{\pat r}\lf(e^{-(s+r){L_i}}\r)g\,dr\r\|_{L^2(F)}\\
&&\ls\dint_0^t\lf\|(s+r){L_i}e^{-(s+r){L_i}}g\r\|_{L^2(F)}\frac{dr}
{s+r}\\
&&\ls\dint_0^t\exp\lf\{-\frac{\wz C\lf[\dist(E,\,F)\r]^{2k/(2k-1)}}
{s^{1/(2k-1)}}\r\}\frac{dr}{s+r}\|g\|_{L^2(E)}\\
&&\ls\frac{t}{s}\exp\lf\{-\frac{\wz
C\lf[\dist(E,\,F)\r]^{2k/(2k-1)}} {s^{1/(2k-1)}}\r\}\|g\|_{L^2(E)}.
\end{eqnarray*}
Thus,
\begin{eqnarray}\label{6.5}
\lf\|\frac{s}{t}\lf[e^{-s{L_i}}-e^{-(s+t){L_i}}\r]g\r\|_{L^2(F)} \ls
\exp\lf\{-\frac{\wz C\lf[\dist(E,\,F)\r]^{2k/(2k-1)}}
{s^{1/(2k-1)}}\r\}\|g\|_{L^2(E)}.
\end{eqnarray}
Therefore, from Minkowski's inequality, \eqref{6.5}, Lemma
\ref{l3.2}, Propositions \ref{p3.1} and \ref{p3.2}, and the change
of variables, we deduce that
\begin{eqnarray*}
\|\mathrm{O}\|_{L^2(F)} &&\ls\dint_t^\fz \lf\|\lf(\sqrt{s}
\nabla^ke^{-s{L_i}}\r)\circ\lf(\frac{s}{t}
\lf[e^{-s{L_i}}-e^{-(s+t){L_i}}\r]\r)^M\circ\lf(e^{-s{L_i}}\r)f\r\|_{L^2(F)}
\lf(\frac{t}{s}\r)^{M}\,\frac{ds}{s}\\
&&\ls\|f\|_{L^2(E)}\dint_t^\fz\lf(\frac{t}{s}\r)^{M}
\exp\lf\{-\frac{\wz C\lf[\dist(E,\,F)\r]^{2k/(2k-1)}}
{s^{1/(2k-1)}}\r\}\,\frac{ds}{s}\\
&&\ls\lf(\frac{\lf[\dist(E,\,F)\r]^{2k}}{t}\r)^{-M}\|f\|_{L^2(E)}.
\end{eqnarray*}
Combining this estimate with \eqref{6.4}, we have
\begin{eqnarray*}
\lf\|\nabla^k{L_i}^{-1/2}\lf(I-e^{-t{L_i}}\r)^Mf\r\|_{L^2(F)}
\ls\lf(\frac{\lf[\dist(E,\,F)\r]^{2k}}{t}\r)^{-M}\|f\|_{L^2(E)},
\end{eqnarray*}
that is, \eqref{6.1} holds.

Now, we prove \eqref{6.2}.  Using \eqref{6.3} and the change of
variables, we see that
\begin{eqnarray*}
\nabla^k{L_i}^{-1/2}\lf(t{L_i}e^{-t{L_i}}\r)^Mf&&=\frac{1}{2\sqrt{\pi}}
\dint_0^\fz \nabla^k e^{-s{L_i}}\lf(t{L_i}e^{-t{L_i}}\r)^Mf
\frac{ds}{\sqrt{s}}\\
&&\sim\dint_0^\fz\nabla^k e^{-(M+1)s{L_i}}\lf(t{L_i}e^{-t{L_i}}
\r)^Mf\frac{ds}{\sqrt{s}}\\
&&\sim\dint_0^t\nabla^k e^{-(M+1)s{L_i}}\lf(t{L_i}e^{-t{L_i}}
\r)^Mf\frac{ds} {\sqrt{s}}+\dint_t^\fz\cdots\\
&&=:\mathrm{B}+\mathrm{D}.
\end{eqnarray*}

By an application of the analytic property of semigroups,
Propositions \ref{p3.1} and \ref{p3.2}, and  Lemmas \ref{l3.1} and
\ref{l3.2}, we conclude that
\begin{eqnarray*}
\|\mathrm{B}\|_{L^2(F)}&&\ls\frac{1}{\sqrt{t}}\dint_0^t
\lf\|\lf(\sqrt{\frac t2}\nabla^k
e^{-\frac{t}{2}{L_i}}\r)\circ\lf(e^{-(M+1)s{L_i}}\r)\circ
\lf(\frac{t}{2}L_ie^{-\frac{t}{2}L_i}\r)\circ\lf(tL_i
e^{-tL_i}\r)^{M-1}f\r\|_{L^2(F)}\frac{ds}{\sqrt{s}}\\
&&\ls\frac{1}{\sqrt{t}}\dint_0^t \exp\lf\{-\frac{\wz
C\lf[\dist(E,\,F)\r]^{2k/(2k-1)}}{t^{1/(2k-1)}}\r\}
\|f\|_{L^2(E)}\frac{ds}{\sqrt{s}}\\
&&\ls\exp\lf\{-\frac{\wz
C\lf[\dist(E,\,F)\r]^{2k/(2k-1)}}{t^{1/(2k-1)}}\r\}
\|f\|_{L^2(E)}\ls\lf(\frac{t}{\lf[\dist(E,\,F)\r]^{2k}}\r)^{M}
\|f\|_{L^2(E)},
\end{eqnarray*}
where $C$ is a positive constant.

For the estimate of $\mathrm{D}$, similar to the estimate for
$\mathrm{B}$, we write
\begin{eqnarray*}
\mathrm{D}=\dint_t^\fz\lf(\sqrt{s}\nabla^ke^{-s{L_i}}\r)\circ
\lf(\frac{t}{s}\r)^M\circ\lf[sL_ie^{-(s+t){L_i}}\r]^Mf\,\frac{ds}s
\end{eqnarray*}
and we estimate $sL_ie^{-(s+t){L_i}}f$ by
\begin{eqnarray*}
\lf\|s{L_i}e^{-(s+t){L_i}}f\r\|_{L^2(F)} &&=\lf\|\frac{s}{t}
e^{-s{L_i}}\dint_0^t\frac{\pat}{\pat
r}\lf(r{L_i}e^{-r{L_i}}\r)f\,dr\r\|_{L^2(F)}\\
&&\ls\lf\|\frac{s}{t}
e^{-s{L_i}}\dint_0^t\lf[L_ie^{-r{L_i}}f-r{L_i}^2e^{-r{L_i}}
f\r]\,dr\r\|_{L^2(F)}\\
&&\ls\lf\|\frac{s}{t}
e^{-s{L_i}}\dint_0^tL_ie^{-r{L_i}}f\,dr\r\|_{L^2(F)}+\lf\|\frac{s}{t}
e^{-s{L_i}}\dint_0^tr{L_i}^2e^{-r{L_i}}f\,dr\r\|_{L^2(F)}\\
&&=:\mathrm{V}_1+\mathrm{V}_2.
\end{eqnarray*}
By Minkowski's inequality, Lemma \ref{l3.1} and $r<t<s$, we
conclude that
\begin{eqnarray*}
\mathrm{V}_1&&\ls\frac{s}{t}\dint_0^t\lf\|{L_i}
e^{-(s+r){L_i}}(f)\r\|_{L^2(F)}
\,dr\ls\frac{s}{t}\dint_0^t\exp\lf\{-\frac{\wz
C\lf[\dist(E,\,F)\r]^{2k/(2k-1)}}
{(s+r)^{1/(2k-1)}}\r\}\|f\|_{L^2(E)}\frac{dr}{s}\\
&&\ls \exp\lf\{-\frac{\wz C\lf[\dist(E,\,F)\r]^{2k/(2k-1)}}
{s^{1/(2k-1)}}\r\}\|f\|_{L^2(E)}.
\end{eqnarray*}
Similarly, we see that
\begin{eqnarray*}
\mathrm{V}_2
&&\ls\frac{s}{t}\dint_0^t\lf\|\lf[(r+s){L_i}\r]^2e^{-(r+s){L_i}}
f\r\|_{L^2(F)}\frac{dr}{r+s}\\
&&\ls\frac{s}{t}\|f\|_{L^2(E)} \dint_0^t\exp\lf\{-\frac{\wz
C\lf[\dist(E,\,F)\r]^{2k/(2k-1)}}
{(r+s)^{1/(2k-1)}}\r\}\frac{dr}{r+s}\\
&&\ls \exp\lf\{-\frac{\wz C\lf[\dist(E,\,F)\r]^{2k/(2k-1)}}
{s^{1/(2k-1)}}\r\}\|f\|_{L^2(E)},
\end{eqnarray*}
which, together with the estimate of $\mathrm{V}_1$, shows that
the family $\{sL_ie^{-(s+t){L_i}}\}_{t>0}$ of operators satisfies
the $k$-Davies-Gaffney estimate in $s$. Thus, from Minkowski's
inequality, Lemmas \ref{l3.1} and \ref{l3.2}, Propositions
\ref{p3.1} and \ref{p3.2}, and the change of variables, we deduce
that
\begin{eqnarray*}
\|\mathrm{D}\|_{L^2(F)}&&\ls\dint_t^\fz \lf\|\lf(\sqrt{s}
\nabla^ke^{-s{L_i}}\r)\circ\lf(\frac{t}{s}\r)^M\circ
\lf(s{L_i}e^{-(s+t){L_i}}
\r)^Mf\r\|_{L^2(F)}\,\frac{ds}s\\
&&\ls\|f\|_{L^2(E)}\dint_t^\fz\lf(\frac{t}{s}\r)^M\exp\lf\{-\frac{\wz
C\lf[\dist(E,\,F)\r]^{2k/(2k-1)}}
{s^{1/(2k-1)}}\r\}\,\frac{ds}s\\
&&\ls\lf(\frac{t}{\lf[\dist(E,\,F)\r]^{2k}}\r)^{M}\|f\|_{L^2(E)}.
\end{eqnarray*}
Combining the estimates for $\mathrm{B}$ and $\mathrm{D}$, we see
that
\begin{eqnarray*}
\lf\|\nabla^k{L_i}^{-1/2}\lf(t{L_i}e^{-t{L_i}}\r)^Mf\r\|_{L^2(F)}\ls
\lf(\frac{t}{\lf[\dist(E,\,F)\r]^{2k}}\r)^{M}\|f\|_{L^2(E)},
\end{eqnarray*}
which shows that \eqref{6.2} also holds. This finishes the proof of
Lemma \ref{l6.1}.
\end{proof}

With the help of Lemma \ref{l6.1}, we show that the Riesz transform
$\nabla^k (L_i^{-1/2})$ is bounded from $H_L^p(\rn)$ to the
classical Hardy space $H^p(\rn)$, which when $p=1$, $i=2$ and $k=1$
was first obtained in \cite{hlmmy}.

\begin{theorem}\label{t6.1}
Let $k\in\nn$, $p\in(n/(n+k),\,1]$, $L_1$ be the $2k$-order
divergence form homogeneous elliptic operator with complex bounded
measurable coefficients and $L_2$ the $2k$-order Schr\"odinger type
operator. Then, for all $i\in\{1,\,2\}$, the Riesz transform
$\nabla^k (L_i^{-1/2})$ is bounded from $H_{L_i}^p(\rn)$ to the
classical Hardy space $H^p(\rn)$.
\end{theorem}

\begin{proof}  Let $i\in\{1,\,2\}$.
We first claim that to prove Theorem \ref{t6.1}, it suffices to show
that $\nabla^k (L_i^{-1/2})$ maps each
$(H_{L_i}^p,\,\ez,\,M)$-molecule $m$ as in Definition \ref{d4.2}
with $\ez>0$ and $M>n(1/p-1/2)/(2k)$ into a classical
$H^p(\rn)$-molecule in \cite{tw} up to a harmless constant multiple.

Indeed, assume this claim for the moment. For any $f\in
\mathbb{H}_{L_i}^p(\rn)$, by Theorem \ref{t4.1}, there exist
$\{\lz_j\}_{j=0}^\fz\in l^p$ and a sequence $\{m_j\}_{j=0}^\fz$ of
$(H_{L_i}^p,\,\ez,\,M)$-molecules such that
$f=\sum_{j=0}^\fz\lz_jm_j$ is a molecular
$(H_L^p,\,2,\,\ez,\,M)$-representation of $f$ and
$$\|f\|_{H_{L_i}^p(\rn)}\sim\lf(\sum_{j=0}^\fz|\lz_j|^p\r)^{1/p}.$$
Moreover, from the $L^2(\rn)$-boundedness of $\nabla^k(L_i^{-1/2})$
and the fact that $f=\sum_{j=0}^\fz\lz_jm_j$ holds in $L^2(\rn)$, it
follows that
\begin{eqnarray}\label{6.6}
\nabla^k(L_i^{-1/2})f=\nabla^k(L_i^{-1/2})\lf(\dsum_{j=0}^\fz
\lz_jm_j\r)=\dsum_{j=0}^\fz\lz_j\nabla^k(L_i^{-1/2})m_j
\end{eqnarray}
in $L^2(\rn)$ and hence in the {\it space $\mathcal{S}'(\rn)$} of
Schwartz distributions, which, together with the above claim,
implies that \eqref{6.6} is a classical molecular decomposition of
$\nabla^k(L_i^{-1/2})f$ in $H^p(\rn)$. Thus, by the molecular
characterization of $H^p(\rn)$ in \cite{tw}, we further conclude
that
\begin{eqnarray*}
\lf\|\nabla^k(L_i^{-1/2})f\r\|_{H^p(\rn)}\ls
\lf(\dsum_{j=0}^\fz|\lz_j|^p\r)^{1/p}\sim\|f\|_{H_{L_i}^p(\rn)},
\end{eqnarray*}
which, combined with a density argument, then shows that
$\nabla^k(L_i^{-1/2})$ is bounded from $H_{L_i}^p(\rn)$ to
$H^p(\rn)$.

Let $m$ be an $(H_{L_i}^p,\,\ez,\,M)$-molecule associated with the
cube $Q$ as in Definition \ref{d4.2} with $\ez\in(0,\fz)$ and
$M>n(1/p-1/2)/(2k)$. To prove the above claim, we need to prove
that $\nabla^k(L_i^{-1/2})m$ is a classical $H^p(\rn)$-molecule in
\cite{tw} up to a harmless constant multiple. To this end, we only
need to show that $\nabla^k(L_i^{-1/2})m$ is a following defined
$H^p(\rn)$-molecule in \cite{hyz,hmm}, from which it follows that
it is also a classical molecule in \cite{tw}. In what follows, for
any $\gz\in\rr$, we denote by $\lfloor\gz\rfloor$ the {\it maximal
integer not more than} $\gz$. Let $p\in(0,\,1]$ and $Q$ be a cube
in $\rn$. A function $\wz m\in L^2(\rn)$ is called an {\it
$H^p(\rn)$-molecule associated with $Q$} if there exists a
positive constant $\ez\in(0,\,\fz)$ such that
\begin{enumerate}
\item[(i)] for all $j\in\zz_+$,
\begin{eqnarray}\label{6.7}
\|\wz m\|_{L^2(S_j(Q))}\ls\lf[2^jl(Q)\r]^{n(1/p-1/2)}2^{-j\ez};
\end{eqnarray}
\item[(ii)] there exists a non-negative integer
$M\in\zz_+$ with $M\ge\lfloor n(1/p-1)\rfloor$ such that for all
multi-indices $\az$ with $0\le|\az|\le M$,
\begin{eqnarray}\label{6.8}
\dint_{\rn} x^{\az} \wz  m(x)\,dx=0.
\end{eqnarray}
\end{enumerate}

We first prove that $\nabla^k (L_i^{-1/2})m$ satisfies
\eqref{6.7}. For all  $j\in\{0,\,1,\,2\}$, by the
$L^2(\rn)$-boundedness of $\nabla^k ({L_i}^{-1/2})$ and
\eqref{4.5}, we see that
\begin{eqnarray*}
\hs\hs\hs\lf\|\nabla^k
({L_i}^{-1/2})m\r\|_{L^2(S_j(Q))}\ls\lf\|\nabla^k
({L_i}^{-1/2})m\r\|_{L^2(\rn)}\ls\|m\|_{L^2(\rn)}\ls|Q|^{1/2-1/p}.
\end{eqnarray*}
When $j>2$, we write
\begin{eqnarray*}
\lf\|\nabla^k ({L_i}^{-1/2})m\r\|_{L^2(S_j(Q))}&&\le\lf\|\nabla^k
({L_i}^{-1/2})\lf(I-e^{-[l(Q)]^{2k}L_i}\r)^Mm\r\|_{L^2(S_j(Q))}\\
&&\hs+\lf\|\nabla^k({L_i}^{-1/2})\lf[I-\lf(I-e^{-[l(Q)]^{2k}L_i}\r)^M\r]m
\r\|_{L^2(S_j(Q))}\\
&&=:\mathrm{I}+\mathrm{O}.
\end{eqnarray*}
From an application of Lemma \ref{l6.1} and \eqref{4.5}, it
follows that
\begin{eqnarray*}
\mathrm{I}&&\ls\lf\|\nabla^k
({L_i}^{-1/2})\lf(I-e^{-[l(Q)]^{2k}L_i}\r)^M(m\chi_{2^{j-2}Q})\r
\|_{L^2(S_j(Q))}\\
&&\hs+\lf\|\nabla^k ({L_i}^{-1/2})\lf(I-e^{-[l(Q)]^{2k}L_i}\r)^M
(m\chi_{\rn\setminus(2^{j+1}Q)})\r\|_{L^2(S_j(Q))}\\
&&\hs+\lf\|\nabla^k ({L_i}^{-1/2})\lf(I-e^{-[l(Q)]^{2k}L_i}\r)^M
(m\chi_{2^{j+1}Q\setminus(2^{j-2}Q)})\r\|_{L^2(S_j(Q))}\\
&&\ls\lf[\frac{\dist(S_j(Q),\,2^{j-2}Q)}{l(Q)}\r]^{2kM}\lf\|
m\chi_{2^{j-2}Q}\r\|_{L^2(\rn)}\\
&&\hs+\lf[\frac{\dist(S_j(Q),\,\rn\setminus(2^{j+1}Q))}{l(Q)}\r]^{2kM}
\lf\|m\chi_{\rn\setminus(2^{j+1}Q)}\r\|_{L^2(\rn)}
+\lf\| m\chi_{2^{j+1}Q\setminus(2^{j-1}Q)}\r\|_{L^2(\rn)}\\
&&\ls2^{-2jkM}\lf[l(Q)\r]^{n(1/2-1/p)}+\lf[2^jl(Q)\r]^{n(1/2-1/p)}
2^{-j\ez}.
\end{eqnarray*}
Let $\wz\ez:=\min\{\ez,\,2kM-n(1/p-1/2)\}>0$. We then have
\begin{eqnarray}\label{6.9}
\mathrm{I}\ls \lf[2^jl(Q)\r]^{n(1/2-1/p)} 2^{-j\wz\ez}.
\end{eqnarray}
To estimate $\mathrm{O}$, from Lemma \ref{l6.1} and \eqref{4.2}, we
deduce that
\begin{eqnarray*}
\mathrm{O} &&\ls\dsup_{1\le \ell\le
M}\lf\|\nabla^k{L_i}^{-1/2}e^{-\ell[l(Q)]^{2k}{L_i}}m\r\|_{L^2(S_j(Q))}\\
&&\sim\dsup_{1\le \ell\le
M}\lf\|\nabla^k{L_i}^{-1/2}\lf(\frac{\ell}{M}[l(Q)]^{2k}
{L_i}e^{-\frac{\ell}{M}[l(Q)]^{2k}{L_i}}\r)^M\lf([l(Q)]^{-2k}{L_i}^{-1}\r)^M
m\r\|_{L^2(S_j(Q))}\\
&&\sim\dsup_{1\le \ell\le
M}\lf\|\nabla^k{L_i}^{-1/2}\lf(\frac{\ell}{M}[l(Q)]^{2k}
{L_i}e^{-\frac{\ell}{M}[l(Q)]^{2k}{L_i}}\r)^M\lf[
\chi_{2^{j-2}Q}\lf([l(Q)]^{-2k}{L_i}^{-1}\r)^M\r]
m\r\|_{L^2(S_j(Q))}\\
&&\hs+\dsup_{1\le \ell\le
M}\lf\|\nabla^k{L_i}^{-1/2}\lf(\frac{\ell}{M}[l(Q)]^{2k}
{L_i}e^{-\frac{\ell}{M}[l(Q)]^{2k}{L_i}}\r)^M\lf[\chi_{\rn
\setminus2^{j+1}Q}
\lf([l(Q)]^{-2k}{L_i}^{-1}\r)^M\r] m\r\|_{L^2(S_j(Q))}\\
&&\hs+\dsup_{1\le\ell\le M}\lf\|\nabla^k{L_i}^{-1/2}\lf(
\frac{\ell}{M}[l(Q)]^{2k}
{L_i}e^{-\frac{\ell}{M}[l(Q)]^{2k}{L_i}}\r)^M
\lf[\chi_{2^{j+1}Q\setminus2^{j-2}Q}\lf([l(Q)]^{-2k}{L_i}^{-1}\r)^M\r]m
\r\|_{L^2(S_j(Q))}\\
&&\ls 2^{-2jkM}\lf\|\lf([l(Q)]^{-2k}L_i^{-1}\r)^Mm\r\|_{L^2(\rn)}
+\lf\|\chi_{2^{j+1}Q\setminus2^{j-1}Q}
\lf([l(Q)]^{-2k}L_i^{-1}\r)^Mm\r\|_{L^2(\rn)}\\
&&\ls 2^{-2jkM}\lf\{\dsum_{\wz k=0}^\fz\lf\|
\lf([l(Q)]^{-2k}L_i^{-1}\r)^M m\r\|^2_{L^2(S_{\wz k}(Q))}\r\}^{1/2}
+\lf[2^jl(Q)\r]^{n(1/2-1/p)}2^{-j\ez}\\
&&\ls 2^{-2jkM}\lf\{\dsum_{\wz k=0}^\fz 2^{-2\wz
k[\ez+n(1/p-1/2)]}\r\}^{1/2}[l(Q)]^{n(1/2-1/p)}
+\lf[2^jl(Q)\r]^{n(1/2-1/p)}2^{-j\ez}\\
&&\ls2^{-j\wz\ez}\lf[2^jl(Q)\r]^{n/p-2/p},
\end{eqnarray*}
which, together with \eqref{6.9}, implies that $\nabla^k
(L_i^{-1/2})m$ satisfies \eqref{6.7} with $\ez$ therein replaced by
$\wz\ez$.

Now, we prove that $\nabla^k (L_i^{-1/2})m$ satisfies \eqref{6.8}
by borrowing some ideas from the proof of Theorem 7.4 in
\cite{jy10}. Let $D(\sqrt{{L_i}})$ be the \emph{domain} of
$\sqrt{{L_i}}$ and $R({L_i}^{-1/2})$ the \emph{range} of
${L_i}^{-1/2}$. From \cite{ahmt,ou}, it follows that
$D(\sqrt{L_i})=D(\mathfrak{a}_i)$, where $D(\mathfrak{a}_i)\subset
 W^{k,\,2}(\rn)$ is the
domain of the sesquilinear form associated with $L_i$, which implies
that $R({L_i}^{-1/2})\subset W^{k,\,2}(\rn)$. Let
$\{\vz_j\}_{j=1}^\fz\subset C_\text{c}^\fz(\rn)$ such that
\begin{enumerate}
\item[(i)] $\sum_{j=1}^\fz\vz_j(x)=1$ for almost every $x\in\rn$;
\item[(ii)] for each $j\in\nn$, there exists a ball $B_j\subset\rn$
such that $\supp \vz_j\subset 2B_j$, $\vz_j:=1$ on $B_j$ and
$0\le\vz_j\le1$;
\item[(iii)] there exists a positive constant $C_\vz$ such that for
all $j\in\nn$ and $x\in\rn$,
$$\sum_{\ell=1}^k|\nabla^\ell\vz_j(x)|\le C_\vz;$$
\item[(iv)] there exists $N_\vz\in\nn$ such that $\sum_{j=1}^\fz
\chi_{2B_j}\le N_\vz$.
\end{enumerate}
For all $j\in\nn$ and multi-indices $\az$, let $\eta_j\in
C_\text{c}^\fz(\rn)$ such that $\eta_j:=1$ on $2B_j$ and $\supp
\eta_{j}\subset 4B_j$. Since $R({L_i}^{-1/2})\subset
W^{k,\,2}(\rn)$ and $\eta_j\,x^\az\in C_\text{c}^\fz(\rn)$, we
conclude that
\begin{eqnarray*}
\dint_{\rn}x^{\az}\nabla^{k}{L_i}^{-1/2}m(x)\,dx &&=\dint_\rn x^\az
\nabla^{k-1}\lf(\dsum_{j=1}^\fz\vz_j\nabla {L_i}^{-1/2}\r)m(x)\,dx\\
&&=\dsum_{j=1}^\fz\dint_\rn x^\az\nabla^{k-1}
\lf(\vz_j\nabla {L_i}^{-1/2}\r)m(x)\,dx\\
&&=\dsum_{j=1}^\fz\dint_\rn\eta_jx^\az\nabla^{k-1}
\lf(\nabla {L_i}^{-1/2}\r)m(x)\,dx\\
&&=\dsum_{j=1}^\fz(-1)^{k-1}\dint_{\rn}\lf(\nabla^{k-1}(\eta_jx^\az)\r)
\nabla({L_i}^{-1/2})m(x)\,dx.
\end{eqnarray*}
Thus, for all $|\az|\le k-1=n(1/[n/(n+k)]-1)$, we see that
\begin{eqnarray*}
\lf|\dint_{\rn}x^{\az}\nabla^{k}{L_i}^{-1/2}m(x)\,dx\r|
&&\le\dsum_{j=1}^\fz\lf|\dint_{\rn}\lf(\nabla^{k-1}(\eta_jx^\az)\r)
\nabla({L_i}^{-1/2}m(x)\,dx\r|\\
&&\le\dsum_{j=1}^\fz\lf|\dint_{\rn}\eta_j
\nabla({L_i}^{-1/2})m(x)\,dx\r|\\
&&=\dsum_{j=1}^\fz\lf|\dint_{\rn}\eta_j
\nabla(\vz_i{L_i}^{-1/2})m(x)\,dx\r|\\
&&=\dsum_{j=1}^\fz\lf|\dint_{\rn}\nabla(\eta_j)
\vz_i{L_i}^{-1/2}m(x)\,dx\r|=0,
\end{eqnarray*}
which implies that $\nabla^k (L_i^{-1/2})m$ satisfies \eqref{6.8}
with $p$ and $M$ respectively replaced by $n/(n+k)$ and
$n(1/[n/(n+k)]-1)$. Thus, $\nabla^k(L_i^{-1/2})m$ is a classical
$H^p(\rn)$ molecule in \cite{tw}, which completes the proof of
Theorem \ref{t6.1}.
\end{proof}

On the Hardy space $H_{L_1}^p(\rn)$, we further obtain its
characterization by the Riesz transforms $\nabla^k(L_1^{-1/2})$. To
this end, we first introduce some notions.

\begin{definition}\label{d6.1}
Let $p\in(0,\,1]$ and $L_1$ be the $2k$-order divergence form
homogenous elliptic operator with complex bounded measurable
coefficients. The {\it Riesz transform Hardy space}
$H_{L_1,\text{Riesz}}^p(\rn)$ is defined to be the completion of the
set
\begin{eqnarray*}
\mathbb{H}_{L_1,\text{Riesz}}^p(\rn):=\lf\{f\in L^2(\rn):\
\nabla^k(L_1^{-1/2})f\in H^p(\rn)\r\}
\end{eqnarray*}
with respect to the quasi-norm
$$\|f\|_{H_{L_1,\text{Riesz}}^p(\rn)}:=\lf\|\nabla^k(L_1^{-1/2})f
\r\|_{H^p(\rn)}$$ for all $f\in \mathbb{H}_{L_1,\riesz}^p(\rn)$.
\end{definition}

We also need the following notion of $L^p-L^q$ $k$-off-diagonal
estimates, which when $k=1$ previously appeared in \cite{a07} (see
also \cite{hmm}).

\begin{definition}\label{d6.2}
Let $k\in\nn$, $r,\,q\in(1,\,\fz)$ and $r\le q$. A family
$\{S_t\}_{t>0}$ of operators is said to satisfy the {\it $L^r-L^q$
$k$-off-diagonal estimate}, if there exist  positive constant $C$
and $\wz C$ such that for all closed sets $E$, $F\subset\rn$ and
$f\in L^r(\rn) \cap L^2(\rn)$ supported in $E$,
\begin{eqnarray*}
\|S_tf\|_{L^q(F)}\le Ct^{\frac{n}{2k}(\frac{1}{q}-\frac{1}{r})}
\exp\lf\{-\wz C\frac{\lf[\dist(E,\,F)\r]^{2k/(2k-1)}}
{t^{1/(2k-1)}}\r\}\|f\|_{L^r(E)}.
\end{eqnarray*}
\end{definition}

On the $L^r-L^q$ $k$-off-diagonal estimate of the $2k$-order
divergence form homogeneous elliptic operator $L_1$ with complex
bounded measurable coefficients, we have the following useful lemma.

\begin{lemma}\label{l6.2}
Let $L_1$ be the $2k$-order divergence form homogeneous elliptic
operator with complex bounded measurable coefficients and
$r\in(1,\,2]$ such that the semigroup $\{e^{-tL_1}\}_{t>0}$
satisfies the $L^r-L^2$ $k$-off-diagonal estimate. Then the family
$\{tL_1e^{-tL_1}\}_{t>0}$ of operators also satisfies the $L^r-L^2$
$k$-off-diagonal estimate.
\end{lemma}

\begin{proof}
By the analytical property of $\{e^{-tL_1}\}_{t>0}$, we have
$\{tL_1e^{-tL_1}\}_{t>0}=\{2(\frac{t}{2}L_1e^{-\frac{t}{2}L_1})
(e^{-\frac{t}{2}L_1})\}_{t>0}$. Since the $k$-Davies-Gaffney
estimate is just the $L^2-L^2$ $k$-off-diagonal estimate, it
follows, from Proposition \ref{p3.1} and Lemma \ref{l3.1}, that
$\{\frac{t}{2}L_1e^{-\frac{t}{2}L_1}\}_{t>0}$ satisfies the
$L^2-L^2$ $k$-off-diagonal estimate. Moreover, by the fact that
$\{e^{-\frac{t}{2}L_1}\}_{t>0}$ satisfies the $L^r-L^2$
$k$-off-diagonal estimate and an argument similar to the proof of
Lemma \ref{l3.2} with $\{A_t\}_{t>0}$ and $\{B_s\}_{s>0}$,
respectively, replaced by
$\{\frac{t}{2}L_1e^{-\frac{t}{2}L_1}\}_{t>0}$ and
$\{e^{-\frac{t}{2}L_1}\}_{t>0}$, we conclude that
$\{tL_1e^{-tL_1}\}_{t>0}$ also satisfies the $L^r-L^2$
$k$-off-diagonal estimate, which completes the proof of Lemma
\ref{l6.2}.
\end{proof}

\begin{proposition}\label{p6.1}
Let $L_1$ be the $2k$-order divergence form homogeneous elliptic
operator with complex bounded measurable coefficients and
$r\in(1,\,2]$ such that the semigroup $\{e^{-tL_1}\}_{t>0}$
satisfies the $L^r-L^2$ $k$-off-diagonal estimate. Then for all
$p\in(0,\,1]$ such that  $p>rn/(n+kr)$ and
$h\in\mathbb{H}_{L_1,\riesz}^p(\rn)$,
\begin{eqnarray*}
\|h\|_{H_{L_1}^p(\rn)}\le C\|\nabla^k{L_1}^{-1/2}h\|_{H^p(\rn)}.
\end{eqnarray*}
\end{proposition}

To prove Proposition \ref{p6.1}, we need to recall some results
concerning the homogenous Hardy-Sobolev space $\dot H^{k,p}(\rn)$
(see, for example, \cite{ck,hpw,tr,wz}).

\begin{definition}\label{d6.3}
Let $k\in\nn$ and $p\in(0,\,1]$. The {\it homogeneous Hardy-Sobolev
space} $\dot H^{k,p}(\rn)$ is defined to be the space
\begin{eqnarray*}
\dot
H^{k,p}(\rn):=\lf\{f\in\mathcal{S}'(\rn)/\mathscr{P}_{k-1}(\rn):\
\|f\|_{\dot H^{k,p}(\rn)}:=\dsum_{|\sz|=k}\|\pat^\sz f\|_{H^p(\rn)}
<\fz\r\},
\end{eqnarray*}
where $\mathcal{S}'(\rn)$ denotes the \emph{space of all Schwartz
distributions on $\rn$} and $\mathscr{P}_{k-1}(\rn)$ the
\emph{class of all polynomials of order strictly less than $k$ on
$\rn$}.
\end{definition}

Let $\ell\in\nn$ be fixed. Let $\mathcal{S}(\rn)$ denote the {\it
space of all Schwartz functions on $\rn$} and
$\phi\in\mathcal{S}(\rn)$ such that
\begin{enumerate}
\item[(i)] $\phi$ is radial, $\supp{\phi}\subset \{x\in\rn:\
|x|<1\}$ and for all $\xi\neq0$,
$\int_0^\fz|\widehat\phi(t\xi)|^2\,\frac{dt}{t}=1$, where
$\widehat\phi$ denotes the \emph{Fourier transform} of $\phi$,
\item[(ii)] for all $|\gz|\le\ell$, $\int_\rn x^\gz\phi(x)\,dx=0$.
\end{enumerate}
For any given $\phi\in\mathcal{S}(\rn)$ as above and all $f\in
\mathcal{S'}(\rn)$, let $Q_tf:=\phi_t*f$, where $\phi_t:=
t^{-n}\phi(x/t)$ for all $t\in(0,\,\fz)$ and $x\in\rn$. Let $p$,
$q\in(0,\,\fz)$ and $\az\in\rr$ such that $|\az|<\ell+1$. The {\it
homogenous Triebel-Lizorkin space} $\dot F^\az_{p,q}(\rn)$ is
defined to be the space
\begin{eqnarray*}
\dot F^\az_{p,q}(\rn)&&
:=\Bigg\{f\in\mathcal{S}'(\rn)/\mathscr{P}(\rn):
\ \|f\|_{\dot F^\az_{p,q}(\rn)}:=\lf.\lf\|\lf\{\dint_0^\fz
\lf(t^{-\az}|Q_tf|\r)^q
\,\frac{dt}{t}\r\}^{1/q}\r\|_{L^p(\rn)} <\fz\r\},
\end{eqnarray*}
where $\mathscr{P}(\rn)$ denotes the {\it class of all polynomials
on $\rn$} (see, for example, \cite{hpw,tr,wz}).

Let $\dot W^{k,2}(\rn)$ for $k\in\nn$ denote the {\it homogenous
Sobolev space of order $k$} endowed with the {\it norm}
$\|\cdot\|_{\dot W^{k,2}(\rn)} := \|\nabla^k(\cdot)\|_{L^2(\rn)}$.
It is known that the homogeneous Sobolev space $\dot W^{k,2}(\rn)$
and Hardy-Sobolev space $\dot H^{k,p}(\rn)$  coincide, respectively,
with the Triebel-Lizorkin space $\dot F^{k}_{2,2}(\rn)$ and $\dot
F^{k}_{p,2}(\rn)$ with equivalent norms (see, for example,
\cite[p.\,242]{tr}).

\begin{definition}\label{d6.4}
Let $k\in\nn$, $\ell\ge k$ be any fixed positive integer and
$p\in(0,\,1]$. A function $b$ is called an $\dot H^{k,p}(\rn)$-{\it
atom} if it satisfies that
\begin{enumerate}
\item[(i)] there exists a ball $B\subset\rn$ such that $\supp b\subset
B$,
\item[(ii)] for any $|\gz|\le\ell$,
$\int_\rn x^\gz b(x)\,dx=0$,
\item[(iii)]
\begin{eqnarray}\label{6.10}
\|b\|_{\dot F^{k}_{2,2}(\rn)}\le|B|^{1/2-1/p}.
\end{eqnarray}
\end{enumerate}
\end{definition}

\begin{lemma}\label{l6.3}
Let $p\in(0,\,1]$, $k\in\nn$ and $f\in \dot W^{k,2}(\rn)\cap \dot
H^{k,p}(\rn)$. Then there exist $\{\lz_j\}_{j=0}^\fz\in l^p$ and a
sequence $\{b_j\}_{j=0}^\fz$ of $\dot H^{k,p}(\rn)$-atoms such that
$f=\sum_{j=0}^\fz\lz_jb_j$ in $\dot W^{k,2}(\rn)\cap\dot
H^{k,p}(\rn)$, and $\|f\|_{\dot H^{k,p}(\rn)}\sim
\{\sum_{j=0}^\fz|\lz_j|^p\}^{1/p}.$
\end{lemma}

\begin{proof}
For any $f\in\dot W^{k,2}(\rn)\cap\dot H^{k,p}(\rn)$, by the
coincidence of Sobolev spaces and Hardy-Sobolev spaces with
Triebel-Lizorkin spaces, we know that $f\in \dot
F^k_{2,2}(\rn)\cap\dot F^k_{p,2}(\rn)$. From this and a slight
modification on the proof of \cite[Proposition 4.3]{wz}, together
with the same observation as in Theorem \ref{t4.2} on the
convergence of the atomic decomposition for elements in the tent
spaces, we deduce all the desired conclusions of Lemma \ref{l6.3},
which completes the proof of Lemma \ref{l6.3}.
\end{proof}

\begin{proof}[Proof of Proposition \ref{p6.1}]
For all $g\in L^2(\rn)$, define the operator $S_1$ by setting, for
all $x\in\rn$,
\begin{eqnarray*}
S_1g(x):=\lf\{\iint_{\Gamma(x)}\lf|t^k\sqrt{L_1}e^{-t^{2k}L_1}
g(y)\r|^2\frac{dy\,dt}{t^{n+1}}\r\}^{1/2}.
\end{eqnarray*}
For all $h\in\mathbb{H}_{\riesz,\,L_1}^p(\rn)$, let $f:=
L_1^{-1/2}h$. Then $f\in \dot W^{k,2} (\rn)\cap \dot H^{k,p}(\rn)$
and, by Lemma \ref{l6.3}, there exist $\{\lz_j\}_{j=0}^\fz\in l^p$
and a sequence $\{b_j\}_{j=0}^\fz$ of $\dot H^{k,p}(\rn)$-atoms such
that $f=\sum_{j=0}^\fz\lz_jb_j$ in $\dot W^{k,2}(\rn)\cap\dot
H^{k,p}(\rn)$ and, moreover, $(\sum_{j=0}^\fz|\lz_j|^p)^{1/p}\sim
\|f\|_{\dot H^{k,p}(\rn)}$. By Theorem \ref{t5.1} with $L$ replaced
by $L_1$, to show Proposition \ref{p6.1}, we only need to prove that
for all $f\in\dot W^{k,2}(\rn)\cap\dot H^{k,p}(\rn)$ with
$p\in(nr/(n+kr),\,1]$,
\begin{eqnarray}\label{6.11}
\lf\|S_1\sqrt{L_1}f\r\|_{L^p(\rn)}\ls\|f\|_{\dot H^{k,p}(\rn)}.
\end{eqnarray}
To prove \eqref{6.11}, it suffices to prove that for all $\dot
H^{k,p}(\rn)$-atoms $b$,
\begin{eqnarray}\label{6.12}
\lf\|S_1\sqrt{L_1}b\r\|_{L^p(\rn)}\ls1.
\end{eqnarray}
Indeed, if \eqref{6.12} holds, by the $L^2(\rn)$-boundedness of
$S_1$ which is deduced from \eqref{4.3}, and \cite[Theorem
1.1]{ahlmt}, we conclude that
\begin{eqnarray*}
\lf\|S_1\sqrt{L_1}f\r\|_{L^2(\rn)}\ls\lf\|\sqrt{L_1}f\r\|_{L^2(\rn)}
\sim\lf\|\nabla^kf\r\|_{L^2(\rn)}\sim\|f\|_{\dot W^{k,2}(\rn)},
\end{eqnarray*}
which, together with an argument similar to the proof of
\eqref{5.11}, yields that for almost every $x\in\rn$,
$|S_1\sqrt{L_1}f(x)|\le\sum_{j=0}^\fz|\lz_jS_1\sqrt{L_1}b_j(x)|$.
This, combined with \eqref{6.12}, shows that \eqref{6.11} is valid.

We now prove \eqref{6.12}. For $j\in\nn$, let $\mathcal{R}(S_j(Q))
:=\cup_{x\in S_j(Q)}\Gamma(x)$ be the {\it saw-tooth region} based
on $S_j(Q)\subset\rn$. From Minkowski's inequality, H\"older's
inequality and Fubini's theorem, we infer that
\begin{eqnarray*}
\lf\|S_1\sqrt{L_1}b\r\|_{L^p(\rn)}^p &&\ls\dsum_{j=0}^\fz
\lf\|S_1\sqrt{L_1}b\r\|_{L^p(S_j(Q))}^p\\
&&\ls \lf\|S_1\sqrt{L_1}b\r\|_{L^2(4Q)}^p
\lf|Q\r|^{n(\frac{1}{p}-\frac{1}{2})p}
+\dsum_{j=3}^\fz\lf\|S_1\sqrt{L_1}b\r\|_{L^2(S_j(Q))}^p
\lf|2^jl(Q)\r|^{n(\frac{1}{p}-\frac{1}{2})p}\\
&&\ls\lf\|S_1\sqrt{L_1}b\r\|_{L^2(4Q)}^p
\lf|Q\r|^{(\frac{1}{p}-\frac{1}{2})p}\\
&&\hs+\dsum_{j=3}^\fz\lf\{\iint_{\mathcal{R}
(S_j(Q))}\lf|t^{2k}L_1e^{-t^{2k}L_1}
b(y)\r|^2\frac{dy\,dt}{t^{2k+1}}\r\}^{p/2}
\lf|2^jl(Q)\r|^{n(\frac{1}{p}-\frac{1}{2})p}\\
&&\ls\lf\|S_1\sqrt{L_1}b\r\|_{L^2(4Q)}^p
\lf|Q\r|^{(\frac{1}{p}-\frac{1}{2})p}\\
&&\hs +\dsum_{j=3}^\fz\lf\{
\dint_{2^{j-2}Q}\dint_{2^{j-3}l(Q)}^\fz\lf|t^{2k}L_1
e^{-t^{2k}L_1}b(y)\r|^2\frac{dy\,dt}{t^{2k+1}}\r\}^{p/2}
\lf|2^jl(Q)\r|^{n(\frac{1}{p}-\frac{1}{2})p}\\
&&\hs+\dsum_{j=3}^\fz\lf\{ \dint_{\rn\setminus
2^{j-2}Q}\dint_0^\fz\cdots\r\}^{p/2}
\lf|2^jl(Q)\r|^{n(\frac{1}{p}-\frac{1}{2})p}\\
&&=:\mathrm{I}+\dsum_{ j=3}^\fz(\mathrm{J}_{j})^p
+\dsum_{j=3}^\fz(\mathrm{V}_{j})^p.
\end{eqnarray*}

For $\mathrm{I}$, by the $L^2(\rn)$-boundedness of $S_1$,
\eqref{6.10} and \cite[Theorem 1.1]{ahmt}, we see that
\begin{eqnarray}\label{6.13}
\mathrm{I}\ls\lf\|\sqrt {L_1}b\r\|_{L^2(\rn)}^p
|Q|^{(\frac{1}{p}-\frac{1}{2})p}\ls \|b\|_{\dot
F^{k}_{2,2}(\rn)}^p|Q|^{(\frac{1}{p}-\frac{1}{2})p}\ls1.
\end{eqnarray}

To estimate $\mathrm{J}_j$,
 recall the following {\it embedding theorem}
(see, for example, \cite{tr}) that for all $f\in \dot
F^k_{\frac{nr}{n+kr},2}(\rn)$,
\begin{eqnarray}\label{6.14}
\|f\|_{L^{r}(\rn)}\ls\lf\|\nabla^kf\r\|_{L^{\frac{nr}{n+kr}}(\rn)}.
\end{eqnarray}

 For each $\mathrm{O}_j$, from  Minkowski's inequality, Lemma \ref{l6.2},
 \eqref{6.14}, Lemma \ref{l6.2}, H\"older's inequality and \eqref{6.10},
 it follows that
\begin{eqnarray*}
\mathrm{J}_j&&\ls\lf\{\dint_{2^{j-3}l(Q)}^\fz\lf\|t^{2k}
L_1e^{-t^{2k}L_1}b
\r\|_{L^2(2^{j-2}Q)}^2\frac{dt}{t^{1+2k}}\r\}^{1/2}
\lf|2^jl(Q)\r|^{n(\frac{1}
{p}-\frac{1}{2})}\\
&&\ls\lf\{\dint_{2^{j-3}l(Q)}^\fz
t^{2n(\frac{1}{2}-\frac{1}{r})}\frac{dt}{t^{1+2k}}\r\}^{1/2}
\lf|2^jl(Q)\r|^{n(\frac{1}{p}-\frac{1}{2})}\|b\|_{L^r(Q)}\\
&&\ls\lf[2^jl(Q)\r]^{n(\frac{1}{p}-\frac{1}{r})-k}\|b\|_{L^r(Q)}\ls
\lf[2^jl(Q)\r]^{n(\frac{1}{p}-\frac{1}{r})-k}
\lf\|\nabla^kb\r\|_{L^{\frac{rn}{n+kr}}(Q)}\\
&&\ls\lf[2^jl(Q)\r]^{n(\frac{1}{p}-\frac{1}{r})-k}
\lf\|\nabla^kb\r\|_{L^{2}(Q)}|l(Q)|^{\frac{n+kr}{r}-\frac{n}{2}}
\ls2^{[n(\frac{1}{p}-\frac{1}{r})-k]j}.
\end{eqnarray*}
Let $\az:= \frac{n}{r}+k-\frac{n}{p}$. Since
$p\in(\frac{nr}{n+kr},\,1]$, we then have $\az>0$ and
\begin{eqnarray}\label{6.15}
\sum_{j=3}^\fz(\mathrm{J}_j)^p\ls\sum_{j=3}^\fz2^{-\az jp} \ls 1.
\end{eqnarray}

To estimate $\mathrm{V}_j$, we write
\begin{eqnarray*}
\mathrm{V}_j
&&\ls\lf[2^jl(Q)\r]^{n(\frac{1}{p}-\frac{1}{2})}\lf\{\dint_{\rn\setminus
2^{j-2}Q}\dint_{0}^{2^{j-3}l(Q)}\lf|t^{2k}L_1
e^{-t^{2k}L_1}b(y)\r|^2\frac{dy\,dt}{t^{2k+1}}\r\}^{1/2}\\
&&\hs+
\lf[2^jl(Q)\r]^{n(\frac{1}{p}-\frac{1}{2})}\lf\{\dint_{\rn\setminus
2^{j-2}Q}\dint_{2^{j-3}l(Q)}^{\fz}\cdots\r\}^{1/2}\\
&&=:\mathrm{V}_{j,1}+\mathrm{V}_{j,2}.
\end{eqnarray*}
Similar to the estimate of $\mathrm{J}_j$, we see that
\begin{eqnarray}\label{6.16}
\mathrm{V}_{j,2}\ls2^{-\az j}.
\end{eqnarray}
To estimate $\mathrm{V}_{j,1}$, let $\bz\in(2k+2n(1/r-1/2),\,\fz)$.
By Lemma \ref{l6.2}, \eqref{6.10} and \eqref{6.14}, there exists a
positive constant $\wz C$ such that
\begin{eqnarray*}
\mathrm{V}_{j,1}&&\ls\lf[2^jl(Q)\r]^{n(\frac{1}
{p}-\frac{1}{2})}\lf[\dint_0^{2^{j-3}l(Q)} t^{2n(\frac{1}{2}-\frac{1}{r})}
\exp\lf\{-\wz C\frac{\lf[2^jl(Q)\r]^{2k/(2k-1)}} {t^{2k/(2k-1)}}\r\}
\frac{dt}{t^{2k+1}}\r]^{1/2}\|b\|_{L^r(Q)},\\
&&\ls\lf[2^jl(Q)\r]^{n(\frac{1}{p}-\frac{1}{2})}\lf[\dint_0^{2^{j-3}l(Q)}
t^{2n(\frac{1}{2}-\frac{1}{r})}\lf[\frac{t}{2^jl(Q)}\r]^\bz\frac{dt}
{t^{2k+1}}\r]^{1/2}\lf\|\nabla^kb\r\|_{L^{\frac{rn}{n+kr}}(Q)}\\
&&\ls\lf[2^jl(Q)\r]^{n(\frac{1}{p}-\frac{1}{2})}\lf[\frac{1}
{\lf[2^jl(Q)\r]^\bz}\dint_0^{2^{j-3}l(Q)} t^{2n(\frac{1}{2}
-\frac{1}{r})+\bz-2k-1}\,dt\r]^{1/2}\lf\|\nabla^kb\r\|_{L^2(Q)}
|l(Q)|^{\frac{n+kr}{r}-\frac{n}{2}}\\
&&\ls2^{j[n(\frac{1}{p}-\frac{1}{r})-k]},
\end{eqnarray*}
which, together with \eqref{6.16}, shows that
$$\sum_{j=3}^\fz(\mathrm{V}_j)^p=\sum_{j=3}^\fz 2^{-\az jp}\ls1.$$
This, combined \eqref{6.13} and \eqref{6.15}, implies \eqref{6.12},
which completes the proof of Proposition \ref{p6.1}.
\end{proof}

Combining Theorem \ref{t6.1} and Proposition \ref{p6.1}, we obtain
the following Riesz transform characterization of $H_{L_1}^p(\rn)$.
We point out that Theorem \ref{t6.2} when $k=1$ is just the Riesz
transform characterization of $H_{-\rm{div}(A\nabla)}^p(\rn)$ for
$p\in(0,\,1]$, which is exactly \cite[Theorem 5.2]{hmm} in the case
that $p\in(0,\,1]$.

\begin{theorem}\label{t6.2}
Let $k\in\nn$, $L_1$ be the $2k$ order divergence form homogeneous
elliptic operator, $r\in(1,\,2]$ such that $rn/(n+kr)\le1$, and
the semigroup $\{e^{-tL_1}\}_{t>0}$ satisfy the $L^r-L^2$
$k$-off-diagonal estimates. Then for all $p\in(rn/(n+kr),\,1]$,
$H_{L_1}^p(\rn)=H_{\riesz,\,L_1}^p(\rn)$ with equivalent norms.
\end{theorem}

\begin{remark}\label{r6.1} We point out that a key fact used in the proof
of Proposition \ref{p6.1} (and hence Theorem \ref{t6.2}) is
$\|\sqrt{L_1}f\|_{L^2(\rn)}\ls\|\nabla^kf\|_{L^2(\rn)}$, which comes
from \cite[Theorem 1.1]{ahmt}. This inequality for $L_2$ is
equivalent to the following inequality that for all $f\in\dot
W^{k,2}(\rn)$,
\begin{eqnarray*}
\lf\|V^{k/2}f\r\|_{L^2(\rn)}\ls\lf\|\nabla^k f\r\|_{L^2(\rn)},
\end{eqnarray*}
which seems impossible even when $V:= 1$. Thus, the method used in
the proof of Proposition \ref{p6.1} seems unsuitable for obtaining a
counterpart of Proposition \ref{p6.1} for $L_2$.
\end{remark}

\Acknowledgements{This work was supported by the National Natural
Science Foundation (Grant No. 11171027) of China and Program for
Changjiang Scholars and Innovative Research Team in University of
China. The authors would like to thank the referees for their
careful reading and many valuable remarks which made this article
more readable.}


\end{document}